\documentclass[reqno,11pt,a4paper]{amsart}
\usepackage{amsmath}
\usepackage[foot]{amsaddr} 
\usepackage{enumerate}
\usepackage{amsfonts}
\usepackage{amssymb}
\usepackage{mathrsfs}
\usepackage{euscript}
\usepackage{empheq}
\usepackage{mathtools}
\usepackage[usenames,dvipsnames]{pstricks}
\usepackage{epsfig}
\usepackage{hyperref}
\usepackage{textcomp}
\usepackage{subcaption}
\usepackage{float}
\usepackage[left= 1in, right= 1in, top= 1.2in, bottom= 1.0in]{geometry}
\usepackage[percent]{overpic}
\usepackage[numbers,sort&compress]{natbib}
\usepackage[all]{xy}

\title[On the spectral problem for the time-periodic NLS]{On the spectral problem associated with the time-periodic nonlinear Schr\"odinger equation}

\author{Jonatan Lenells and Ronald Quirchmayr}

\address{Department of Mathematics, KTH Royal Institute of Technology,  
Stockholm, Sweden} 

\email{jlenells@kth.se, ronaldq@kth.se}

\newtheorem{theorem}{Theorem}[section]

\newtheorem*{theorem*}{Theorem}
\newtheorem*{lemma*}{Lemma}
\newtheorem{proposition}[theorem]{Proposition}
\newtheorem{lemma}[theorem]{Lemma}
\newtheorem{corollary}[theorem]{Corollary}
\theoremstyle{definition}
\newtheorem{definition}[theorem]{Definition}
\newtheorem{remark}[theorem]{Remark}
\newtheorem*{remark*}{Remark}

\numberwithin{equation}{section}

\newcommand{\be}{\begin{equation}}
\newcommand{\ee}{\end{equation}}
\newcommand{\bew}{\begin{equation*}} 
\newcommand{\eew}{\end{equation*}}

\newcommand{\D}{\partial}
\newcommand{\Dt}{\partial_t}

\newcommand{\I}{\mathrm{i}}
\newcommand{\C}{\mathbb{C}}
\newcommand{\R}{\mathbb{R}}
\newcommand{\Z}{\mathbb{Z}}
\newcommand{\N}{\mathbb{N}}
\newcommand{\K}{\mathbb{K}}
\newcommand{\T}{\mathbb{T}}
\newcommand{\LL}{\mathrm{X}}
\newcommand{\e}{\mathrm{e}}
\newcommand{\m}{\grave{m}}
\newcommand{\M}{\grave{M}}
\newcommand{\eps}{\varepsilon}
\newcommand{\id}{\mathrm{I}}
\DeclareMathOperator{\pr}{pr}

\DeclareMathOperator{\sgn}{sgn}

\newcommand{\overbar}[1]{\mkern 1.5mu\overline{\mkern-1.5mu#1\mkern-1.5mu}\mkern 1.5mu}

\makeatletter
\def\@setemails{%
\mbox{{\itshape E-mail addresses}:\space}{\ttfamily\emails}.
}
\makeatother

\begin{document}
\begin{abstract} 
\noindent
According to its Lax pair formulation, the nonlinear Schr\"odinger (NLS) equation can be expressed as the compatibility condition of two linear ordinary differential equations with an analytic dependence on a complex parameter. The first of these equations---often referred to as the \emph{$x$-part} of the Lax pair---can be rewritten as an eigenvalue problem for a Zakharov-Shabat operator. The spectral analysis of this operator is crucial for the solution of the initial value problem for the NLS equation via inverse scattering techniques. For space-periodic solutions, this leads to the existence of a Birkhoff normal form, which beautifully exhibits the structure of NLS as an infinite-dimensional completely integrable system.
In this paper, we take the crucial steps towards developing an analogous picture for time-periodic solutions by performing a spectral analysis of the \emph{$t$-part} of the Lax pair with a periodic potential.
\end{abstract}

\maketitle



\thispagestyle{empty}

\section{Introduction}
\vspace{0em}
\noindent
The nonlinear Schr\"odinger (NLS) equation 
\begin{align}\label{nls}
  \I u_t + u_{xx} - 2\sigma |u|^2 u = 0, \qquad \sigma = \pm 1,
\end{align}
is one of the most well-studied nonlinear partial differential equations. As a universal model equation for the evolution of weakly dispersive wave packets, it arises in a vast number of applications, ranging from nonlinear fiber optics and water waves to Bose-Einstein condensates. 
Many aspects of the mathematical theory for \eqref{nls} are well-understood. 
For example, for spatially periodic solutions (i.e., $u(x,t) = u(x+1,t)$), there exists a normal form theory for \eqref{nls} which beautifully exhibits its structure as an infinite-dimensional completely integrable system (see \cite{GrebertKappeler14} and references therein).  
This theory takes a particularly simple form in the case of the defocusing (i.e., $\sigma = 1$) version of \eqref{nls}. Indeed, for $\sigma = 1$, the normal form theory ascertains the existence of a single global system of Birkhoff coordinates (the Cartesian version of action-angle coordinates) for \eqref{nls}. For the focusing (i.e., $\sigma = -1$) NLS, such coordinates also exist, but only locally \cite{Kappeler_etal_09}. The existence of Birkhoff coordinates has many implications. Among other things, it provides an explicit decomposition of phase space into invariant tori, thereby making it evident that an $x$-periodic solution of the defocusing NLS is either periodic, quasi-periodic, or almost periodic in time. The construction of Birkhoff coordinates for \eqref{nls} is a major achievement which builds on ideas going back all the way to classic work of Gardner, Greene, Kruskal and Miura on the Korteweg--de~Vries (KdV) equation \cite{GGKM1967, GGKM1974}, and of Zakharov and Shabat on the NLS equation \cite{ZS1972}. Early works on the (formal) introduction of action-angle variables include \cite{ZS1971, ZM1974}. More recently, Kappeler and collaborators have developed powerful methods which have led to a rigorous construction of Birkhoff coordinates for both KdV \cite{KM1999, KM2001, KappelerPoeschel03} and NLS \cite{Kappeler_etal_09, GrebertKappeler14} in the spatially periodic case.

The key element in the construction of Birkhoff coordinates is the spectral analysis of the Zakharov-Shabat operator $L(u)$ defined by
\bew
L(u) = \I \sigma_3\bigg(\frac{\mathrm d}{\mathrm d x} - U\bigg), \quad \text{where} \quad
U = \begin{pmatrix} 0 & u \\
\sigma \bar{u} & 0 \end{pmatrix} \quad \text{and}\quad 
\sigma_3 = \begin{pmatrix} 1 & 0 \\ 0 & -1 \end{pmatrix}.
\eew
In particular, the periodic eigenvalues of this operator are independent of time if $u$ evolves according to \eqref{nls} and thus encode the infinite number of conservation laws for \eqref{nls}.
The time-independence is a consequence of the fact that equation \eqref{nls} can be viewed as the compatibility condition $\phi_{xt} = \phi_{tx}$ of the Lax pair equations \cite{L1968, ZS1972}
\begin{align}\label{xlax}
 & \phi_x + \I \lambda \sigma_3 \phi = U\phi,
  	\\ \label{tlax}
 & \phi_t + 2 \I \lambda^2 \sigma_3 \phi = V \phi,	
\end{align}
where $\lambda \in \C$ is the spectral parameter, $\phi(x,t,\lambda)$ is an eigenfunction,
\be \label{V(u)def} 
V = \begin{pmatrix} - \I \sigma |u|^2 & 2 \lambda u + \I u_x \\
2 \sigma \lambda \bar{u} - \I \sigma \bar{u}_x & \I \sigma |u|^2 \end{pmatrix},
\ee
and we note that \eqref{xlax} is equivalent to the eigenvalue problem $L(u)\phi = \lambda \phi$.

Strangely enough, although the spectral theory of equation \eqref{xlax} (or, equivalently, of the Zakharov-Shabat operator) has been so thoroughly studied, it appears that no systematic study of the spectral theory of the $t$-part \eqref{tlax} with a periodic potential has yet been carried out (there only exist a few studies of the NLS equation on the half-line with asymptotically time-periodic boundary conditions which touch tangentially on the issue \cite{BIK2007, BIK2009, Ldefocusing, LeFo15, LFunifiedII}).  The general scope of this paper is to lay the foundation for a larger project with the goal of showing that \eqref{nls}, viewed as an evolution equation in the $x$-variable, is an integrable PDE and in particular admits a normal form in a neighborhood of the trivial solution $u\equiv0$. This means that one can construct Birkhoff coordinates---often referred to as nonlinear Fourier coefficients---on appropriate function spaces so that, when expressed in these coordinates, the PDE can be solved by quadrature. Our approach is inspired by the methods and ideas of~\cite{Kappeler_etal_09, GrebertKappeler14}, where such coordinates for \eqref{nls} as a $t$-evolution equation were constructed on the phase space of $x$-periodic functions. This work at hand provides the key ingredients needed to adapt the scheme of construction developed in~\cite{Kappeler_etal_09, GrebertKappeler14} to the $x$-evolution of time-periodic solutions of NLS, and ultimately to establish local Birkhoff coordinates, hence integrability. In particular, we provide asymptotic estimates for the fundamental matrix solution of the $t$-part \eqref{tlax}, which we exploit to study the periodic spectrum of the corresponding generalized eigenvalue problem.

For the spectral analysis, it is appropriate (at least initially) to treat the four functions $u$, $\sigma\bar{u}$, $u_x$, $\sigma \bar{u}_x$ in the definition of $V$ as independent. We will therefore consider the spectral problem \eqref{tlax} with potential $V$ given by
\begin{align}\label{Vdef}
V = V(\lambda,\psi) = \begin{pmatrix} - \I  \psi^1 \psi^2 & 2\lambda \psi^1 + \I \psi^3 \\
2\lambda \psi^2 - \I   \psi^4 & \I  \psi^1 \psi^2 \end{pmatrix},
\end{align}
where $\psi=\{\psi^j(t)\}_1^4$ are periodic functions of $t \in \R$ with period one. 

Apart from the purely spectral theoretic interest of studying \eqref{tlax}, there are at least three other reasons motivating the present study:

\begin{enumerate}[$-$]
\item
First, in the context of fiber optics, the roles of the variables $x$ and $t$ in equation \eqref{nls} are interchanged, see e.g. \cite{A2013}. In other words, in applications to fiber optics, $x$ is the temporal and $t$ is the spatial variable. Since the analysis of \eqref{tlax} plays the same role for the $x$-evolution of $u(x,t)$ as the analysis of the Zakharov-Shabat operator plays for the $t$-evolution, this motivates the study of \eqref{tlax}.  
\item
Second, one of the most important problems for nonlinear integrable PDEs is to determine the solution of initial-boundary value problems with asymptotically time-periodic boundary data \cite{BF2008, BKSZ2010, LFunifiedII}. For example, consider the problem of determining the solution $u(x,t)$ of \eqref{nls} in the quarter-plane $\{x>0,t>0\}$, assuming that the initial data $u(x,0)$, $x \geq 0$, and the boundary data $u(0,t)$, $t \geq 0$ are known, and that $u(0,t)$ approaches a periodic function as $t \to \infty$. The analysis of this problem via Riemann-Hilbert techniques relies on the spectral analysis of \eqref{tlax} with a periodic potential determined by the asymptotic behavior of $u(0,t)$ \cite{BIK2007, LeFo15}. 
\item
Third, at first sight, the differential equations \eqref{xlax} and \eqref{tlax} may appear unrelated. However, the fact that they are connected via equation \eqref{nls} implies that they can be viewed as different manifestations of the same underlying mathematical structure. Indeed, for the analysis of elliptic equations and boundary value problems, a coordinate-free intrinsic approach in which the two parts of the Lax pair are combined into a single differential form has proved the most fruitful \cite{F1997, HE1989b}. In such a formulation, eigenfunctions which solve both the $x$-part \eqref{xlax} and the $t$-part \eqref{tlax} simultaneously play a central role. It is therefore natural to investigate how the spectral properties of \eqref{xlax} are related to those of \eqref{tlax}. Since the NLS equation is just one example of a large number of integrable equations with a Lax pair formulation, the present work can in this regard be viewed as a case study with potentially broader applications.
\end{enumerate}

\subsection{Comparison with the analysis of the $x$-part}
Compared with the analysis of the $x$-part \eqref{xlax}, the spectral analysis of the $t$-part \eqref{tlax} presents a number of novelties. Some of the differences are:
\begin{enumerate}[$-$]
\item
Whereas equation \eqref{xlax} can be rewritten as the eigenvalue equation $L(u)\phi = \lambda \phi$ for an operator $L(u)$, no (natural) such formulation is available for \eqref{tlax} due to the more complicated $\lambda$-dependence. Nevertheless, it is possible to define spectral quantities associated with \eqref{tlax} in a natural way. 
\item
Asymptotically for large $|\lambda|$, the periodic and antiperiodic eigenvalues of \eqref{xlax} come in pairs which lie in discs centered at the points $n \pi$, $n \in \Z$, along the real axis \cite{GrebertKappeler14}. 
In the case of \eqref{tlax}, a similar result holds, but in addition to discs centered at points on the real axis, there are also discs centered at points on the imaginary axis (see Lemma \ref{counting_lemma}). Moreover, the spacing between these discs shrinks to zero as $|\lambda|$ becomes large.  
\item
For so-called real type potentials (the defocusing case), the Zakharov-Shabat operator is self-adjoint, implying that the spectrum associated with \eqref{xlax} is real. No such statement is true for the $t$-part \eqref{tlax}. This is clear already from the previous statement that there exist pairs of eigenvalues tending to infinity contained in discs centered on the imaginary axis. However, it is also true that the eigenvalues of \eqref{tlax} near the real axis need not be purely real and the eigenvalues near the imaginary axis need not be purely imaginary. This can be seen from the simple case of a single-exponential potential. Indeed, consider the potential
\be \label{intro_singexp}
\big(\psi^1(t), \psi^2(t),\psi^3(t),\psi^4(t)\big)  = 
\big(\alpha e^{\I\omega t}, \sigma \bar{\alpha} e^{-\I\omega t}, c e^{\I\omega t}, \sigma \bar{c} e^{-\I\omega t}\big),
\ee
where $\alpha, c \in \C$, $\omega \in 2\pi \Z$, and $\sigma = \pm 1$. For potentials of this form, equation \eqref{tlax} can be solved explicitly (see Section \ref{sec_singexp}) and Fig.~\ref{fig:intro} shows the periodic and antiperiodic eigenvalues of \eqref{tlax} for two choices of the parameters. 
\item
Whereas the matrix $U$ in \eqref{xlax} is off-diagonal and contains only the function $u$ and its complex conjugate $\bar{u}$, the matrix $V$ in \eqref{tlax} is neither diagonal nor off-diagonal and involves also $u_x$ and $\bar{u}_x$. This has implications for the spectral analysis---an obvious one being that \eqref{Vdef} involves four instead of two scalar potentials $\psi^j(t)$.
\item
The occurrence of the factor $\lambda^2$ in \eqref{tlax} implies that the derivation of the fundamental solution's asymptotics for $|\lambda|\to \infty$ requires new techniques (see the proof of Theorem \ref{thm_asymptotics_M_1}). For the $x$-part, the analogous result can be established via an application of Gronwall's lemma \cite{GrebertKappeler14}. This approach does not seem to generalize to the $t$-part, but instead we are able to perform an asymptotic analysis inspired by \cite[Chapter 6]{CL1955} (see also \cite{LNonlinearFourier}).
\item
In Theorems \ref{thm_conn_per_ev} and \ref{cor_conn_per_ev}, we will, for sufficiently small potentials, establish the existence of analytic arcs which connect periodic eigenvalues close to the real line in a pairwise manner and along which the discriminant is real. A similar result for \eqref{xlax} can be found in \cite[Proposition 2.6]{Kappeler_etal_09}. In both cases, the proof relies on the implicit function theorem in infinite dimensional Banach spaces. However, the proof of \eqref{tlax} is quite a bit more involved and requires, for example, the introduction of more complicated function spaces, see \eqref{lpsdef}.
\end{enumerate}

\begin{figure}[h!] \centering
\begin{subfigure}{.43\textwidth}  \centering
\begin{overpic}[width=.95\textwidth]{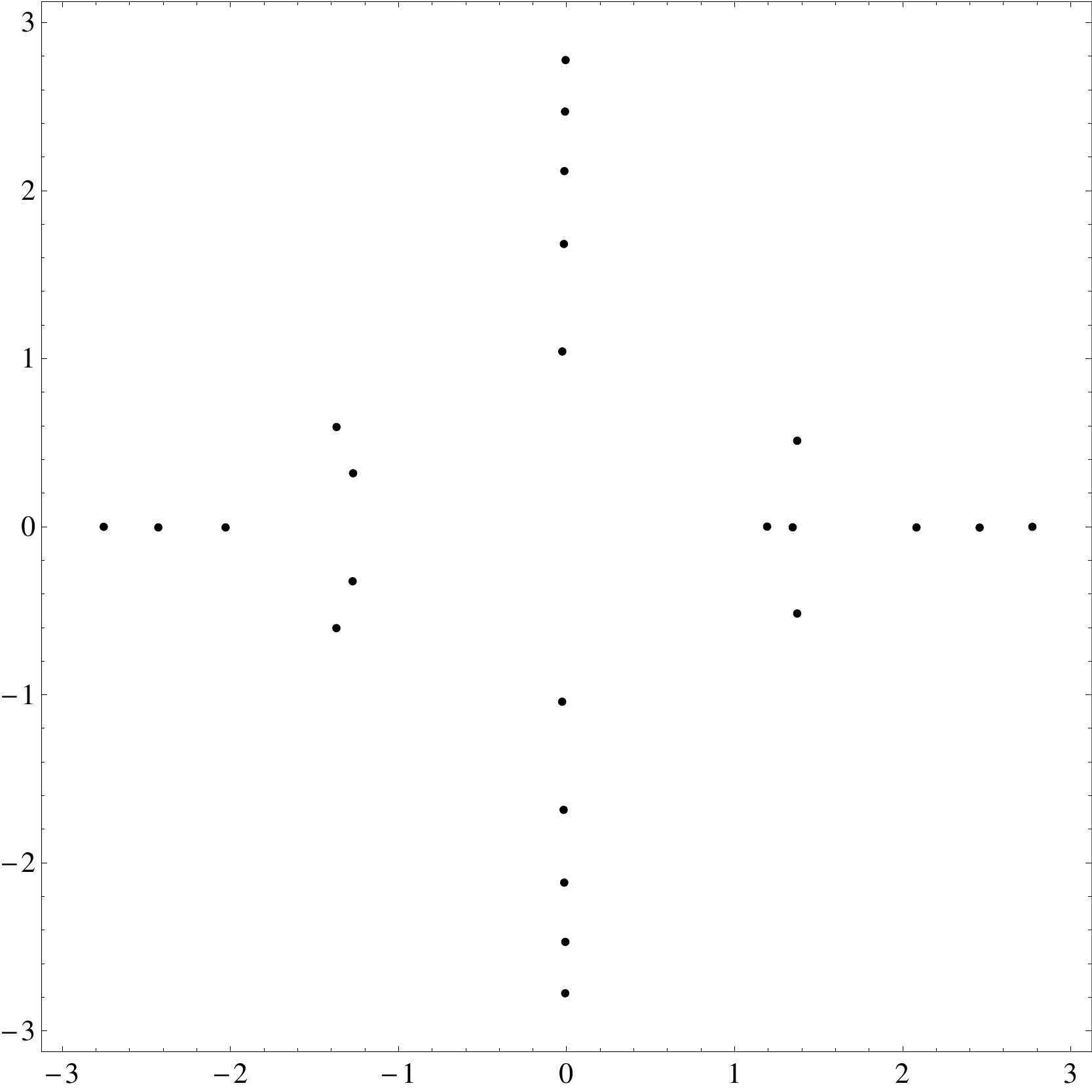}
\end{overpic}
\caption{}
  \label{fig:intro1a}
\end{subfigure}%
\quad
\begin{subfigure}{.43\textwidth}  \centering
\begin{overpic}[width=.95\textwidth]{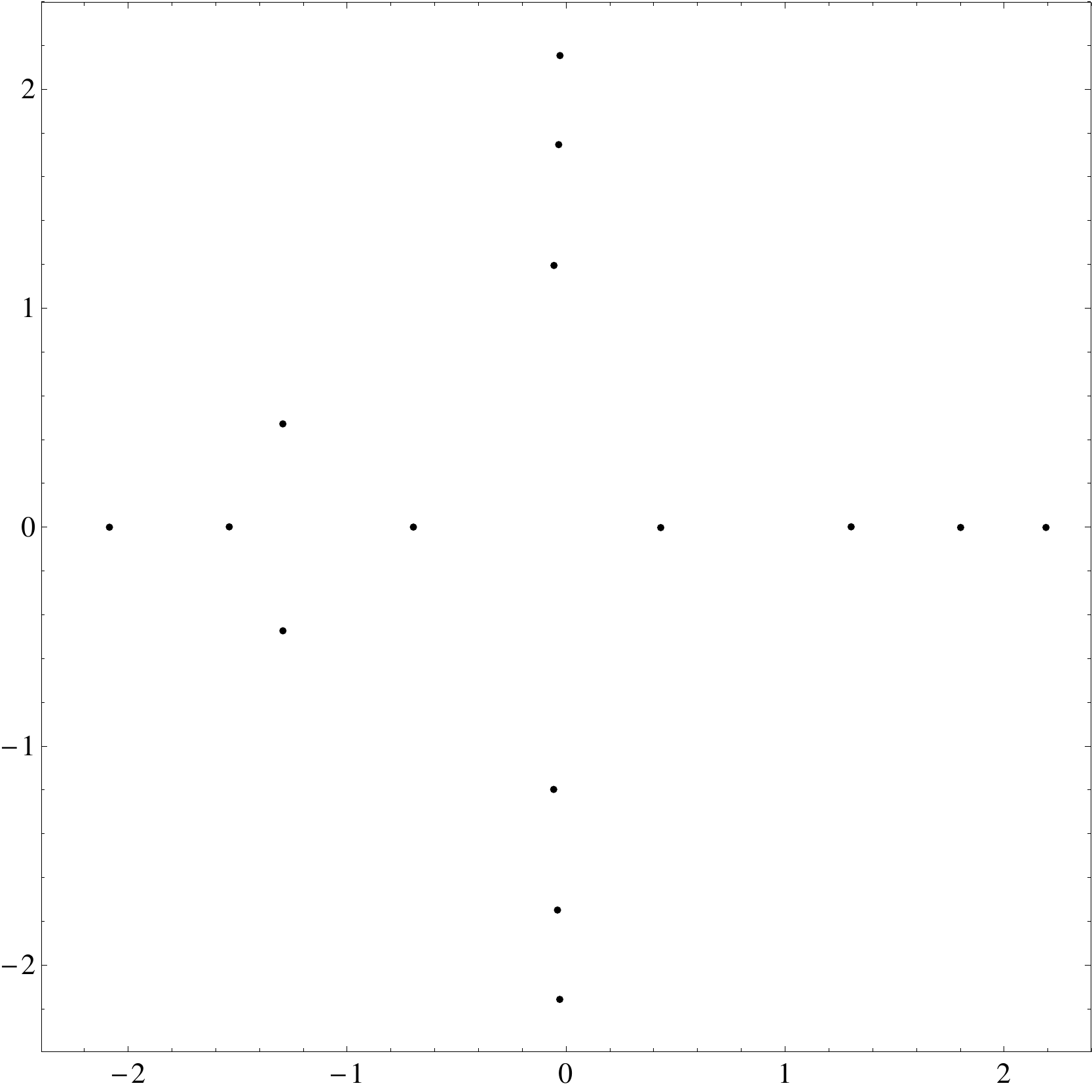}
\end{overpic}
\caption{}
  \label{fig:intro1b}
\end{subfigure}%
\caption{Plots of the periodic and antiperiodic eigenvalues for two single exponential potentials with different sets of parameters $\sigma$, $\omega$, $\alpha$ and $c$; cf.~\eqref{intro_singexp}.
Fig.~\ref{fig:intro1a} shows the periodic and antiperiodic eigenvalues for the real type potential given by $\sigma=1$, $\omega=-2\pi$, $\alpha=\frac{6}{15}+\frac{11}{4}\I$, $c=\frac{1}{10}$; Fig.~\ref{fig:intro1b} shows the spectrum of the imaginary type potential with $\sigma=-1$, $\omega=-2\pi$, $\alpha=\frac{1}{2}$, $c=\I \alpha \sqrt{ 2\alpha^2-\omega}$, which arises from an exact plane wave solution of the focusing NLS.}
\label{fig:intro}
\end{figure} 

\subsection{Outline of the paper}
In order to facilitate comparison with the existing literature on the $x$-part \eqref{xlax}, our original intention was to closely   follow the scheme and methods developed in \cite{GrebertKappeler14}, adapting them to equation \eqref{tlax}. 
As pointed out in the previous paragraph we have that equation \eqref{tlax} is quadratic in the spectral parameter $\lambda$ and hence it is a \emph{generalized eigenvalue problem} making its treatment more challenging.
Nevertheless, some resemblance to the first two chapters of \cite{GrebertKappeler14} remains.
The main novelty of the paper is the proof of the leading order asymptotics for large $|\lambda|$ of the fundamental matrix solution associated with  \eqref{tlax}, cf.~Theorem \ref{thm_asymptotics_M_1}.
These asymptotics are a key ingredient for the subsequent two sections. 
The discussion of the asymptotic localization of the Dirichlet eigenvalues, Neumann eigenvalues and periodic eigenvalues
in Section \ref{sec_spectra}, as well as the study of the zero set of the imaginary part of the discriminant for potentials of real and imaginary type (corresponding to the defocusing and focusing NLS, respectively) in Section \ref{sec_potRI} then follow closely~\cite{GrebertKappeler14} and, respectively,~\cite{Kappeler_etal_09}.
In Section \ref{sec_singexp}, we consider the special (but important) case of single-exponential potentials for which the fundamental matrix solution permits an exact formula. This enables us to illustrate the theoretical results from the previous sections.
We provide useful formulas for the gradients of the fundamental solution and the discriminant in Section \ref{sec_formgrad}.
The last section reviews the standard bi-Hamiltonian structure of NLS as a time-evolution equation and establishes a Hamiltonian structure for NLS viewed as an $x$-evolution equation. More precisely, we show that the NLS system 
\bew
\begin{cases}
   q_{xx}  = -\I q_t +2 q^2 r
  	\\
   r_{xx}  = \I r_t + 2 r^2 q,
\end{cases}
\eew
which is associated with \eqref{nls}, can be written as
\bew
(q,r,p,s)^\intercal_x= \tilde{\mathcal{D}} \, \partial \tilde H_1,
\eew
where the $4$-vector on the left hand side is understood as a column vector (indicated by the transpose operation ${}^\intercal$), and
the Hamiltonian $\tilde H_1$, its gradient $\partial \tilde H_1$, and the Hamiltonian operator $\tilde{\mathcal D}$ are given by 
\bew
\tilde{H}_1 = \int \big(ps + \I q_tr - q^2 r^2\big) \, \mathrm d t, \quad
\partial \tilde{H}_1 = 
\begin{pmatrix}
-\I r_t - 2 r^2 q\\
\I q_t - 2 q^2 r\\
s\\
p
\end{pmatrix}, \quad 
\tilde{\mathcal{D}} = 
\begin{pmatrix}  0 & 0 & 0 & 1  \\
 0 & 0 &1&0  \\
 0 &-1&0& 0 \\
  -1 &0& 0 &0
 \end{pmatrix}.
\eew
The associated Poisson bracket for two functionals $F$ and $G$ is given by 
\bew
\{F, G\}_{\tilde{\mathcal{D}}} = \int  (\partial F)^\intercal \, \tilde{\mathcal{D}} \, \partial G \, \mathrm d t.
\eew

\vspace{0em}

\section{Fundamental solution}\label{sec_fund_sol} 
\vspace{0em}
\noindent
In Section \ref{subsec_baprop}, we introduce the framework for the study of \eqref{tlax} and establish basic properties of the fundamental solution. In Section \ref{subsec_asympM} we derive estimates for the fundamental matrix solution and its $\lambda$-derivative for large $|\lambda|$. These estimates will be used  in Section \ref{sec_spectra}  to asymptotically localize the Dirichlet, Neumann and periodic eigenvalues as well as the critical points of the discriminant of \eqref{tlax}.

\subsection{Framework and basic properties} \label{subsec_baprop}
The potential matrix $V$ in \eqref{tlax} depends on the \emph{spectral parameter} $\lambda \in \C$ and the potential
$\psi=(\psi^1,\psi^2,\psi^3,\psi^4)$ taken from the space 
\bew
\LL \coloneqq H^1(\T,\C) \times H^1(\T,\C) \times H^1(\T,\C) \times H^1(\T,\C),
\eew
where $H^1(\T,\C)$ denotes the Sobolev space of complex absolutely continuous functions on the one-dimensional torus $\T=\R/\Z$ with square-integrable weak derivative, which is equipped with the usual norm induced by the $H^1$-inner product 
\bew
(\cdot,\cdot)\colon H^1(\T,\C)\times H^1(\T,\C) \to \C, \quad (u,v) \mapsto \int^1_0 ( u \bar v + u_t \bar v_t )  \, \mathrm d t.
\eew
We endow the space $\LL$ with the inner product
\bew 
\langle \psi_1,\psi_2 \rangle \coloneqq
 (\psi^1_1,  \psi^1_2) + (\psi^2_1 , \psi^2_2) + (\psi^3_1 , \psi^3_2) + (\psi^4_1 , \psi^4_2),
\eew
which induces the norm $\|\psi\| = \sqrt{\langle \psi,\psi \rangle}$ 
on $\LL$.
Likewise we consider the space
\bew
\LL_\tau \coloneqq  H^1([0,\tau],\C) \times H^1([0,\tau],\C) \times H^1([0,\tau],\C) \times H^1([0,\tau],\C)
\eew
on the interval $[0,\tau]$ for fixed $\tau>0$, where the Sobolev space $H^1([0,\tau],\C)$ is equipped with the inner product
\bew
(\cdot,\cdot) _\tau\colon H^1([0,\tau],\C)\times H^1([0,\tau],\C) \to \C, \quad 
(u,v) \mapsto \int^\tau_0 ( u \bar v + u_t \bar v_t )  \, \mathrm d t.
\eew
We set
\bew 
\langle \psi_1,\psi_2 \rangle_\tau \coloneqq
 (\psi^1_1,  \psi^1_2)_\tau + (\psi^2_1 , \psi^2_2)_\tau + (\psi^3_1 , \psi^3_2)_\tau + (\psi^4_1 , \psi^4_2)_\tau,
\eew
which makes $\LL_\tau$ an inner product space and induces the norm $\|\psi\|_\tau \coloneqq \sqrt{\langle \psi,\psi \rangle_\tau}$. 
For the components $\psi^j$ of $\psi\in \LL$ or $\psi\in \LL_\tau$ respectively, we write
\bew
\|\psi^j\| = \sqrt{(\psi^j,\psi^j)}, \quad
\|\psi^j\|_\tau = \sqrt{(\psi^j,\psi^j)_\tau}, \qquad j=1,2,3,4.
\eew
Since not every $\psi\in\LL_1$ is periodic, $\LL$ is a proper closed subspace of $\LL_1$. The spaces $\LL$ and $\LL_\tau$ inherit completeness from $H^1(\T,\C)$ and $H^1([0,\tau],\C)$ respectively, hence they are Hilbert spaces.

On the space $M_{2\times2}(\C)$ of complex valued $2\times2$-matrices we consider the norm $| \cdot |$,
which is induced by the standard norm in $\C^2$, also denoted by $|\cdot|$, i.e.
\bew
|A| 
 \coloneqq \max_{z\in \C^2,|z|=1} |A z|.
\eew
The norm $|\cdot|$ is submultiplicative, i.e.~$|A B| \leq |A| \, |B|$ for $A,B \in M_{2\times2}(\C)$. 

\smallskip

For given $\lambda\in\C$ and $\psi\in\LL$, let us write the initial value problem corresponding to \eqref{tlax} as
\begin{align}
\mathrm D \phi &= R \phi + V \phi, \label{main_eq} \\
 \phi(0) &= \phi_0, \label{main_eq_id}
\end{align}
where $V$ is given by \eqref{Vdef},
\begin{align*}
\mathrm D&\coloneqq
  \begin{pmatrix}
   \Dt &  \\
   &\Dt \\
  \end{pmatrix}  ,
\quad
R \equiv R(\lambda) \coloneqq -2\I\lambda^2 \sigma_3,
\end{align*}
and
\bew
\phi=
\begin{pmatrix}
\phi^1 \\ \phi^2
\end{pmatrix} \colon \T\to\C^2.
\eew
Equation \eqref{main_eq} reduces to \eqref{tlax} if we identify $(\psi^1, \psi^2, \psi^3, \psi^4)  = (u,\sigma \bar{u}, u_x, \sigma \bar{u}_x)$.

In analogy to the conventions for the eigenvalue problem \eqref{xlax} for the $x$-part of the NLS Lax pair, we say that the spectral problem \eqref{main_eq} is of Zakharov-Shabat (ZS) type. The corresponding equation written in AKNS \cite{AKNS1974} coordinates $(q_0,p_0,q_1,p_1)$ reads 
\be \label{AKNS_system}
\mathrm D \phi =
-2\lambda^2 
  \begin{pmatrix}
    &-1  \\
   1& \\
  \end{pmatrix} \phi  +
  \begin{pmatrix}
   2\lambda q_0 - p_1 & 2\lambda p_0 + p_0^2 +q_0^2  + q_1  \\
  2\lambda p_0 -( p_0^2 +q_0^2)  + q_1 & -2\lambda q_0 + p_1 \\
  \end{pmatrix} \phi. 
\ee
It is obtained by multiplying the operator equation $\mathrm D = R + V$ from the right with $T$ and from the left with $T^{-1}$, where
\be \label{def_T}
T = 
  \begin{pmatrix}
   1 &\I  \\
   1&-\I \\
  \end{pmatrix} , \quad 
  T^{-1} = \frac{1}{2} 
  \begin{pmatrix}
   1 &1  \\
   -\I&\I \\
  \end{pmatrix},
\ee
and by writing 
\bew
\psi^1=q_0+\I p_0, \;
\psi^2=q_0-\I p_0, \;
\psi^3=q_1+\I p_1, \;
\psi^4=q_1-\I p_1, \;
\eew
that is,
\bew
q_0=\frac{1}{2}(\psi^1+\psi^2), \; p_0=-\frac{\I}{2}(\psi^1-\psi^2), \; q_1=\frac{1}{2}(\psi^3+\psi^4), \; 
p_1=-\frac{\I}{2}(\psi^3-\psi^4).
\eew

In what follows we show the existence of a unique matrix-valued fundamental solution $M$ of \eqref{main_eq}, that is, a solution of
\be \label{fund_solution}
\mathrm D M = R M + V M, \quad M(0)=\id,
\ee
where $\id \in M_{2\times2}(\C)$ denotes the identity matrix.
The proof relies on a standard iteration technique.
We first observe that the fundamental matrix solution for the zero potential $\psi=0$ is given by
\bew
E_\lambda(t) \coloneqq   \e^{-2 \lambda^2 \I \sigma_3 t} =
  \begin{pmatrix}
    \e^{-2\lambda^2\I t}&  \\
   & \e^{2\lambda^2\I t} \\
  \end{pmatrix},
  \quad t\geq 0.
\eew
Indeed, $E_\lambda$ solves the initial value problem
\bew
\mathrm D E_\lambda = R E_\lambda, \quad E_\lambda(0) = \id.
\eew
For $\lambda\in \C$, $\psi \in \LL$ and $0\leq t <\infty$ we inductively define 
\be \label{M_n}
M_0\coloneqq E_\lambda(t), \quad
M_{n+1}(t) \coloneqq \int^t_0 E_\lambda(t-s) V(s) M_n(s) \, \mathrm d s, \qquad n \geq 0,
\ee
where  $V\equiv V(s,\lambda,\psi)$ is defined for all $s\geq0$ by periodicity. 
For each $n\geq1$, $M_n$ is continuous on $[0,\infty)\times\C\times\LL$ and satisfies
\bew
M_n(t) =\int_{0\leq s_n \leq \cdots \leq s_1\leq t} E_\lambda(t) \prod^n_{i=1} E_\lambda(-s_i) V(s_i) E_\lambda(s_i) \, \mathrm d s_n \cdots \mathrm d s_1.
\eew
Using that $| E_\lambda(t)|=\e^{2 | \Im (\lambda^2)| t}$ for $t\geq0$, we estimate
\begin{align*}
|M_n(t)| &\leq \e^{2(2n+1) | \Im (\lambda^2)| t}
\int_{0\leq s_n \leq \cdots \leq s_1\leq t} 
\prod^n_{i=1} | V(s_i) | \, \mathrm d s_n \cdots \mathrm d s_1 \\
&\leq \frac{\e^{2(2n+1) | \Im (\lambda^2)| t}}{n!}
\int_{[0,t]^n} 
\prod^n_{i=1} | V(s_i) | \, \mathrm d s_n \cdots \mathrm d s_1 \\
&\leq \frac{\e^{2(2n+1) | \Im (\lambda^2)| t}}{n!}
\bigg(\int^t_0    
 | V(s) | \, \mathrm d s \bigg)^n\\
&\leq \frac{\e^{2(2n+1) | \Im (\lambda^2)| t}}{n!} \, t^{n/2} \,  \big(2 \max(1,|\lambda|)\big)^n \, [C(\psi,t)]^n,
\end{align*}
where one can choose
\bew
C(\psi,t) \coloneqq     
\big\| \max \big(|\psi^1 \psi^2|, |\psi^1| + |\psi^3|, |\psi^2| + |\psi^4|\big) \big\|_t
\eew
as a uniform bound for bounded sets of $[0,\infty)\times \LL$. Therefore the matrix
\be \label{M}
M(t) \coloneqq \sum^\infty_{n=0} M_n(t)
\ee
exists and converges uniformly on bounded subsets of $[0,\infty)\times\C\times\LL$. 
By construction, $M$ solves the integral equation
\be \label{M_int_eq}
M(t,\lambda,\psi)=E_\lambda(t) + \int^t_0 E_\lambda(t-s) V(s,\lambda,\psi) M(s,\lambda,\psi) \, \mathrm d s,
\ee
hence $M$ is the unique matrix solution of the initial value problem \eqref{fund_solution}.
Since each $M_n$, $n\geq0$ is continuous on $[0,\infty)\times\C\times\LL$ and moreover analytic in $\lambda$ and $\psi$ for fixed $t\in[0,\infty)$, $M$ inherits the same regularity due to uniform convergence. Thus we have proved the following result.
\begin{theorem}[Existence of the fundamental solution $M$]\label{thm_fund_sol}
The power series \eqref{M} with coefficients given by \eqref{M_n} converges uniformly on bounded subsets of 
$[0,\infty)\times \C\times \LL$ to a continuous function denoted by $M$, which is analytic in $\lambda$ and $\psi$ for each fixed $t\geq 0$ and satisfies 
the integral equation \eqref{M_int_eq}.
\end{theorem}

The fundamental solution $M$ is in fact compact:

\begin{proposition}[Compactness of $M$] \label{Prop_M_comp}
For any sequence $(\psi_k)_k$ in $\LL$ which converges weakly to an element $\psi \in \LL$ as $k\to\infty$, i.e.~$\psi_k \rightharpoonup\psi$, one has 
\bew
|  M(t,\lambda,\psi_k) - M(t,\lambda,\psi)  | \to 0
\eew
uniformly on bounded sets of $[0,\infty)\times\C$.
\end{proposition}
\begin{proof}
It suffices to prove the statement for each $M_n$, since the series \eqref{M} converges uniformly on bounded subsets of $[0,\infty)\times \C\times \LL$. 
The assertion is true for $M_0=E_\lambda$, which is independent of $\psi$. To achieve the inductive step, we assume that the statement holds for $M_n$, $n\geq1$, and consider an arbitrary sequence $\psi_k\rightharpoonup\psi$ in $\LL$. Then 
\bew
M_n(t,\lambda,\psi_k)\to M_n(t,\lambda,\psi)
\eew
uniformly on bounded subsets of $[0,\infty)\times\C$. Thus
\begin{align*}
M_{n+1}(t,\lambda,\psi_k) &= \int^t_0 E_\lambda(t-s) V(s,\lambda,\psi_k) M_n(s,\lambda,\psi_k) \, \mathrm d s \\
&\to \int^t_0 E_\lambda(t-s) V(s,\lambda,\psi) M_n(s,\lambda,\psi) \, \mathrm d s
\end{align*}
uniformly on bounded subsets of $[0,\infty)\times\C$.
\end{proof}

Furthermore, $M$ satisfies the \emph{Wronskian identity}:

\begin{proposition}[Wronskian identity] \label{Wronskian_identity}
Everywhere on $[0,\infty)\times \C\times \LL$ it holds that
\bew 
\mathrm{det} \, M(t, \lambda, \psi) = 1.
\eew
In particular, the inverse $M^{-1}$ is given by
\bew
M^{-1} = \begin{pmatrix}m_4 & -m_2 \\ -m_3& m_1\end{pmatrix} \quad \text{if} \quad M = 
  \begin{pmatrix}
   m_1 & m_2  \\
   m_3&m_4 \\
  \end{pmatrix}.
\eew
\end{proposition}
\begin{proof}
The fundamental solution $M$ is regular for all $t\geq0$. Therefore a direct computation yields
\bew
\Dt\, \mathrm{det}\,M = \mathrm{tr} (\Dt  M \cdot M^{-1}) \, \mathrm{det}\, M.
\eew
Since 
\bew
\mathrm{tr} (\Dt  M \cdot M^{-1})= \mathrm{tr} (R+V) = 0
\eew
it follows that $\mathrm{det} \, M(t)=\mathrm{det} \, M(0)=1$ for all $t\geq0$.
\end{proof}

The solution of the inhomogeneous problem corresponding to the initial value problem \eqref{main_eq}-\eqref{main_eq_id} has the usual ``variation of constants representation'':

\begin{proposition} \label{Prop_inhom_eq}
The unique solution of the inhomogeneous equation
\bew
\mathrm D f = (R+V) f +g, \quad f(0)=v_0
\eew 
with $g\in L^2([0,1],\C)\times L^2([0,1],\C)$ is given by
\be \label{prop_inhom}
f(t) = M(t) \bigg( v_0 + \int^t_0 M^{-1}(s) g(s) \, \mathrm d s \bigg).
\ee
\end{proposition}
\begin{proof}
Differentiating \eqref{prop_inhom} with respect to $t$ and using that $M$ is the fundamental solution of \eqref{fund_solution}, we find that
\begin{align*}
f'(t)=\mathrm D f(t)&=M'(t) v_0 + M'(t) \int^t_0 M^{-1}(s) g(s) \, \mathrm d s + M(t) M^{-1}(t) g(t) \\
&= (R + V) M(t) \bigg( v_0 + \int^t_0 M^{-1}(s) g(s) \, \mathrm d s \bigg) + g(t)\\
&=(R+V) f(t) +g(t)
\end{align*}
and $f(0)=v_0$.
\end{proof}
As a corollary we obtain a formula for the $\lambda$-derivative $\dot M$ of $M$.
\begin{corollary} \label{Cor_lambda_deriv_M}
The $\lambda$-derivative $\dot M$ of $M$ is given by
\be \label{Cor_form_dotM}
\dot M(t) = M(t)    \int^t_0 M^{-1}(s) N(s) M(s) \, \mathrm d s,
\ee
where
\bew 
N=2 
  \begin{pmatrix}
    -2\lambda \I& \psi^1 \\
  \psi^2 & 2\lambda \I \\
  \end{pmatrix}.
\eew
In particular, $\dot M$ is analytic on $\C\times \LL$ and compact on $[0,\infty)\times \C\times \LL$ uniformly on bounded subsets of $[0,\infty)\times \C$.
\end{corollary}
\begin{proof}
Differentiation of $\mathrm D M = (R+V)M$
with respect to $\lambda$ gives
\begin{align*}
\mathrm D \dot M &= (R + V) \dot M + \frac{\mathrm d}{\mathrm d \lambda}\Big( R(\lambda) + V(\lambda)  \Big) M 
=  (R + V) \dot M  + N  M,
\end{align*}
and Proposition \ref{Prop_inhom_eq} yields \eqref{Cor_form_dotM}. The second claim is a consequence of Proposition \ref{Prop_M_comp}.
\end{proof}

The fundamental solution $M$ of the ZS-system is related to the fundamental solution $K$ of the AKNS-system by 
\be \label{fund_K}
K=T^{-1} M T, 
\ee
cf.~\eqref{def_T}. That is, if
\bew
M = 
  \begin{pmatrix}
   m_1 & m_2  \\
   m_3&m_4 \\
  \end{pmatrix}, \quad
  K = 
  \begin{pmatrix}
   k_1 & k_2  \\
   k_3&k_4 \\
  \end{pmatrix},
\eew
then
\begin{align*}
k_1= \frac{m_1+m_2+m_3+m_4}{2}, \quad &k_2=\frac{m_1-m_2+m_3-m_4}{-2\I}, \\
k_3= \frac{m_1+m_2-m_3-m_4}{2\I}, \quad &k_4=\frac{m_1-m_2-m_3+m_4}{2}.
\end{align*}
The fundamental solution for the zero potential in AKNS coordinates is therefore given by
\bew
\e^{2 \I \lambda^2 \sigma_2 t} =  
  \begin{pmatrix}
   \cos 2\lambda^2 t & \sin 2\lambda^2 t  \\
   -\sin 2\lambda^2 t&\cos 2\lambda^2 t 
  \end{pmatrix}, \quad
    \sigma_2 = 
  \begin{pmatrix}
   0 & -\I  \\
   \I &0 
  \end{pmatrix}.
\eew

\begin{remark}
It is obvious that all results in this section possess an analogous version in which the space $\LL$ of $1$-periodic potentials is replaced by the space $\LL_\tau$ of potentials defined on the interval $[0,\tau]$, $\tau>0$.
\end{remark}


\subsection{Leading order asymptotics} \label{subsec_asympM}

The results in this section hold for $0\leq t\leq1$ and hence apply to the time-periodic problem we are primarily interested in.

It was pointed out in~\cite{LeFo15} that the fundamental matrix solution $M$ of \eqref{fund_solution} for a potential with sufficient smoothness and decay admits an asymptotic expansion (as $|\lambda| \to \infty$) of the form
\be \label{M_formal}
M(\lambda,t) = 
\bigg(\id + \frac{Z_1(t)}{\lambda} + \frac{Z_2(t)}{\lambda^2} + \cdots \bigg) \e^{-2\I \lambda^2 t \sigma_3}
+ \bigg(\frac{W_1(t)}{\lambda} + \frac{W_2(t)}{\lambda^2} +  \cdots \bigg) \e^{2\I \lambda^2 t \sigma_3},
\ee
where the matrices $Z_k$, $W_k$, $k=1,2, \dots$, can be explicitly expressed in terms of the potential and therefore only depend on the time variable $t\geq0$, and satisfy $Z_k(0) + W_k(0)=0$ for all integers $k\geq1$.
This suggests that $M$ satisfies 
\bew
M(\lambda,t) = \e^{-2\I \lambda^2 t \sigma_3} + \mathcal O\big(|\lambda|^{-1} \, \e^{2 |\Im (\lambda^2)| t}\big) \qquad
\text{as} \quad |\lambda|\to \infty
\eew 
for $t$ within a given bounded interval. These considerations suggest the following result. 

\begin{theorem}[Asymptotics of $M$ and $\dot{M}$ as $|\lambda| \to \infty$] \label{thm_asymptotics_M_1}
Uniformly on $[0,1]\times\C$ and on bounded subsets of $\LL_1$,
\bew
M(t,\lambda,\psi) = E_\lambda(t) +  \mathcal O\big(|\lambda|^{-1} \,\e^{2|\Im (\lambda^2)|t}\big)
\eew
in the sense that there exist constants $C>0$ and $K>0$ such that 
\be \label{thm_asymptotics_M_2}
 |\lambda|\, \e^{-2|\Im (\lambda^2)|t} \, | M(t,\lambda,\psi) - E_\lambda(t)| \leq C
\ee
uniformly for all $0\leq t \leq 1$, all $\lambda \in\C$ with $|\lambda|>K$ and all $\psi$ contained in a given bounded subset of $\LL_1$. 
Moreover, the $\lambda$-derivative of $M$ satisfies
\be\label{thm_asymptotics_Mdot_2}
\dot M(t,\lambda,\psi) = \dot E_\lambda(t) +  \mathcal O\big(\e^{2|\Im (\lambda^2)| t}\big)
\ee
uniformly on $[0,1]\times\C$ and on bounded subsets of $\LL_1$.
\end{theorem}

Theorem \ref{thm_asymptotics_M_1} will be established via a series of lemmas. We first introduce some notation and briefly discuss the idea of the proof.

For $ \lambda \in \C$ and $\psi \in \LL_1$, let $M$ be the fundamental solution of \eqref{fund_solution}, which will be considered on the unit interval $[0,1]$. We set
\bew
\theta \coloneqq 2 \lambda ^2 
\eew
and define $M^+$ and $M^-$ by
\bew
M^+(t,\lambda,\psi) \coloneqq M(t,\lambda,\psi) \, \e^{\I \theta t \sigma_3}, 
\quad M^-(t,\lambda,\psi) \coloneqq M(t,\lambda,\psi) \, \e^{-\I \theta t \sigma_3},
\eew
For a given complex $2\times2$-matrix
\bew
A = 
\begin{pmatrix}
a & b \\
c & d
\end{pmatrix}
\eew 
we denote by $A^{\mathrm d}$ its diagonal part and by $A^{\mathrm{od}}$ its off-diagonal part, i.e.
\bew
A^{\mathrm d} = 
\begin{pmatrix}
a &  \\
 & d
\end{pmatrix}, \quad 
A^{\mathrm{od}} = 
\begin{pmatrix}
 & c \\
b & 
\end{pmatrix}.
\eew
We will always identify a potential $\psi \in \LL_1$, with its absolutely continuous version. This allows us to evaluate $\psi$ at each point; we set $\psi_0\coloneqq \psi(0)$ and $\psi^j_0\coloneqq \psi^j(0)$ for $j=1,2,3,4$. 
For a given potential $\psi \in \LL_1$, $t\in[0,1]$ and $\lambda \in \C$ we define
\bew
Z_p(t,\lambda,\psi) \coloneqq \id + \frac{Z_1(t,\psi)}{\lambda} + \frac{Z^{\mathrm{od}}_2(t,\psi)}{\lambda^2} ,
\eew
where
\begin{align*}
Z_1(t,\psi) \coloneqq -\frac{\I}{2} 
\begin{pmatrix}
 & \psi^1 \\
-\psi^2 & 
\end{pmatrix} 
+ \frac{1}{2} \Gamma \sigma_3, \quad
Z^{\mathrm{od}}_2(t,\psi) \coloneqq \frac{1}{4}
\begin{pmatrix}
 & \psi^3 + \I \psi^1 \Gamma   \\
\psi^4 + \I \psi^2 \Gamma & 
\end{pmatrix},
\end{align*}
with
\bew
\Gamma \equiv \Gamma(t,\psi) \coloneqq \int^t_0 ( \psi^1 \psi^4 -  \psi^2 \psi^3) \, \mathrm d \tau.
\eew
Furthermore we set
\bew
W_p(t,\lambda,\psi) \coloneqq \frac{W_1(t,\psi)}{\lambda} + \frac{W_2(t,\psi)}{\lambda^2} + \frac{W^{\mathrm d}_3(t,\psi)}{\lambda^3} ,
\eew
where
\begin{align*}
&W_1 (t,\psi) = W_1 (\psi) \coloneqq \frac{\I}{2}
\begin{pmatrix}
 & \psi^1_0 \\
-\psi^2_0 & 
\end{pmatrix}, \\
&W_2(t,\psi) \coloneqq - \frac{1}{4} 
\begin{pmatrix}
\psi^2_0 \psi^1   & -\I \psi^1_0 \Gamma +\psi^3_0 \\
- \I \psi^2_0 \Gamma +\psi^4_0 &  \psi^1_0 \psi^2
\end{pmatrix}, \\
&W^{\mathrm d}_3(t,\psi) \coloneqq  \frac{\I}{8} 
\begin{pmatrix}
-\psi^2_0( \psi^3 + \I \psi^1 \Gamma) +\psi^4_0 \psi^1&  \\
 & \psi^1_0 (\psi^4 + \I\psi^2 \Gamma) -\psi^3_0 \psi^2
\end{pmatrix}.
\end{align*}
We finally define $M_p$, which will serve as an approximation of $M$, by
\bew
M_p(t,\lambda,\psi)\coloneqq Z_p(t,\lambda,\psi) \, \e^{-\I \theta t \sigma_3} + W_p(t,\lambda,\psi) \, \e^{\I \theta t \sigma_3},
\eew
and set
$M^+_p \coloneqq M_p \,\e^{\I \theta t \sigma_3}$, $M^-_p \coloneqq M_p \,\e^{-\I \theta t \sigma_3}$, i.e.
\begin{align*}
M^+_p(t,\lambda,\psi) &= Z_p(t,\lambda,\psi) +  W_p(t,\lambda,\psi)\, \e^{2\I \theta t \sigma_3},
	 \\
M^-_p(t,\lambda,\psi) &= Z_p(t,\lambda,\psi) \,\e^{-2\I \theta t \sigma_3} +  W_p(t,\lambda,\psi). 
\end{align*}
Letting $Q_j$, $j=1,2,3,4$, denote the four open quadrants of the complex $\lambda$-plane, we set
\bew
D_+ \coloneqq Q_1 \cup Q_3 \quad \text{and} \quad D_-\coloneqq Q_2 \cup Q_4.
\eew
For an arbitrary complex number $\lambda = x + \I y$ with $x,y\in \R$ and $t\geq0$, it holds that 
\bew
|\e^{2\I \lambda^2 t}| = \e^{-4 x y} \leq 1 \iff \lambda \in \overbar{D_+},
\qquad
|\e^{-2\I \lambda^2 t}| = \e^{4 x y}  \leq 1 \iff \lambda \in \overbar{D_-}.
\eew

We will prove Theorem \ref{thm_asymptotics_M_1} by establishing asymptotic estimates for the distance between the fundamental solution $M$ and the explicit expression $M_p$ that approximates $M$.
For this purpose we will consider the columns of $M^+$ and $M^-$ separately and compare them with the columns of $M^+_p$ and $M^-_p$, respectively, after restricting attention to either $\overline{D_+}$ or $\overline{D_-}$. By combining all possible combinations, we are able to infer asymptotic estimates for the full matrix $M$ valid on the whole complex plane.

\begin{remark} \label{rem_derivation_of_Z_k_W_k}
For a given smooth potential $\psi$, the matrices $Z_k$ and $W_k$ can be determined recursively up to any order $k\geq0$ by integration by parts. Indeed, note that $V = V_0 + \lambda V_1$ where 
\bew
V_0 \coloneqq
\begin{pmatrix}
-\I \psi^1 \psi^2 & \I \psi^3 \\
- \I \psi^4 & \I \psi^1 \psi^2
\end{pmatrix}, \quad 
V_1 \coloneqq
\begin{pmatrix}
 & 2 \psi^1 \\
2\psi^2 &
\end{pmatrix}. 
\eew 
Assuming that the formal expression 
\bew
\bigg( \sum^\infty_{k=-1} \frac{Z_k(t,\psi)}{\lambda^k} \bigg)  \e^{-\I \theta t \sigma_3} +
\bigg( \sum^\infty_{k=-1} \frac{W_k(t,\psi)}{\lambda^k} \bigg)  \e^{\I \theta t \sigma_3}
\eew
with 
\bew
Z_0(t,\psi) \equiv \id, \quad Z_{-1}(t,\psi)=W_{-1}(t,\psi)=W_0(t,\psi) \equiv 0
\eew
solves \eqref{fund_solution}, one infers the following recursive equations for the coefficients $Z_k$ and $W_k$: 
\begin{align*}
(Z_k)_t + 4 \I \sigma_3 Z^{\mathrm{od}}_{k+2} &= V_0 Z_k + V_1 Z_{k+1},  \\
(W_k)_t + 4 \I \sigma_3 W^{\mathrm{d}}_{k+2} &= V_0 W_k + V_1 W_{k+1} 
\end{align*}
for all integers $k\geq -1$ and $Z_k(0,\psi) + W_k(0,\psi) = 0$ for all integers $k\geq1$.

For $\psi \in \LL_1$, the matrices $Z_p$ and $W_p$ satisfy
\bew
Z_p(0,\lambda,\psi) + W_p(0,\lambda,\psi) = \id + \mathcal O(|\lambda|^{-2}),
\eew
since the values of $Z^{\mathrm d}_2$ are not determined,
which turns out to be sufficient to prove the asymptotic estimates of $M$ asserted in Theorem \ref{thm_asymptotics_M_1}.
\end{remark}

\begin{lemma} \label{lem_M^+_fund_sol_char}
Let $\psi \in \LL_1$ be an arbitrary potential.
Then $M$ is the fundamental matrix solution of the Cauchy problem \eqref{fund_solution} if and only if $M^+$ satisfies
\be \label{fund_sol_M^+}
M^+_t +  2 \I \theta \sigma_3 (M^+)^{\mathrm{od}} = V M^+, \quad M^+(0,\lambda) = \id.
\ee
\end{lemma}
\begin{proof}
By applying the product rule, assuming that \eqref{fund_solution} holds and noting that $\sigma_3$ commutes with diagonal matrices, we obtain 
\begin{align*}
M^+_t = (M \e^{\I \theta t \sigma_3})_t &= M_t \, \e^{\I \theta t \sigma_3} + M \, \e^{\I \theta t \sigma_3} \, \I \theta \sigma_3 \\
&= (V M - \I \theta \sigma_3 \, M) \, \e^{\I \theta t \sigma_3} + \I \theta M^+   \sigma_3 \\
&= V M^+ - \I \theta [\sigma_3, M^+]  \\
&= V M^+ - 2 \I \theta \sigma_3 (M^+)^{\mathrm{od}}.
\end{align*}
Conversely, if \eqref{fund_sol_M^+} holds, we similarly obtain
\bew
M_t \, \e^{\I \theta t \sigma_3} = (V M - \I \theta \sigma_3 \, M) \, \e^{\I \theta t \sigma_3},
\eew
and a multiplication with $\e^{-\I \theta t \sigma_3}$ from the right yields that $M$ satisfies the differential equation in \eqref{fund_solution}. 
The statement concerning the initial conditions holds because $M(0,\lambda)=M^+(0,\lambda)$.
\end{proof}

The following lemma is concerned with the invertibility of $Z_p$. 
We set $\C^K \coloneqq \{\lambda \in \C \colon |\lambda|>K \}$ for $K>0$, and denote by $B_r(0,\LL_1)$ the ball of radius $r>0$ in $\LL_1$ centered at $0$. 
\begin{lemma} \label{lem_W_p_Z_p_invertible}
Let $r>0$. 
There exists a constant $K_r>0$ such that $Z_p$ is invertible on 
$[0,1]\times \C^{K_r} \times B_r(0,\LL_1)$ with
\be \label{Z_p_inverse}
Z_p^{-1}(t,\lambda,\psi) =
\sum^\infty_{n=0} \bigg(- \frac{Z_1(t,\psi)}{\lambda} - \frac{Z^{\mathrm{od}}_2(t,\psi)}{\lambda^2}  \bigg)^n.
\ee
\label{lem_W_p_Z_p_invertible_1}
\end{lemma}
\begin{proof}
We use the general fact that if an element $A$ of a Banach algebra $(\mathcal A,\|\cdot\|)$ satisfies $\|A\| <1$, then $\id - A$ is invertible and its inverse is given by the Neumann series $\sum_{n\geq0}A^n$.
Let $K_r>0$ be so large that
\bew
\bigg| \frac{Z_1(t,\psi)}{\lambda} + \frac{Z^{\mathrm{od}}_2(t,\psi)}{\lambda^2} \bigg| < \frac{1}{2}
\eew 
for all $t\in[0,1]$, $\lambda\in \C^{K_r}$ and $\psi \in B_r(0,\LL_1)$. This can always be achieved, because the functions $\{\psi^j\}_1^4$, and hence also the functions $| Z_1(t,\psi)|$ and $| Z^{\mathrm{od}}_2(t,\psi)|$, are uniformly bounded on $[0,1]\times B_r(0,\LL_1)$.
It follows that the inverse of $Z_p$ on $[0,1]\times \C^{K_r} \times B_r(0,\LL_1)$ exists and is given by its Neumann series, i.e.~\eqref{Z_p_inverse} is satisfied.
\end{proof}

Lemma \ref{lem_W_p_Z_p_invertible} and its proof suggest the introduction of the following notation.
\begin{definition} \label{def_K_K'}
For each $r>0$, we define
\bew
K_r \coloneqq
\inf_{\substack{|\lambda|>1\\ \lambda\in\C}} \big\{ 
|\id - Z_p(t,\lambda,\psi)| < 1/2 \; \forall t \in [0,1] \, \forall \psi \in B_r(0,\LL_1) \big\}.
\eew
\end{definition}

\begin{corollary} \label{cor_W_p_Z_p_invertible}
Let $r>0$.
The matrix $Z_p$ is invertible on $[0,1]\times \C^{K_r}\times B_r(0,\LL_1)$ and its inverse $Z^{-1}_p$ is given by \eqref{Z_p_inverse}. Both $Z_p$ and $Z^{-1}_p$ are uniformly bounded on $[0,1]\times \C^{K_r}\times B_r(0,\LL_1)$. Furthermore,
\be \label{Z_p_inverse_asymp}
Z^{-1}_p =  \id - \frac{Z_1}{\lambda} + \frac{Z^2_1 - Z^{\mathrm{od}}_2}{\lambda^2} 
+ \mathcal O\big(|\lambda|^{-3}\big)
\ee
uniformly on $[0,1]\times \C^{K_r}\times B_r(0,\LL_1)$ as $|\lambda|\to\infty$.
\end{corollary}
\begin{proof}
The expansion \eqref{Z_p_inverse_asymp} follows directly from \eqref{Z_p_inverse}, the uniform $\mathcal O(|\lambda|^{-3})$ error follows from the uniform convergence of the respective Neumann series on $[0,1]\times \C^{K_r}\times B_r(0,\LL_1)$, see the proof of Lemma \ref{lem_W_p_Z_p_invertible}.
\end{proof}

For $t\in[0,1]$, $\lambda\in \C$ and $\psi\in\LL_1$, we define
\bew
A_{p,Z} \coloneqq \big[ (\Dt Z_p) -\I \theta Z_p \sigma_3 \big] Z^{-1}_p,  
\quad
\mathcal A \coloneqq V - \I \theta \sigma_3,
\quad
\Delta_Z \coloneqq \mathcal A - A_{p,Z}, 
\eew
whenever the inverse $Z_p^{-1}$ exists. By Lemma \ref{lem_W_p_Z_p_invertible}, the inverse of $Z_p$ exists uniformly on $[0,1]$ and on bounded sets in $\LL_1$ provided that $|\lambda|$ is large enough.

\begin{lemma}  \label{lem_(Z^{-1}_p M^+)_t}
For $\lambda\in \C$ and $\psi \in \LL_1$, let $M$ be the fundamental solution of \eqref{fund_solution} on the unit interval.  
If $|\lambda|$ is so large that $Z^{-1}_p$ exists for all $t\in[0,1]$, then
\begin{align} \label{(Z^{-1}_p M^+)_t}
&(Z^{-1}_p M^+)_t = Z^{-1}_p \Delta_Z M^+ - \I \theta [\sigma_3, Z^{-1}_p M^+], \\
&(Z^{-1}_p M^-)_t = Z^{-1}_p \Delta_Z M^- - \I \theta (\sigma_3 Z^{-1}_p M^- + Z^{-1}_p M^- \sigma_3). \label{(Z^{-1}_p M^-)_t}
\end{align}
\end{lemma}
\begin{proof}
By Lemma  \ref{lem_M^+_fund_sol_char} we can write $M^+_t = \mathcal A M^+ + \I \theta M^+ \sigma_3$, hence
equation \eqref{(Z^{-1}_p M^+)_t} is directly obtained by:
\begin{align*}
(Z^{-1}_p M^+)_t &= - Z^{-1}_p (\Dt Z_p) Z^{-1}_p M^+ + Z^{-1}_p M^+_t \\
&= - Z^{-1}_p (A_{p,Z} Z_p + \I \theta Z_p \sigma_3) Z^{-1}_p M^+ + Z^{-1}_p(\mathcal A M^+ + \I \theta M^+ \sigma_3) \\
&= Z^{-1}_p \Delta_Z M^+ - \I \theta [\sigma_3, Z^{-1}_p M^+].
\end{align*}
Equation \eqref{(Z^{-1}_p M^-)_t} is similarly obtained by noting that
\begin{align*}
M^-_t = (M^+ \e^{-2\I \theta t \sigma_3})_t &= (V M^+ - \I \theta [\sigma_3,M^+])  \e^{-2\I \theta t \sigma_3} -2\I \theta M^+ \e^{-2\I \theta t \sigma_3} \sigma_3 \\
&=  \mathcal A M^- - \I \theta M^- \sigma_3.
\end{align*}
\end{proof}

For $z\in\C$, we define the linear map $\e^{z \hat\sigma_3}$ on the space of complex $2\times2$-matrices by 
\bew
\e^{z\hat\sigma_3}(A) \coloneqq \e^{z \sigma_3} A \e^{-z \sigma_3};
\eew
furthermore we define $\e^{z \check\sigma_3}$ via
\bew
\e^{z\check\sigma_3}(A) \coloneqq \e^{z \sigma_3} A \e^{z \sigma_3}.
\eew

Lemma \ref{lem_(Z^{-1}_p M^+)_t} yields Volterra equations for $M^+$ and $M^-$.

\begin{lemma} \label{lem_Volterra_M^+_M^-}
For $\lambda\in \C$ and $\psi \in \LL_1$, let $M$ be the fundamental solution of \eqref{fund_solution} on the unit interval. 
If $|\lambda|$ is so large that $Z^{-1}_p$ exists for all $t\in[0,1]$, then $M^+$ satisfies
\begin{align}
\begin{aligned} \label{eq_lem_Volterra_M^+_M^-_1}
M^+(t,\lambda,\psi) &=  Z_p(t,\lambda,\psi) \, \e^{-\I \theta t \hat \sigma_3}[Z^{-1}_p(0,\lambda,\psi)] \\
& \quad+ \int^t_0 Z_p(t,\lambda,\psi) \, \e^{-\I \theta (t-\tau) \hat \sigma_3} [(Z^{-1}_p \Delta_Z M^+)(\tau,\lambda,\psi)] \,\mathrm d \tau,
\end{aligned}
\end{align}
and $M^-$ satisfies
\begin{align}
\begin{aligned} \label{eq_lem_Volterra_M^+_M^-_2}
M^-(t,\lambda,\psi) &=  Z_p(t,\lambda,\psi) \, \e^{-\I \theta t \check \sigma_3}[Z^{-1}_p(0,\lambda,\psi)] \\
& \quad+ \int^t_0 Z_p(t,\lambda,\psi) \, \e^{-\I \theta (t-\tau) \check \sigma_3} [(Z^{-1}_p \Delta_Z M^+)(\tau,\lambda,\psi)] \,\mathrm d \tau.
\end{aligned}
\end{align}
\end{lemma}
\begin{proof}
Using identity \eqref{(Z^{-1}_p M^+)_t} in Lemma \ref{lem_(Z^{-1}_p M^+)_t}, we infer that
\be \label{pr_lem_Volterra_M^+_M^-_1}
\big(\e^{\I \theta t \hat \sigma_3} (Z^{-1}_p M^+)\big)_t = \e^{\I \theta t \hat \sigma_3} (Z^{-1}_p \Delta_Z M^+).
\ee
In order to obtain \eqref{eq_lem_Volterra_M^+_M^-_1}, we first integrate \eqref{pr_lem_Volterra_M^+_M^-_1} from $0$ to $t$ and use that $M^+(0,\lambda)=\id$ to determine the integration constant. Applying $\e^{-\I \theta t \hat \sigma_3}$ to both sides of the resulting integral equation and multiplying by $Z_p$ from the left, we find \eqref{eq_lem_Volterra_M^+_M^-_1}. 

The Volterra equation for $M^-$ follows in an analogous way from the equation
\bew 
(\e^{\I \theta t \check \sigma_3} (Z^{-1}_p M^-))_t = \e^{\I \theta t \check \sigma_3} (Z^{-1}_p \Delta_Z M^-),
\eew
which is a consequence of \eqref{(Z^{-1}_p M^-)_t}.
\end{proof}

For a $t$-dependent matrix $A$ with entries in $L^p([0,1],\C)$ we define 
\bew
\|A\|_{L^p([0,1],\C)} \coloneqq \bigg(\int^1_0 |A(t)|^p \, \mathrm d t \bigg)^{1/p}, \quad 1\leq p < \infty.
\eew

\begin{lemma} \label{lem_Z_k_W_k_O(1)}
Let $B$ be an arbitrary bounded subset of $\LL_1$ and let $1\leq q \leq 2$. Then 
\bew
\|\Dt Z_1(\psi) \|_{L^q([0,1],\C)} = \mathcal O(1),  \quad \| \Dt Z^{\mathrm{od}}_2(\psi)\|_{L^q([0,1],\C)} = \mathcal O(1), 
\eew
uniformly on $B$. 
\end{lemma}
\begin{proof}
The case $q=2$ follows directly from the definitions of $Z_1$ and $Z^{\mathrm{od}}_2$, the continuity of the operator
\bew
\|\cdot\|_{L^2([0,1],\C)}  \circ \Dt  \colon H^1(0,1) \to \R,
\eew
and the fact that $H^1(0,1)$ is an algebra. 
The cases $1\leq q < 2$ follow from the case $q=2$ in view of the continuous embeddings $L^2([0,1],\C) \hookrightarrow  L^q([0,1],\C)$, $1\leq q < 2$. 
\end{proof}

\begin{lemma} \label{lem_Delta_ZW_asymp}
Let $r>0$. There exists a constant $C>0$ such that uniformly for $(\lambda,\psi)$ in $\C^{K_r}\times B_r(0,\LL_1)$,
\be\label{DeltaZbound}
 |\lambda| \,\|\Delta_Z(\lambda,\psi)\|_{L^1([0,1],\C)}
 \leq C.
\ee
\end{lemma}
\begin{proof}
Note that 
\be \label{pr_lem_Delta_ZW_asymp_1}
4 \I \sigma_3 Z^{\mathrm{od}}_1 = V_1, \quad 
4 \I \sigma_3 Z^{\mathrm{od}}_2 = V_0 + V_1 Z_1
\ee
for arbitrary $\psi\in\LL_1$. By Corollary \ref{cor_W_p_Z_p_invertible} the asymptotic estimate 
\eqref{Z_p_inverse_asymp} holds uniformly on $[0,1]\times \C^{K_r} \times B_r(0,\LL_1)$ as $|\lambda|\to \infty$.
In particular, $\Delta_Z$ is well-defined on $[0,1]\times \C^{K_r} \times B_r(0,\LL_1)$ and satisfies
\begin{align} \label{pr_lem_Delta_ZW_asymp_1b}
\Delta_Z =V_0 &+ \lambda V_1 - 2 \I \lambda^2 \sigma_3 
+  2 \I \lambda^2 \bigg(\id + \frac{Z_1}{\lambda} + \frac{Z^{\mathrm{od}}_2}{\lambda^2}\bigg) \sigma_3
\bigg( \id - \frac{Z_1}{\lambda} + \frac{Z^2_1 - Z^{\mathrm{od}}_2}{\lambda^2} \bigg) 
+ \mathcal O\big(|\lambda|^{-1}\big)
\end{align}
in $L^1([0,1],\C)$ uniformly on $\C^{K_r} \times B_r(0,\LL_1)$ as $|\lambda|\to \infty$, where we have used Lemma \ref{lem_Z_k_W_k_O(1)} to estimate the $\Dt Z_p$-term.
By keeping only the  $\lambda^k$-terms for $k=0,1,2$ in \eqref{pr_lem_Delta_ZW_asymp_1b} and by employing \eqref{pr_lem_Delta_ZW_asymp_1}, we obtain 
\begin{align*}
\Delta_Z &=V_0 + \lambda V_1 - 2 \I \lambda [\sigma_3,Z_1] + 2 \I [Z^{\mathrm{od}}_2,\sigma_3] + 2 \I [\sigma_3,Z_1] Z_1
+  \mathcal O\big(|\lambda|^{-1}\big) \\
&= V_0 - 4 \I \sigma_3 Z^{\mathrm{od}}_2 +  4 \I \sigma_3 Z^{\mathrm{od}}_1 Z_1 
+  \mathcal O\big(|\lambda|^{-1}\big) \\
&= 0
+  \mathcal O\big(|\lambda|^{-1}\big)
\end{align*}
in $L^1([0,1],\C)$ uniformly on $\C^K_r \times B_r(0,\LL_1)$ as $|\lambda|\to \infty$. 
\end{proof}

Let $[A]_1$ and $[A]_2$ denote the first and second columns of a $2\times2$-matrix $A$. 
Let $|[A]_i|$, $i=1,2$, denote the standard $\C^2$-norm of the vector $[A]_i$. 

\begin{lemma} \label{lem_asymp_M_Z_p}
Let $r>0$.
There exists a constant $C>0$ such that 
\be \label{lem_asymp_M_Z_p_1}
|\lambda|\, \big|\big[M^+(t,\lambda,\psi) - Z_p (t,\lambda,\psi) \, \e^{-\I \theta t \hat \sigma_3}\big(Z^{-1}_p(0,\lambda,\psi)\big) \big]_2 \big|
\leq C,
\ee
\be \label{lem_asymp_M_Z_p_2_2}
|\lambda| \,\big|\big[M^-(t,\lambda,\psi) - Z_p (t,\lambda,\psi) \, \e^{-\I \theta t \check \sigma_3}\big(Z^{-1}_p(0,\lambda,\psi)\big) \big]_1\big| 
\leq C
\ee
uniformly on $[0,1]  \times \overbar{D^{K_r}_-} \times B_r(0,\LL_1)$, and
\be \label{lem_asymp_M_Z_p_1_2}
|\lambda|\, \big|\big[M^-(t,\lambda,\psi) - Z_p (t,\lambda,\psi) \, \e^{-\I \theta t \check \sigma_3}\big(Z^{-1}_p(0,\lambda,\psi)\big) \big]_2 \big|
\leq C,
\ee
\be \label{lem_asymp_M_Z_p_2}
|\lambda| \,\big|\big[M^+(t,\lambda,\psi) - Z_p (t,\lambda,\psi) \, \e^{-\I \theta t \hat \sigma_3}\big(Z^{-1}_p(0,\lambda,\psi)\big) \big]_1\big| 
\leq C
\ee
uniformly on $[0,1]  \times \overbar{D^{K_r}_+} \times B_r(0,\LL_1)$.
\end{lemma}
\begin{proof}
For $\lambda \in \C^{K_r}$, the functions
\begin{align*}
\mathcal M(t,\lambda,\psi) &\coloneqq [ M^+(t,\lambda,\psi)]_2, 
	\\
\mathcal M_0(t,\lambda,\psi) &\coloneqq \big[Z_p (t,\lambda,\psi) \e^{-\I \theta t \hat \sigma_3}\big(Z^{-1}_p(0,\lambda,\psi)\big) \big]_2,
	\\
E(t,\tau,\lambda,\psi) & \coloneqq Z_p (t,\lambda,\psi) 
\begin{pmatrix}
\e^{-2\I \theta (t-\tau)} & \\
&1
\end{pmatrix}
  Z^{-1}_p (\tau,\lambda,\psi),
\end{align*}
are well-defined on their domains $[0,1]\times \C^{K_r} \times B_r(0,\LL_1)$ and $[0,1]^2\times \C^{K_r} \times B_r(0,\LL_1)$ respectively, where the inverse $Z^{-1}_p$ is given by \eqref{Z_p_inverse} and is uniformly bounded on $[0,1]\times \C^{K_r} \times B_r(0,\LL_1)$ by Lemma \ref{lem_W_p_Z_p_invertible} and Corollary \ref{cor_W_p_Z_p_invertible}.
Due to Lemma \ref{lem_Volterra_M^+_M^-}, $\mathcal M$ satisfies the following Volterra equation for $t\in[0,1]$, $\lambda \in \C^{K_r}$ and $\psi \in B_r(0,\LL_1)$:
\bew 
\mathcal M (t,\lambda) = \mathcal M_0(t,\lambda) 
+ \int^t_0 E(t,\tau,\lambda) \Delta_Z(\tau,\lambda) \mathcal M(\tau,\lambda)  \, \mathrm d \tau,
\eew
where the $\psi$-dependence has been suppressed for simplicity. Thus $\mathcal M$ admits the power series representation
\bew
\mathcal M (t,\lambda)  = \sum^\infty_{n=0} \mathcal M_n(t,\lambda), 
\eew
which converges (pointwise) absolutely and uniformly on $[0,1]\times\C^{K_r} \times B_r(0,\LL_1)$, where 
\bew
\mathcal M_n(t,\lambda) \coloneqq 
 \int^t_0 E(t,\tau,\lambda) \Delta_Z(\tau,\lambda) \mathcal M_{n-1}(\tau,\lambda)  \, \mathrm d \tau \qquad (n\geq1)
\eew
satisfies the estimate
\begin{align*}
|\mathcal M_n(t,\lambda)| &\leq \int_{0\leq \tau_n \leq \cdots \leq \tau_1 \leq t} \prod^n_{i=1}
|E(t,\tau_i,\lambda) \Delta_Z(\tau,\lambda) \mathcal M_0(\tau,\lambda)| 
\, \mathrm d \tau_n \cdots \mathrm d \tau_1 \\
&\leq \frac{1}{n!} \bigg (   \int^t_0 |  E(t,\tau,\lambda) | \, | \Delta_Z(\tau,\lambda)  |  \, |\mathcal M_0(\tau,\lambda)|   \, \mathrm d \tau \bigg)^n
\end{align*}
uniformly on $[0,1]\times \C^{K_r} \times B_r(0,\LL_1)$.
The functions $E$ and $\mathcal M_0$ satisfy 
\begin{align*}
|\mathcal M_0(t,\lambda)| &\leq |Z_p (t,\lambda)| \, | Z^{-1}_p(0,\lambda)|,
	 \\
| E(t,\tau,\lambda) | & \leq  | Z_p (t,\lambda) | \, | Z^{-1}_p (\tau,\lambda)|,
\end{align*}
for $0\leq \tau \leq t\leq 1$ and $(\lambda, \psi) \in \overbar{D^{K_r}_-} \times B_r(0,\LL_1)$. Therefore, in view of Corollary \ref{cor_W_p_Z_p_invertible} and Lemma \ref{lem_Delta_ZW_asymp}, there exists a constant $C>0$ such that 
\bew
|\mathcal M_n(t,\lambda)| \leq \frac{C^n}{n! \, |\lambda|^n}
\eew
uniformly on $[0,1]\times \overbar{D^{K_r}_-}\times  B_r(0,\LL_1)$, and thus
\begin{align*}
|\mathcal M(t,\lambda) - \mathcal M_0 (t,\lambda)| \leq \sum^\infty_{n=1} \frac{C^n}{n! \, |\lambda|^n} 
\leq \frac{C  \e^{\frac{C}{ |\lambda|}}}{ |\lambda|} 
\end{align*}
uniformly on $[0,1]\times \overbar{D^{K_r}_-}\times B_r(0,\LL_1)$. This proves \eqref{lem_asymp_M_Z_p_1}; the proofs of \eqref{lem_asymp_M_Z_p_2_2}--\eqref{lem_asymp_M_Z_p_2} are similar.
\end{proof}

\begin{lemma} \label{lem_asymp_M_p_Z_p}
Let $r>0$.
There exists a constant $C>0$ such that 
\be \label{asymp_M_p_Z_p_1}
 |\lambda|^2 \, \big| \big[  Z_p(t,\lambda,\psi) \, \e^{-\I \theta t \hat \sigma_3}\big(Z^{-1}_p(0,\lambda,\psi)\big)  
- M^+_p (t,\lambda,\psi)  \big]_2  \big| \leq C,
\ee
\be \label{asymp_M_p_Z_p_4}
|\lambda|^2 \, 
\big| \big[  Z_p(t,\lambda,\psi) \, \e^{-\I \theta t \check \sigma_3}\big(Z^{-1}_p(0,\lambda,\psi)\big)  
- M^-_p (t,\lambda,\psi)  \big]_1  \big| \leq C
\ee
uniformly on $[0,1]\times \overbar{D^{K_r}_-} \times B_r(0,\LL_1)$, and
\be \label{asymp_M_p_Z_p_3}
|\lambda|^2 \, 
\big| \big[  Z_p(t,\lambda,\psi) \, \e^{-\I \theta t \check \sigma_3}\big(Z^{-1}_p(0,\lambda,\psi)\big)  
- M^-_p (t,\lambda,\psi)  \big]_2  \big| \leq C,
\ee
\be \label{asymp_M_p_Z_p_2}
|\lambda|^2 \, \big| \big[  Z_p(t,\lambda,\psi) \, \e^{-\I \theta t \hat \sigma_3}\big(Z^{-1}_p(0,\lambda,\psi)\big)  
- M^+_p (t,\lambda,\psi)  \big]_1  \big| \leq C
\ee
uniformly on $[0,1]\times \overbar{D^{K_r}_+} \times B_r(0,\LL_1)$. 
\end{lemma}
\begin{proof}
Since
\bew
\big[\e^{\I \theta t \hat \sigma_3} \big(Z_1(0,\psi) \big)\big]_2 = 
\e^{2\I \theta t } 
\begin{pmatrix}
-\frac{\I}{2} \psi^1_0 \\
0
\end{pmatrix}
= \e^{2\I \theta t } \big[Z_1(0,\psi)\big]_2
= - \e^{2\I \theta t } \big[W_1(0,\psi)\big]_2
\eew 
for $\psi\in\LL_1$, Corollary \ref{cor_W_p_Z_p_invertible} yields  
\begin{align*}
\big[Z_p(t,\lambda,\psi) \, &\e^{-\I \theta t \hat \sigma_3}(Z^{-1}_p(0,\lambda,\psi)) \big]_2 \\
&= \Big[ Z_p(t,\lambda,\psi) \, \e^{-\I \theta t \hat \sigma_3}
\Big(\id - \frac{Z_1(0,\psi)}{\lambda}  + \mathcal O(|\lambda|^{-2})\Big) \Big]_2\\
&= \Big[Z_p(t,\lambda,\psi)\Big]_2 - 
\frac{\e^{-2\I \theta t }  }{\lambda} \Big[Z_1(0,\psi)\Big]_2  
+ \mathcal O\bigg(\frac{\e^{2\Im \theta t }}{|\lambda|^2}\bigg)  \\
&= \Big[Z_p(t,\lambda,\psi) \Big]_2 + 
\frac{ \e^{-2\I \theta t }}{\lambda} \Big[ W_1(0,\psi) \Big]_2  
+ \mathcal O\bigg(\frac{\e^{2\Im \theta t }}{|\lambda|^2}\bigg)
\end{align*}
uniformly on $[0,1] \times \C^{K_r} \times B_r(0,\LL_1)$ as $|\lambda|\to\infty$. 
On the other hand, we have  that
\bew
\big[ M^+_p(t,\lambda,\psi) \big]_2 = 
\big[  Z_p (t,\lambda,\psi) \big]_2
+ \frac{ \e^{-2\I \theta t }}{\lambda} \Big[ W_1(0,\psi) \Big]_2  
+ \mathcal O\bigg(\frac{\e^{2\Im \theta t }}{|\lambda|^2}\bigg)
\eew
uniformly on $[0,1] \times \C^{K_r} \times B_r(0,\LL_1)$ as $|\lambda|\to\infty$. Since $\Im \theta \leq 0$ for $\lambda \in \overbar{D_-}$, the estimate \eqref{asymp_M_p_Z_p_1} follows. The estimates \eqref{asymp_M_p_Z_p_4}--\eqref{asymp_M_p_Z_p_3} are proved in a similar way.
\end{proof}

\smallskip

\begin{proof}[Proof of Theorem \ref{thm_asymptotics_M_1}]
The first assertion of the theorem follows by combining Lemma \ref{lem_asymp_M_Z_p} and Lemma \ref{lem_asymp_M_p_Z_p}. Let $r>0$. By \eqref{lem_asymp_M_Z_p_1} and \eqref{asymp_M_p_Z_p_1}, there exists a $C>0$ such that 
\bew
|\lambda| \, \big| \big[M^+(t,\lambda,\psi) - M^+_p (t,\lambda,\psi)  \big]_2 \big| \leq C
\eew
uniformly on $[0,1]\times \overbar{D^{K_r}_-} \times B_r(0,\LL_1)$. Thus, 
\bew
|\lambda| \, \e^{-2|\Im (\lambda^2)| t} \, \big| \big[M(t,\lambda,\psi) - M_p (t,\lambda,\psi)  \big]_2 \big| \leq C
\eew
uniformly on $[0,1]\times \overbar{D^{K_r}_-} \times B_r(0,\LL_1)$. Since 
$M_p  (t,\lambda,\psi) = E_\lambda (t) + \mathcal O(|\lambda|^{-1} \e^{2|\Im (\lambda^2)| t} )$ uniformly on 
$[0,1]\times \overbar{D^{K_r}_-} \times B_r(0,\LL_1)$ as $|\lambda|\to\infty$, we infer that there exists a $C>0$ such that 
\bew
|\lambda| \, \e^{-2|\Im (\lambda^2)| t} \, \big| \big[M(t,\lambda,\psi) - E_\lambda (t)  \big]_2 \big| \leq C
\eew
uniformly on $[0,1]\times \overbar{D^{K_r}_-} \times B_r(0,\LL_1)$. Analogously, by using \eqref{lem_asymp_M_Z_p_2} and \eqref{asymp_M_p_Z_p_2}, one infers the existence of a constant $C>0$ such that 
\bew
|\lambda| \, \e^{-2|\Im (\lambda^2)| t} \, \big| \big[M(t,\lambda,\psi) - E_\lambda (t)  \big]_1 \big| \leq C
\eew
uniformly on $[0,1]\times \overbar{D^{K_r}_+} \times B_r(0,\LL_1)$. 
The estimates \eqref{lem_asymp_M_Z_p_1_2} and \eqref{asymp_M_p_Z_p_3} (resp.~\eqref{lem_asymp_M_Z_p_2_2} and \eqref{asymp_M_p_Z_p_4}) yield the same asymptotic estimates for $[M-E_\lambda]_2$ (resp.~$[M-E_\lambda]_1$) for $\lambda$ restricted to  $\overbar{D^{K_r}_+}$ (resp.~$\overbar{D^{K_r}_-}$). In summary, this yields the existence of constants $C,K>0$ such that 
\bew
|\lambda| \, \e^{-2|\Im (\lambda^2)| t} \, \big| M(t,\lambda,\psi) - E_\lambda (t)   \big| \leq C
\eew
uniformly on $[0,1]\times \C^K \times B_r(0,\LL_1)$. 
This proves \eqref{thm_asymptotics_M_2}.

To prove \eqref{thm_asymptotics_Mdot_2}, we recall Cauchy's inequality: the derivative $f'$ of a holomorphic function $f \colon \C \supseteq G \to\C$  satisfying $|f(z)| \leq C$ on a disc $D(r,a) \subseteq G$ of radius $r$ centered at $a$ in the open domain $G$ can be estimated at the point $a$ by $|f'(a)| \leq C r^{-1}$.
According to the first part of the theorem, for any $r>0$ there is a $K>0$ such that
\bew
|\lambda|\, \e^{-2|\Im (\lambda^2)|t} \, \big|M(t,\lambda,\psi)-E_\lambda(t)\big|=\mathcal O (1)
\eew
uniformly for $t\in[0,1]$, $\lambda\in \C^K$ and $\psi\in B_r(0,\LL_1)$ as $|\lambda|\to\infty$. 
By applying Cauchy's inequality to this estimate, we immediately obtain \eqref{thm_asymptotics_Mdot_2}.
\end{proof}

Let us for $n\in\N$ and $i=1,2,3,4$ consider the complex numbers $\zeta^i_n$, which are given by
\bew
\zeta^1_n \coloneqq \sqrt{\frac{n \pi}{2}}, \quad 
\zeta^2_n \coloneqq -\sqrt{\frac{n \pi}{2}}, \quad 
\zeta^3_n \coloneqq \I \sqrt{\frac{n \pi}{2}}, \quad 
\zeta^4_n \coloneqq -\I \sqrt{\frac{n \pi}{2}}.
\eew
\begin{theorem} \label{thm_asymp_seq_near_per_ev}
For any potential $\psi\in \LL_1$ and any sequence $(z^i_n)_{n\in\N}$ of complex numbers, whose elements 
$z^i_n$, $i=1,2,3,4$, satisfy
\bew
z^i_n = \zeta^i_n + O\bigg(\frac{1}{\sqrt{n}}\bigg) \quad \text{as} \quad n\to\infty,
\eew 
it holds that
\begin{align}
\sup_{0\leq t\leq 1} \big| M(t,z^i_n)- E_{z^i_n}(t) \big| &= \mathcal O\big(n^{-1/2}\big), 
 \label{thm_asymp_seq_near_per_ev_1} \\
\sup_{0\leq t\leq 1} \big| \dot M(t,z^i_n)- \dot E_{z^i_n}(t) \big| &= \mathcal O(1)
\label{thm_asymp_seq_near_per_ev_1'}
\end{align}
as $n\to\infty$.
If moreover the squares $(z^i_n)^2$ satisfy
\bew
(z^i_n)^2=(\zeta^i_n)^2 + \mathcal O\big(n^{-1/2}\big) \quad \text{as} \quad n\to\infty,
\eew
then it holds additionally that
\be \label{thm_asymp_seq_near_per_ev_2}
\sup_{0\leq t\leq 1} \big| M(t,z_n)- E_{\zeta^i_n}(t) \big| =  \mathcal O\big(n^{-1/2}\big).
\ee
The estimates in \eqref{thm_asymp_seq_near_per_ev_1} and \eqref{thm_asymp_seq_near_per_ev_1'}  hold uniformly on bounded subsets of $\mathrm \LL_1$ and for sequences $(z^i_n)_{n\in\N}$, which  satisfy  $|z^i_n - \zeta^i_n |\leq C \sqrt{1/n}$ for all $n\geq1$  with a uniform constant $C>0$. 
The estimate in \eqref{thm_asymp_seq_near_per_ev_2} holds uniformly on bounded subsets of $\mathrm \LL_1$ and for sequences $(z_n)_{n\in\Z}$, which satisfy 
$|(z^i_n)^2-(\zeta^i_n)^2|\leq C \sqrt{1/n}$ for all $n\geq1$  with a uniform  constant $C>0$. 
\end{theorem}
\begin{proof}
The estimates \eqref{thm_asymp_seq_near_per_ev_1} and \eqref{thm_asymp_seq_near_per_ev_1'}
follow directly from Theorem \ref{thm_asymptotics_M_1}, because $\Im z^2_n = \mathcal O(1)$ as $|n|\to \infty$ by assumption, and therefore $\e^{2|\Im z^2_n| t}=\mathcal O (1)$ uniformly in $t\in[0,1]$ as $|n|\to\infty$.

To prove \eqref{thm_asymp_seq_near_per_ev_2} we note that $|\e^z -1| \leq | z | \, \e^{| z|}$ for arbitrary $z\in\C$, thus the additional restriction on $z^i_n$ implies that
\begin{align}
\begin{aligned}
\label{pr_thm_asymptotics_M_2}
\big| \e^{2 (z^i_n)^2 \I t} - \e^{2(\zeta^i_n)^2 \I t } \big| 
= \big| \e^{2( (z^i_n)^2 - (\zeta^i_n)^2) \I t} - 1 \big| 
= \mathcal O\big(n^{-1/2}\big)
\end{aligned}
\end{align}
uniformly for $t\in[0,1]$ as $n\to\infty$.
The triangle inequality implies that
\bew
\big| M(t,z^i_n)- E_{\zeta^i_n}(t) \big|
\leq \big| M(t,z^i_n)- E_{z^i_n}(t)\big| + \big| E_{z^i_n}(t) - E_{\zeta^i_n}(t) \big|
\eew
for $t\in[0,1]$, and hence \eqref{thm_asymp_seq_near_per_ev_2} follows from \eqref{thm_asymp_seq_near_per_ev_1} and \eqref{pr_thm_asymptotics_M_2}.
\end{proof}

\begin{remark}
For convenience, the asymptotic results in this section are stated for the space $\LL_\tau$ with $\tau=1$ (which contains the periodic space $\LL$ as a subspace). It is clear that analogous results hold for an arbitrary fixed $\tau>0$. 
\end{remark}

\vspace{0em}


\section{Spectra} \label{sec_spectra}
\vspace{0em}

\noindent
We will consider three different notions of spectra associated with the spectral problem \eqref{main_eq}: the Dirichlet, Neumann and periodic spectrum. These spectra are the zero sets of certain spectral functions, which are defined in terms of the entries of the fundamental solution $M$ evaluated at time $t=1$. We introduce the following notation:
\bew
\M \coloneqq M|_{t=1}, \quad \m_i \coloneqq m_i|_{t=1}, \quad i=1,2,3,4.
\eew

\vspace{0em}
\subsection{Dirichlet and Neumann spectrum}

We define the Dirichlet domain $\mathcal A_{\mathrm D}$ of the AKNS-system \eqref{AKNS_system} by\footnote{Following~\cite{GrebertKappeler14}, we define the Dirichlet spectrum in terms of $f_2$ and the Neumann spectrum in terms of $f_1$.}
\bew
\mathcal A_{\mathrm D} \coloneqq \big\{ f\in H^1([0,1],\C) \times H^1([0,1],\C) \; \big| \; f_2(0)=0=f_2(1) \big\}.
\eew
The Dirichlet domain $\mathcal{D}_{\mathrm D}$ of the corresponding ZS-system \eqref{main_eq} is then given by
\bew
\mathcal D_{\mathrm D} \coloneqq \big\{ g\in H^1([0,1],\C) \times H^1([0,1],\C) \; \big| \; (g_1-g_2)(0)=0=(g_1-g_2)(1)\big\},
\eew
as $\mathcal A_{\mathrm D}$ corresponds to $\mathcal D_{\mathrm D}$ under the transformation $T$, cf.~\eqref{def_T}.
For a given potential $\psi\in\LL$, we say that $\lambda\in\C$ lies in the \emph{Dirichlet spectrum} if there exists a 
$\phi\in\mathcal D_{\mathrm D}\setminus\{0\}$ which solves \eqref{main_eq}. 

\begin{theorem}
Fix $\psi\in\LL$. The Dirichlet spectrum of \eqref{main_eq} is the zero set of the entire function 
\be \label{char_func_dirich}
\chi_{\mathrm D}(\lambda,\psi) \coloneqq \frac{\m_4+\m_3-\m_2-\m_1}{2\I}\bigg|_{(\lambda,\psi)}.
\ee
In particular, $\chi_{\mathrm D}(\lambda,0) = \sin 2\lambda^2$.
\end{theorem}
\begin{proof}
Due to the definition of $\mathcal D_{\mathrm D}$, a complex number $\lambda$ lies in the Dirichlet spectrum of \eqref{main_eq} if and only if the fundamental solution $M$ maps the initial value $(1,1)$ to a collinear vector at $t=1$. That is,
if and only if $\m_1+\m_2 =\m_3+\m_4$.
\end{proof}

By Theorem \ref{thm_asymptotics_M_1} the characteristic function $\chi_{\mathrm D}$ satisfies 
\be \label{chi_asymptotics}
\chi_{\mathrm D}(\lambda,\psi) = \sin 2\lambda^2 + \mathcal O \big( |\lambda|^{-1}\, \e^{2|\Im (\lambda^2)|}\big)
\ee
uniformly on bounded sets in $\LL$ as $|\lambda|\to\infty$.
For $\psi \in \LL$, we set
\bew
\sigma_{\mathrm D}(\psi) \coloneqq \big\{\lambda \in \C \colon \chi_{\mathrm D}(\lambda,\psi)=0 \big\}.
\eew

We aim to localize the Dirichlet eigenvalues with the help of \eqref{chi_asymptotics}, see Lemma \ref{lem_count_dirichlet_ev} below. The proof makes use of the following elementary estimate, cf.~\cite[Appendix F]{GrebertKappeler14}: if $\lambda\in\C$ satisfies  $|\lambda-n \pi| \geq \pi/4$ for all integers $n$, then
$4 \, | \sin \lambda| > \e^{|\Im \lambda|}$.
Let us rephrase this inequality as follows.
\begin{lemma} \label{lem_est_sin_exp}
If $\lambda\in\C$ satisfies 
$|2 \lambda^2 - n \pi|\geq \pi/4$
for all integers $n$, then
\bew
4\, |\sin 2\lambda^2|>  \e^{2|\Im (\lambda^2)|}.
\eew
\end{lemma}

We denote the right, left, upper and lower open complex halfplane by
\begin{align*}
\C_+ \coloneqq \{z\in\C\colon \Re z >0\}, \quad \C_- \coloneqq \{z\in\C\colon \Re z <0\}, \\
\C^+ \coloneqq \{z\in\C\colon \Im z >0\}, \quad \C^- \coloneqq \{z\in\C\colon \Im z <0\}.
\end{align*}
Lemma \ref{lem_est_sin_exp} motivates the definition of the discs $D^i_n$, which are introduced below, to localize the Dirichlet eigenvalues.
For $|n|\geq1$, we consider the set
\bew
D_n \coloneqq 
\Big\{ \lambda \in \C \colon \big|2 \lambda^2 - n \pi \big| < \frac{\pi}{4}  \Big\}, 
\eew
which consists of two open discs, and define the disc $D^i_n$, $i=1,2$, by
\bew
D^1_n \coloneqq 
 \begin{cases}
     D_n \cap \C_+, &  n \geq 1, \\
     D_n \cap \C_-, &  n \leq -1, 
   \end{cases} \qquad
D^2_n \coloneqq 
 \begin{cases}
     D_n \cap \C^+, &  n \geq 1, \\
     D_n \cap \C^-, &  n \leq -1.
   \end{cases}
\eew  
For a given integer $N\geq1$ we define the disc $B_N$ by  
\bew
B_N \coloneqq \bigg\{ \lambda \in \C \colon  | \lambda | < \sqrt{\frac{(N+1/4) \pi}{2} } \, \bigg\}.
\eew
Furthermore we set $D_0 \coloneqq B_0 \coloneqq \{ \lambda \in \C \colon |\lambda| < \sqrt{\pi/8} \}$ and impose the convention $D^i_0 \coloneqq D_0$, $i=1,2$. Then for each $N\geq0$ the disc $B_N$ contains all the discs $D^i_n$ with $|n|\leq N$.
An illustration of the discs $B_N$ and $D^i_n$ can be found in Fig.~\ref{fig:counting_lemma} (see also Fig.~\ref{fig:singexp2}).

\begin{lemma}[Counting Lemma for Dirichlet eigenvalues] \label{lem_count_dirichlet_ev}
Let $B$ be a bounded subset of $\LL$.
There exists an integer $N\geq1$, such that for every $\psi\in B$, the entire function $\chi_{\mathrm D}(\cdot,\psi)$ has exactly one root in each of the two discs $D^i_n$, $i=1,2$, for $n\in\Z$ with $|n|>N$, and exactly $2(2N+1)$ roots in the disc $B_N$
when counted with multiplicity. There are no other roots.
\end{lemma}
\begin{proof}
Outside of the set
\be \label{set_Pi}
\Pi \coloneqq \bigcup_{ \substack{ n\in\Z \\ i\in\{1,2\} }} D^i_n    
\ee
it holds that 
$\frac{\e^{2|\Im (\lambda^2)|}}{|\sin 2\lambda^2|} < 4$
by the previous lemma. 
Therefore we obtain from \eqref{chi_asymptotics} that 
\bew
\chi_{\mathrm D}(\lambda,\psi) = \sin 2\lambda^2 + o\big(\e^{2|\Im (\lambda^2)|}\big) = \chi_{\mathrm D}(\lambda,0) \big(1+o(1)\big)
\eew
for $|\lambda|\to\infty$ with $\lambda\notin \Pi$ uniformly for $\psi\in B$.
More precisely, this means that there exists an integer $N\geq1$ such that, for all $\psi \in B$,
\bew
|\chi_{\mathrm D}(\lambda,\psi) - \chi_{\mathrm D}(\lambda,0)| < |\chi_{\mathrm D}(\lambda,0)|
\eew
on the boundaries of all discs $D^i_n$ with $|n|>N$, $i=1,2$, and also on the boundary of $B_N$. (Note that $|\chi_{\mathrm D}(\lambda,0)|>\delta$ on these boundaries for some $\delta>0$ which can be chosen independently of $|n|>N$.) Then Rouch\'e's theorem tells us that the analytic functions $\chi_{\mathrm D}(\cdot,\psi)$ possess the same number of roots inside these discs as 
$\chi_{\mathrm D}(\cdot,0)$. This proves the first statement, because 
$\chi_{\mathrm D}(\cdot,0)\colon \lambda \mapsto\sin 2\lambda^2$ has exactly one root in each $D^i_n$ for 
$|n|>N$, $i=1,2$, and $2(2N+1)$ roots in the disc $B_N$.

It is now obvious that there are no other roots, because the number of roots of $\chi_{\mathrm D}(\cdot,\psi)$ in each of the discs $B_{N+k}$, $k\geq1$, is exactly $2(N+k+1)$ due to the same argument as we used before. But these roots correspond to the $2(2N+1)$ roots of $\chi_{\mathrm D}(\cdot,\psi)$ inside $B_N\subseteq B_{N+k}$ \emph{plus} the $2k$ roots inside the discs $D_l \subseteq B_{N+k}$ with $N<|l|\leq N+k$
that we have already identified earlier in the proof. 
\end{proof}

\smallskip

Lemma \ref{lem_count_dirichlet_ev} allows us to introduce a systematic procedure for labeling the Dirichlet eigenvalues. For this purpose, we first consider the spectrum $\sigma_{\mathrm D}(0)$ of the zero potential, which consists of the two bi-infinite sequences
\be \label{dirichlet_ev_zero}
\mu^i_n(0) = \sgn (n) \sqrt{\frac{(-1)^{i-1}|n|\pi}{2}}, \quad i=1,2, \; n \in \Z,
\ee
where
\bew
\sgn (n) \coloneqq 
 \begin{cases}
     1 &  n > 0, \\
     0 &  n = 0, \\
     -1 &  n < 0. 
   \end{cases}
\eew
The eigenvalues $\mu^1_n(0)$, $n \in \Z$, are real while the eigenvalues $\mu^2_n(0)$, $n \in \Z\setminus\{0\}$, are purely imaginary. The Dirichlet eigenvalue $0$ has multiciplity two; all the other Dirichlet eigenvalues corresponding to $\psi=0$ are simple roots of $\chi_{\mathrm D}(\cdot,0)$. 

Let now $\psi\in\LL$ be an arbitrary potential. By Lemma~\ref{lem_count_dirichlet_ev} there exists a minimal integer $N\geq 0$ such that for all  $|n|>N$, each disc $D^i_n$, $i=1,2$, contains precisely one Dirichlet eigenvalue of multiplicity one---this eigenvalue will henceforth be denoted by $\mu^i_n \equiv \mu^i_n(\psi)$---and $B_N$ contains the remaining $2(2N+1)$ eigenvalues when counted with multiplicity. 
In order to label the $2(2N+1)$ roots of $\chi_{\mathrm D}(\cdot,\psi)$ in $B_N$, we proceed as follows. 
We make a (so far unordered) list of all the elements of $\sigma_{\mathrm D}(\psi) \cap B_N$. For any multiple root of $\chi_{\mathrm D}(\psi)$ in this list, we include multiple copies of it in the list according to its multiplicity. In this way, we make sure that the list has exactly $2(2N+1)$ entries (the set $\sigma_{\mathrm D}(\psi) \cap B_N$ contains strictly less than $2(2N+1)$ elements if $\chi_{\mathrm D}(\psi)$ has non-simple roots).  
We employ the lexicographical ordering of the complex numbers, i.e.~for $z_1,z_2\in \C$, 
\bew
z_1 \preceq z_2 \quad \Leftrightarrow \quad 
 \begin{cases}
    \Re z_1 < \Re z_2 \\
     \quad \text{or} \\
      \Re z_1 = \Re z_2 \quad \text{and} \quad \Im z_1 \leq \Im z_2,
   \end{cases}
\eew
to label the $2(2N+1)$ entries of list list of roots in $B_N$ in such a way that 
\begin{align*}
\mu^1_{-N} \preceq  \cdots \preceq 
\mu^1_{-1} \preceq \mu^2_{-N}   
\preceq  \cdots \preceq\mu^2_{-1}
 \preceq \mu^1_0   \preceq \mu^2_0  \preceq\mu^2_{1} \preceq
\cdots \preceq \mu^2_{N} \preceq \mu^1_1 \preceq \cdots \preceq \mu^1_{N}.
\end{align*}
The labeling of the roots of $\chi_{\mathrm D}(\cdot,\psi)$ according to this procedure is unique except that the label of a particular Dirichlet eigenvalue is ambiguous whenever it is not a simple root of $\chi_{\mathrm D}(\cdot,\psi)$. Sequences of Dirichlet eigenvalues of the form $(\mu^i_n)_{n\in\Z}$, where $i=1$ or $i=2$, are always well-defined, since each element of such a sequence has a uniquely defined value.

\begin{remark} \label{rem_mult_ord}
In general neither the multiplicity, nor the label of a Dirichlet eigenvalue $\mu^i_n(\psi)$ is preserved under continuous deformations of the potential $\psi$---not even locally around the zero potential, where discontinuities of the functions $\psi\mapsto\mu^i_0(\psi)$ may occur due to the lexicographical ordering. Merging and splitting of Dirichlet eigenvalues can occur under continuous deformations of the potential, cf.~Fig.~\ref{fig:singexp2} where such behavior is illustrated in the case of periodic eigenvalues. However, for sufficiently large $|n|$, the mappings $\psi \mapsto \mu^i_n(\psi)$, $i=1,2$, are continuous on bounded subsets of $\LL$; in fact, we will see in the proof of Theorem \ref{prop_mu} that these mappings are analytic. In particular, the eigenvalue $\mu^i_n(\psi)$ remains simple under continuous deformations within a bounded subset of $\LL$ for large enough $|n|$.
\end{remark}

\smallskip

To continue our analysis, we introduce the Banach spaces $\ell^{p,s}_\K$ of (bi-infinite) sequences
\be\label{lpsdef}
\ell^{p,s}_\K \coloneqq \Big\{ u= (u_n)_{n\in\Z}  \; \big| \; \big((1+n^2)^{\frac{s}{2}} u_n\big)_{n\in\Z} \in \ell^p_\K  \Big\},
\ee
where $1\leq p \leq \infty$, $s\in\R$, and $\K=\R$ or $\K=\C$, as well as the closed subspaces
\bew
\check \ell^{p,s}_\K \coloneqq \Big\{ u= (u_n)_{n\in\Z} \in \ell^{p,s}_\K  \colon  u_0=0  \Big\}.
\eew
We consider basic properties of these Banach spaces in Section \ref{sec_potRI}.
Theorem \ref{prop_mu} below establishes asymptotic estimates and continuity properties of the Dirichlet eigenvalues in the $\ell^{\infty,1}_\C$ setting. 
We use the notation $\ell^{p,s}_n$ to denote the $n$-th coordinate of a sequence $(\ell^{p,s}_n)_{n\in\Z}$ in $\ell^{p,s}_\C$, see~\cite{GrebertKappeler14, KappelerPoeschel03, Kappeler_etal_09}.

From Lemma \ref{lem_count_dirichlet_ev} we infer that, uniformly on bounded subsets of $\LL$, 
\be \label{est_mu_1}
\mu^i_n(\psi) = \mu^i_n(0) + \ell^{\infty,1/2}_n, \quad i=1,2, 
\ee
where $\mu^i_n(0)$ is given by \eqref{dirichlet_ev_zero}. More precisely, this means that the sequence $(\mu^i_n(\psi)- \mu^i_n(0))_{n\in\Z}$, $i=1,2$, lies in $\ell^{\infty,1/2}_\C$ for every $\psi\in\LL$ with a uniform bound in $\ell^{\infty,1/2}_\C$ for $\psi$ ranging within arbitrary but fixed bounded subsets of $\LL$.
To see this, we note that the radius of the disc $D^i_n$ centered at 
$\mu^i_n(0)=\sgn(n)\sqrt{(-1)^{i-1}|n|\pi/2}$, $i=1,2$, which contains the Dirichlet eigenvalue $\mu^i_n(\psi)$ for each $|n|>N$, is of order $\mathcal O (|n|^{-1/2})$ as $|n|\to\infty$; the integer $N\geq1$ can be chosen uniformly for all $\psi$ within a fixed bounded subset of $\LL$. 

The following theorem improves the estimate in \eqref{est_mu_1} considerably; furthermore it establishes an equicontinuity property of the set of Dirichlet eigenvalues  as functions of $\psi$.

\begin{theorem}
 \label{prop_mu}
Let $B$ be a bounded neighborhood of $\psi=0$ in $\LL$. 
\begin{enumerate}
\item
Uniformly on $B$, 
\be \label{est_mu_2} 
\mu^i_n(\psi) = \mu^i_n(0) + \ell^{\infty,1}_n, \quad i=1,2, 
\ee
where $\mu^i_n(0)$ is given by \eqref{dirichlet_ev_zero}. 
\item
There exists a  neighborhood $W\subseteq B$ of the zero potential such that for all $\psi\in W$, 
\be \label{prop_mu_2}
\mu^i_n(\psi)\in D^i_n \quad \text{for} \quad |n|\geq1, \quad  \mu^i_0(\psi)\in D_0 \qquad (i=1,2).
\ee
\item
Let $W\subseteq B$
be an open neighborhood  of the zero potential such that \eqref{prop_mu_2} is fulfilled for all $\psi\in W$.
Then the Dirichlet eigenvalues $\mu^i_n$, considered as functions of $\psi$, are analytic on $W$ for all $n\in\Z\setminus\{0\}$, $i=1,2$.
Furthermore, for every $\psi\in W$ and every sequence $(\psi_k)_{k\in\N}$ in $W$ with $\psi_k\to\psi$ as $k\to\infty$ it holds that
\be \label{prop_mu_3}
\lim_{k\to\infty} \bigg(\sup_{n\in\Z\setminus\{0\}} (1+n^2)^\frac{1}{2} \big|\mu^i_n(\psi_k) - \mu^i_n(\psi)\big| \bigg) = 0, \quad i=1,2.
\ee
\end{enumerate}
\end{theorem}
\begin{proof}
To prove the first assertion, we note that \eqref{est_mu_1} and Theorem \ref{thm_asymp_seq_near_per_ev} imply (cf.~the asymptotic estimate in \eqref{thm_asymp_seq_near_per_ev_1}) 
\bew
0= \chi_{\mathrm D} (\mu^i_n(\psi)) = \sin[2 (\mu^i_n(\psi))^2] + \ell^{\infty,1/2}_n, \quad i=1,2
\eew 
uniformly on $B$.
Therefore, the fundamental theorem of calculus implies that, uniformly on $B$,
\begin{align*}
\ell^{\infty,1/2}_n &= \sin[2 \mu^i_n(\psi)^2] - \sin[2 \mu^i_n(0)^2] \\ 
&= 2\big(\mu^i_n(\psi)^2 - \mu^i_n(0)^2\big) 
\int^1_0  \cos[t \, 2 \mu^i_n(\psi)^2 + (1-t) \, 2 \mu^i_n(0)^2] \, \mathrm d t, \quad  i=1,2.
\end{align*}
The integral in the above equation is uniformly bounded in $\ell^\infty_\C$ for all potentials in $B$, since the line segments connecting $\mu^i_n(\psi)^2$ with $\mu^i_n(0)^2$ are uniformly bounded in $\ell^\infty_\C$ due to \eqref{est_mu_1}. Thus,
\bew
\ell^{\infty,1/2}_n =  \mu^i_n(\psi)^2 - \mu^i_n(0)^2 
= \big(\mu^i_n(\psi) - \mu^i_n(0)\big) \big(\mu^i_n(\psi) + \mu^i_n(0)\big),  \quad  i=1,2
\eew
uniformly  on $B$. By employing \eqref{est_mu_1} once again, we infer that  
\bew
\ell^{\infty,1/2}_n = \mu^i_n(0) \big(\mu^i_n(\psi) - \mu^i_n(0)\big),  \quad  i=1,2
\eew
uniformly on $B$,
which is equivalent to the first assertion of the theorem.

To prove the second assertion, we recall that the characteristic function $\chi_{\mathrm D}$ is analytic\footnote{Cf.~Theorem \ref{thm_fund_sol}, where we proved analyticity of the fundamental matrix solution with respect to the potential; see also the short review of analytic maps between complex Banach spaces in Section~\ref{sec_potRI}.} (and hence also continuous) with respect to $\psi$, and that \eqref{prop_mu_2} clearly holds for $\psi=0$.
Hence, if $W\subseteq B$ is a sufficiently small neighborhood of the zero potential, then, for each $\psi\in W$, the disc $D^i_n$ contains precisely the simple Dirichlet eigenvalue $\mu^i_n(\psi)$ while $D_0$ contains either a double eigenvalue $\mu^1_0(\psi)=\mu^2_0(\psi)$ or two distinct simple eigenvalues $\mu^1_0(\psi)\neq\mu^2_0(\psi)$.

To prove the third assertion, we first show that the function $\mu^i_n\colon W\to \C$, $n\in\Z\setminus\{0\}$, $i=1,2$, is analytic and takes values in the disc $D^i_n$. Analyticity is inherited from $\chi_{\mathrm D}$ as a consequence of the implicit function theorem for analytic mappings between complex Banach spaces; see e.g.~\cite{Whittlesey65} for various generalizations of the classical implicit function theorem to infinite dimensional Banach spaces. Indeed, the restriction of the characteristic function for the Dirichlet eigenvalues to the domain $D^i_n\times W$,
\bew
\chi_{\mathrm D}\big|_{D^i_n\times W}\colon D^i_n\times W \to \C, \quad n\in\Z\setminus\{0\}, \; i=1,2,
\eew
is analytic. By assumption, for each $\psi'\in W$ there exists a unique $\lambda'\in D^i_n$ such that $\chi_{\mathrm D}\big|_{D^i_n\times W}(\lambda',\psi')=0$. 
Furthermore, denoting by $\frac{\partial}{\partial\lambda}$ the partial derivative of $\chi_{\mathrm D}\big|_{D^i_n\times W}$ with respect to its first variable $\lambda$, we claim that  
\be \label{pr_prop_mu_0}
\frac{\partial}{\partial\lambda} \chi_{\mathrm D}\Big|_{D^i_n\times W} (\lambda',\psi') \neq 0,
\quad n\in\Z\setminus\{0\}, \; i=1,2,
\ee
so that the partial derivative is a linear isomorphism $\C\to\C$. Indeed, by Theorem \ref{thm_asymp_seq_near_per_ev} (cf.~\eqref{thm_asymp_seq_near_per_ev_2}),
\be \label{pr_prop_mu_01}
\frac{\partial }{\partial\lambda}\chi_{\mathrm D}\Big|_{\dot\Pi\times B}(\lambda,\psi) 
= 4\lambda \cos 2\lambda^2 + \mathcal O(1), \quad 
\dot\Pi \coloneqq \bigcup_{\substack{n\in\Z\setminus\{0\} \\ i=1,2}}D^i_n
\ee
uniformly on $B$ as $|\lambda|\to\infty$. 
This implies that 
\be \label{pr_prop_mu_1}
\frac{\partial}{\partial\lambda} \chi_{\mathrm D}\neq 0
\quad \text{on} \quad
 \bigcup_{\substack{|n|>N \\ i=1,2}} D^i_n   \times B
\ee
for some large enough $N\geq0$.
By continuity of $\frac{\partial \chi_{\mathrm D}}{\partial\lambda}$ with respect to $\psi$ and the fact that the formula for $\frac{\partial \chi_{\mathrm D}}{\partial\lambda}$ in \eqref{pr_prop_mu_01} holds without the error term when $\psi=0$, it is clear that \eqref{pr_prop_mu_1} holds for $N=0$ with $B$ replaced by $W$. 
This proves \eqref{pr_prop_mu_0}. 

In view of \eqref{pr_prop_mu_0}, the implicit function theorem guarantees the existence of a unique (global) analytic function $\tilde \mu^i_n\colon W\to \C$ such that, for all $(\lambda,\psi)\in D^i_n\times W$,
\bew
\chi_{\mathrm D}\big|_{D^i_n\times W} (\lambda,\psi) =0 \iff \lambda=\tilde \mu^i_n(\psi), 
\quad n\in\Z\setminus\{0\}, \; i=1,2.
\eew
Since $\tilde \mu^i_n= \mu^i_n$, this shows that $\mu^i_n\colon W\to\C$ is analytic for $n\in\Z\setminus\{0\}$, $i=1,2$.

For $n \in \Z$ and $i=1,2$, we consider the analytic mappings
\bew
\beta^i_n \colon W\to\C, \quad \psi\mapsto\beta^i_n(\psi)  \coloneqq
\begin{cases}
 (1+n^2)^\frac{1}{2}\big( \mu^i_n(\psi)-\mu^i_n(0)\big) &n\in\Z\setminus\{0\} \\
0 &n=0.
\end{cases}
\eew
The first assertion of the theorem implies that the family $\{\beta^i_n\}^{i=1,2}_{n\in\Z}$ is uniformly bounded in $\C$. Since all the functions of this family are analytic, it follows that  $\{\beta^i_n\}^{i=1,2}_{n\in\Z}$ is uniformly equicontinuous on $B$, cf.~\cite[Proposition 9.15]{Mujica86}. 
That is, for each $\eps>0$ there exists a $\delta>0$ such that, for all $n\in\Z$, $i\in\{1,2\}$, and all $\psi,\psi' \in B$,
\bew
\|\psi-\psi'\| < \delta \quad \Rightarrow \quad \eps > |\beta^i_n(\psi)-\beta^i_n(\psi')| 
= \big|(1+n^2)^\frac{1}{2}\big( \mu^i_n(\psi)-\mu^i_n(\psi')\big)\big|.
\eew 
This  implies  that the two mappings 
\bew
W\to \check \ell^{\infty,1/2}_\C, \quad \psi \mapsto
\begin{cases}
\mu^i_n(\psi) &n\in\Z\setminus\{0\} \\
0 & n=0,
\end{cases}
\qquad i=1,2,
\eew
are continuous, which proves the third assertion.
\end{proof}

\begin{remark}
A more abstract proof of \eqref{prop_mu_3} proceeds as follows. By the general version of Montel's theorem for analytic functions on separable complex Banach spaces, see e.g.~\cite[Proposition 9.16]{Mujica86}, the family $\{\beta^i_n\}^{i=1,2}_{n\in\Z}$ in the proof of Theorem~\ref{prop_mu} is \emph{normal} in the locally convex topological vector space $\mathcal H(W)$, the space  of all analytic functions from $W$ to $\C$, endowed with the topology $\tau_c$ of uniform convergence on compact subsets of $W$. That is, each sequence of elements of $\{\beta^i_n\}^{i=1,2}_{n\in\Z}$ has a subsequence which converges in $(\mathcal H(W), \tau_c)$.  This allows us to obtain \eqref{prop_mu_3} by interchanging the order of taking the limit and supremum as follows:
\bew
\lim_{k\to\infty} \bigg(\sup_{n\in\Z} \big|\beta^i_n(\psi_k) - \beta^i_n(\psi)\big| \bigg) =
\sup_{n\in\Z} \bigg( \lim_{k\to\infty}  \big|\beta^i_n(\psi_k) - \beta^i_n(\psi)\big|  \bigg).
\eew
\end{remark}

We define the Neumann domain $\mathcal A_{\mathrm N}$ of the AKNS-system \eqref{AKNS_system} by
\bew
\mathcal A_{\mathrm N} \coloneqq \big\{ f\in H^1([0,1],\C) \times H^1([0,1],\C) \; \big| \; f_1(0)=0=f_1(1)   \big\}.
\eew
The Neumann domain $\mathcal D_{\mathrm N} $ of the corresponding ZS-system \eqref{main_eq} is then given by
\bew
\mathcal D_{\mathrm N} \coloneqq \big\{ g\in H^1([0,1],\C) \times H^1([0,1],\C) \; \big| \; (h_1+h_2)(0)=0=(h_1+h_2)(1)   \big\},
\eew
as $\mathcal A_{\mathrm N}$ corresponds to $\mathcal D_{\mathrm N}$ under the transformation $T$.
For a given potential $\psi\in\LL$, we say that $\lambda\in\C$ lies in the \emph{Neumann spectrum} if there exists a 
$\phi\in\mathcal D_{\mathrm N}\setminus\{0\}$ which solves \eqref{main_eq}. 

\begin{theorem}
Fix $\psi\in\LL$. The Neumann spectrum related to \eqref{main_eq} is the zero set of the entire function 
\bew 
\chi_{\mathrm N}(\lambda,\psi) \coloneqq \frac{\m_4-\m_3 + \m_2 - \m_1}{2\I}\bigg|_{(\lambda,\psi)}.
\eew
In particular, $\chi_{\mathrm N}(\lambda,0) = \chi_{\mathrm D}(\lambda,0) = \sin 2\lambda^2$.
\end{theorem}
\begin{proof}
Due to the definition of $\mathcal D_{\mathrm N}$, a complex number $\lambda$ lies in the Neumann spectrum of system \eqref{main_eq} if and only if the fundamental solution $M$ maps the initial value $(1,-1)$ to a collinear vector at $t=1$. That is,
if and only if $\m_1-\m_2 =-\m_3+\m_4$.
\end{proof}

By Theorem \ref{thm_asymptotics_M_1}, the characteristic function $\chi_{\mathrm N}$ satisfies 
\bew
\chi_{\mathrm N}(\lambda,\psi) = \sin 2\lambda^2 + \mathcal O \big( |\lambda|^{-1}\, \e^{2|\Im (\lambda^2)|}\big)
\eew
uniformly on bounded  subsets of $\LL$ as $|\lambda|\to\infty$.
For $\psi \in \LL$ we set
\bew
\sigma_{\mathrm N}(\psi) \coloneqq \big\{\lambda \in \C \colon \chi_{\mathrm N}(\lambda,\psi)=0 \big\}.
\eew

As for the Dirichlet case, we obtain the following asymptotic localization for the elements of the Neumann spectrum.

\begin{lemma}[Counting Lemma for Neumann eigenvalues] \label{lem_count_neum_ev}
Let $B$ be a bounded subset of $\LL$.
There exists an integer $N\geq1$, such that for every $\psi\in B$, the entire function $\chi_{\mathrm N}(\cdot,\psi)$ has exactly one root in each of the two discs $D^i_n$, $i=1,2$, for $n\in\Z$ with $|n|>N$, and exactly $2(2N+1)$ roots in the disc $B_N$
when counted with multiplicity. There are no other roots.
\end{lemma}

We label the Neumann eigenvalues in the same way as the Dirichlet eigenvalues.
The Neumann spectrum of the zero potential $\psi=0$ coincides with the corresponding Dirichlet spectrum:
\be \label{neumann_ev_zero}
\nu^i_n(0) = \mu^i_n(0) = \sgn (n) \sqrt{\frac{(-1)^{i-1}|n|\pi}{2}}, \quad i=1,2. 
\ee
Remark \ref{rem_mult_ord} applies also to the Neumann eigenvalues. The analog of Theorem \ref{prop_mu} for Neumann eigenvalues reads as follows.

\begin{theorem} \label{prop_nu}
Let $B$ be a bounded neighborhood of $\psi=0$ in $\LL$. 
\begin{enumerate}
\item
Uniformly on $B$, 
\bew 
\nu^i_n(\psi) = \nu^i_n(0) + \ell^{\infty,1}_n, \quad i=1,2, 
\eew
where $\nu^i_n(0)$ is given by \eqref{neumann_ev_zero}. 
\item
There exists a  neighborhood $W\subseteq B$ of the zero potential such that for all $\psi\in W$, 
\be \label{prop_nu_2}
\nu^i_n(\psi)\in D^i_n \quad \text{for} \quad |n|\geq1, \quad  \nu^i_0(\psi)\in D_0 \qquad (i=1,2).
\ee
\item
Let $W\subseteq B$ be an open neighborhood  of the zero potential such that \eqref{prop_nu_2} is fulfilled for all $\psi\in W$.
Then the Neumann eigenvalues $\nu^i_n$ are analytic on $W$ for all $n\in\Z\setminus\{0\}$, $i=1,2$.
Furthermore, for every $\psi\in W$ and every sequence $(\psi_k)_{k\in\N}$ in $W$ with $\psi_k\to\psi$ as $k\to\infty$ it holds that
\bew
\lim_{k\to\infty} \bigg(\sup_{n\in\Z\setminus\{0\}} (1+n^2)^\frac{1}{2} \big|\nu^i_n(\psi_k) - \nu^i_n(\psi)\big| \bigg) = 0, \quad i=1,2.
\eew
\end{enumerate}
\end{theorem}

\subsection{Periodic spectrum}

The trace of the fundamental matrix solution $M(t,\lambda,\psi)$ at time $t=1$ is called the \emph{discriminant} and is denoted by $\Delta$:
\bew
\Delta \equiv \Delta(\lambda,\psi) \coloneqq \mathrm{tr} \, \M =  \m_1 + \m_4.
\eew
The sum of the off-diagonal entries of $\M$ is referred to as the \emph{anti-discriminant}:
\bew
\delta \equiv \delta(\lambda,\psi) \coloneqq \m_2+\m_3.
\eew

\begin{theorem}
The discriminant $\Delta$, the anti-discriminant $\delta$ and their respective $\lambda$-derivatives $\dot\Delta$ and $\dot\delta$ are compact analytic functions on $\C\times\LL$.
At the zero potential, 
\be  \label{Delta_zero}
\Delta(\lambda,0)=2 \cos 2\lambda^2, \quad \lambda\in\C.
\ee
\end{theorem}
\begin{proof}
Both the discriminant and the anti-discriminant are analytic due to Theorem \ref{thm_fund_sol} and compactness follows from Proposition \ref{Prop_M_comp}. From Corollary \ref{Cor_lambda_deriv_M} we infer that the $\lambda$-derivatives $\dot\Delta$ and $\dot\delta$ inherit both properties. 
\end{proof}

The \emph{periodic domain} $\mathcal D_{\mathrm P}$ of the ZS-system \eqref{main_eq} is defined by\footnote{Note that $\mathcal D_{\mathrm P}$ consists of both periodic and antiperiodic functions.} 
\bew
\mathcal D_{\mathrm P} \coloneqq \big\{ f\in H^1([0,1],\C) \times H^1([0,1],\C) \; \big| \; f(1)=f(0) 
\; \text{or} \;  f(1)=-f(0)  \big\} 
\eew
A complex number $\lambda$ is called a \emph{periodic eigenvalue} if \eqref{main_eq} is satisfied for some $\phi\in\mathcal D_{\mathrm P}\setminus\{0\}$.

\begin{theorem} \label{thm_char_func_per_ev}
Let $\psi\in\LL$. A complex number $\lambda$ is a periodic eigenvalue if and only if it is a zero of the entire function 
\be \label{char_func_per}
\chi_{\mathrm P}(\lambda,\psi) \coloneqq \Delta^2 (\lambda,\psi) -4.
\ee
\end{theorem}
\begin{proof}
Let us fix $\psi\in\LL$. Since $M$ is the fundamental solution of \eqref{main_eq}, a complex number $\lambda$ is a periodic eigenvalue if and only if there exists a nonzero element $f\in \mathcal D_{\mathrm P}$ with
\bew
f(1) = M(1,\lambda) f(0) = \pm f(0),
\eew
hence if and only if $1$ or $-1$ is an eigenvalue of $M(1,\lambda)$. As $\det M(1,\lambda) =1$ by Proposition \ref{Wronskian_identity}, the two eigenvalues of $M(1,\lambda)$ are either both equal to $1$ or both equal to $-1$. Therefore we either have $\Delta(\lambda)=2$ or $\Delta(\lambda)=-2$, that is, $\chi_{\mathrm P}(\lambda)=0$.
\end{proof}

For $\psi \in \LL$, we set
\bew
\sigma_{\mathrm P}(\psi) \coloneqq \big\{\lambda \in \C \colon \chi_{\mathrm P}(\lambda,\psi)=0 \big\}.
\eew
The characteristic function for the zero potential $\psi=0$ is given by
\bew
\chi_{\mathrm P}(\lambda,0) = -4 \sin^2 2\lambda^2;
\eew
each root has multiplicity two, except the root $\lambda=0$ has multiplicity four. Thus the periodic spectrum of the zero potential consists of two bi-infinite sequences of double eigenvalues 
\be \label{per_ev_zero_pot}
\lambda^{i,\pm}_n(0) = \sgn (n) \sqrt{\frac{(-1)^{i-1}|n|\pi}{2}}, \quad i=1,2 
\ee
on the real and imaginary axes in the complex plane.
The $\lambda$-derivative of the discriminant at the zero potential is given by
\bew
\dot \Delta(\lambda,0) = -8 \lambda \sin 2\lambda^2,
\eew
and its roots, the so-called \emph{critical points} of the discriminant for the zero potential, denoted by $\dot \lambda^i_n(0)$, $i=1,2$, $n\in \Z$, coincide with the periodic eigenvalues (and the Dirichlet and Neumann eigenvalues):
\be \label{crit_val_Delta_zero_pot}
\dot \lambda^i_n(0) = \sgn (n) \sqrt{\frac{(-1)^{i-1}|n|\pi}{2}}, \quad i=1,2. 
\ee
Note that $\lambda=0$ has multiplicity three; all the other roots of $\dot\Delta(\cdot,0)$ are simple roots.

\begin{lemma} \label{lem_for_count_lem}
Fix $\psi\in\LL$. As $|\lambda|\to\infty$ with $\lambda \notin \Pi = \bigcup_{n\in\Z,\, i=1,2}D^i_n$,
\begin{align}
\chi_{\mathrm P}(\lambda,\psi) 
&= (-4\sin^2 2\lambda^2) \big(1+\mathcal O(|\lambda|^{-1})\big), \label{lem_for_count_lem_1} \\
 \dot \Delta(\lambda,\psi) &= (-8\lambda \sin  2\lambda^2) \big(1+\mathcal O(|\lambda|^{-1})\big). 
 \label{lem_for_count_lem_2}
\end{align}
These asymptotic estimates hold uniformly on bounded subsets of $\LL$. For the zero potential these formulas hold without the error terms.
\end{lemma}
\begin{proof}
By Theorem \ref{thm_asymptotics_M_1}, we have $\Delta(\lambda,\psi)= 2 \cos 2\lambda^2 + \mathcal O(|\lambda|^{-1}\e^{2|\Im (\lambda^2)|})$  uniformly on bounded subsets of $\LL$, and thus 
\bew
\chi_{\mathrm P}(\lambda,\psi)  = 
(-4 \sin^2 2\lambda^2) 
\bigg[ 1 + \frac{ \mathcal O\big(|\lambda|^{-1}  \, \e^{2|\Im (\lambda^2)|}\big) \cos 2\lambda^2 }{\sin^2 2\lambda^2} 
+  \frac{ \mathcal O\big(|\lambda|^{-2} \, \e^{4|\Im (\lambda^2)|}\big) }{ \sin^2 2\lambda^2}  \bigg].
\eew
For $\lambda \notin \Pi$, we have $4 \, |\sin 2\lambda^2| >  \e^{2|\Im (\lambda^2)|}$, cf.~Lemma~\ref{lem_est_sin_exp}, and therefore
\be \label{pr_lem_for_count_lem_1}
\bigg| \frac{\cos 2\lambda^2}{\sin 2\lambda^2} \bigg|  \leq \frac{ \e^{2|\Im (\lambda^2)|}}{|\sin 2\lambda^2|} < 4
\quad  \text{for} \quad \lambda \in \C \setminus \Pi.
\ee 
The estimate \eqref{lem_for_count_lem_1} follows. 
Moreover, by Theorem \ref{thm_asymptotics_M_1},  
\bew
\dot \Delta(\lambda,\psi) = (-8\lambda \sin  2\lambda^2)
\bigg[ 1 + \frac{\mathcal O\big(\e^{2|\Im (\lambda^2)|}\big)}{\lambda \sin  2\lambda^2} \bigg]
\eew
uniformly on bounded subsets of $\LL$
and thus \eqref{pr_lem_for_count_lem_1} yields \eqref{lem_for_count_lem_2}.
\end{proof}

The next result provides an asymptotic localization of the periodic eigenvalues.

\begin{lemma}[Counting Lemma for periodic eigenvalues] \label{counting_lemma}
Let $B$ be a bounded subset of $\LL$.
There exists an integer $N\geq1$, such that for every $\psi\in B$, the entire function $\chi_{\mathrm P}(\cdot,\psi)$ has exactly two roots in each of the two discs $D^i_n$, $i=1,2$, and exactly $4(2N + 1)$ roots in the disc $B_N$, when counted with multiplicity. There are no other roots.
\end{lemma}
\begin{proof}
Let $B\subseteq \LL$ be bounded. By Lemma \ref{lem_for_count_lem}, 
\bew
\chi_{\mathrm P}(\lambda,\psi) = \chi_{\mathrm P}(\lambda,0) \big(1+o(1)\big)
\eew
for $|\lambda| \to \infty$ with $\lambda \notin \Pi$ uniformly for $\psi\in B$. Hence there exists an integer $N\geq 1$ such that, for all $\psi \in B$,
\bew
|\chi_{\mathrm P}(\lambda,\psi) - \chi_{\mathrm P}(\lambda,0)| < | \chi_{\mathrm P}(\lambda,0)|
\eew
on the boundaries of all discs $D^i_n$ with $|n|>N$, $i=1,2$, and also on the boundary of $B_N$.
As in the proof of Lemma \ref{lem_count_dirichlet_ev}, the result follows by an application of Rouch\'e's theorem.
\end{proof}

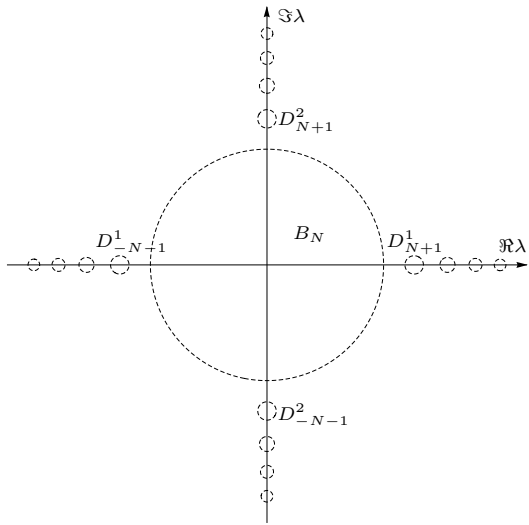
\begin{figure}[h!] 
\begin{center}
\begin{overpic}[width=.43\textwidth]{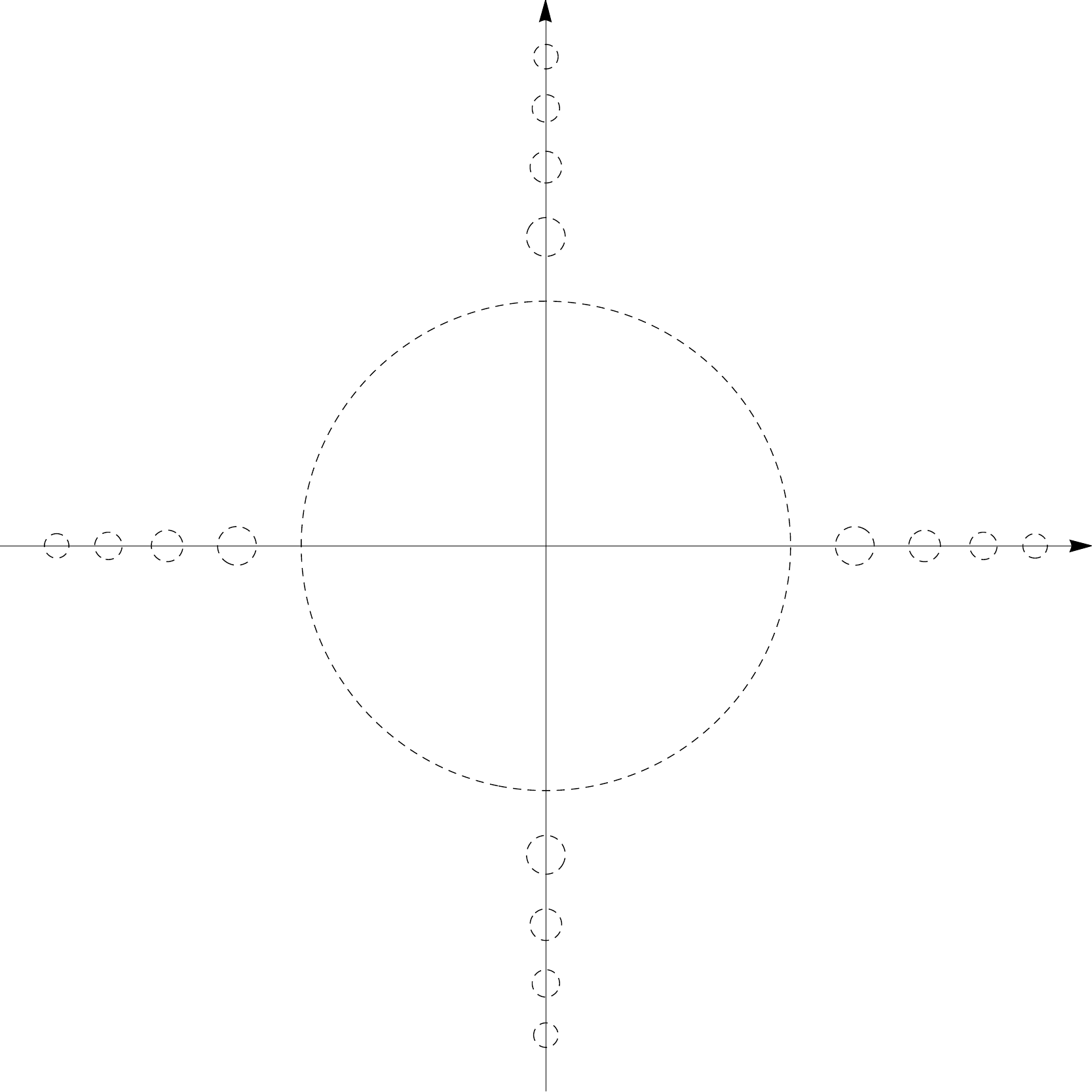}
     \put(55,55){{\tiny $B_N$}}
     \put(73,53.5){{\tiny $D^1_{N+1}$}}
     \put(17,53.5){{\tiny $D^1_{-N-1}$}}
     \put(52,77){{\tiny $D^2_{N+1}$}}
     \put(52,20){{\tiny $D^2_{-N-1}$}}
     \put(94.5,52.5){{\tiny $\Re \lambda$}}
     \put(52,97){{\tiny $\Im \lambda$}}
\end{overpic}
\caption{Localization of the periodic eigenvalues according to the Counting Lemma. The first $4(2N+1)$ roots of $\chi_{\mathrm P}(\cdot,\psi)$ lie in the large disc $B_N$. The remaining periodic eigenvalues lie in the discs $D^i_n$ centered at $\lambda^{\pm,i}_n(0)$, $i=1,2$, $|n|>N$; each disc contains precisely two of them. The radii of these discs shrink to zero at order $\mathcal O(|n|^{-1/2})$ as $|n|\to\infty$.}
\label{fig:counting_lemma}
\end{center}
\end{figure}

Lemma \ref{lem_for_count_lem} yields a Counting Lemma for the critical points of $\Delta$ as well: 

\begin{lemma}[Counting Lemma for critical points] \label{counting_lemma_roots_Delta}
Let $B$ be a bounded subset of $\LL$.
There exists an integer $N\geq1$, such that for every $\psi\in B$, the entire function $\dot \Delta(\cdot,\psi)$ has exactly one root in each of the two discs $D^i_n$, $i=1,2$
for all $|n|>N$, and exactly $4N + 3$ roots in the disc $B_N$, when counted with multiplicity. There are no other roots.
\end{lemma}

Let $\psi\in\LL$ be an arbitrary potential.
Inspired by \eqref{per_ev_zero_pot} and \eqref{crit_val_Delta_zero_pot}, we denote the corresponding periodic eigenvalues and critical points by $\lambda^{i,\pm}_n \equiv \lambda^{i,\pm}_n(\psi)$ and $\dot \lambda^i_n \equiv  \dot \lambda^i_n(\psi)$ respectively, $n\in\Z$, $i=1,2$. 

The critical points are labeled in the same way as the Dirichlet and Neumann eigenvalues (except that there is one additional root close the origin, which we will ignore whenever we consider sequences of critical points of the form $(\dot\lambda^i_n)_{n\in\Z}$ with $i=1$ or $i=2$). Remark \ref{rem_mult_ord} applies also to the critical points.

Concerning periodic eigenvalues, we adapt our labeling procedure as follows. 
Let $N\geq 0$ be be the minimal integer such that (a) for each $|n|>N$ and each $i=1,2$, the disc $D^i_n$ contains either two simple periodic eigenvalues or one periodic double eigenvalue and (b) $B_N$ contains precisely $4(2N+1)$ roots of $\chi_{\mathrm P}(\cdot,\psi)$ when counted with multiplicity. 
The two eigenvalues in the disc $D^i_n$, $|n|>N$, $i=1,2$, will be denoted by $\lambda^{i,\pm}_n$ and ordered so that $\lambda^{i,-}_n  \preceq \lambda^{i,+}_n$.
The remaining $4(2N+1)$ roots in $B_N$ are labeled such that 
\begin{align*}
\lambda^{1,-}_{-N} &\preceq \lambda^{1,+}_{-N} \preceq \cdots \preceq 
\lambda^{1,-}_{-1} \preceq \lambda^{1,+}_{-1}\preceq \lambda^{2,-}_{-N}   \preceq \lambda^{2,+}_{-N}
\preceq \cdots  
\preceq\lambda^{2,-}_{-1} \preceq \lambda^{2,+}_{-1} \preceq\lambda^{1,-}_0 \preceq \lambda^{1,+}_0
 \\
& \preceq \lambda^{2,-}_0 \preceq \lambda^{2,+}_0 
\preceq \lambda^{2,-}_1  \preceq \lambda^{2,+}_1 \preceq
\cdots \preceq \lambda^{2,-}_{N}  \preceq \lambda^{2,+}_{N} 
\preceq \lambda^{1,-}_1 \preceq \lambda^{1,+}_1 \preceq \cdots 
\preceq \lambda^{1,-}_{N} \preceq \lambda^{1,+}_{N}.
\end{align*}
The labeling of the roots of $\chi_{\mathrm P}(\cdot,\psi)$ according to the above procedure is unique except that the label of a particular periodic eigenvalue is ambiguous whenever it is not a simple root of $\chi_{\mathrm P}(\cdot,\psi)$. Sequences of Dirichlet eigenvalues of the form $(\mu^i_n)_{n\in\Z}$, where $i=1$ or $i=2$, are always well-defined, since each element of such a sequence has a uniquely defined value.  
If $N=0$ happens to be the minimal integer, we agree on the convention that merely 
\bew
\lambda^{1,-}_0  \preceq \lambda^{1,+}_0 \quad  \text{and}  \quad \lambda^{2,-}_0 \preceq \lambda^{2,-}_0
\eew
is required rather than $\lambda^{1,-}_0  \preceq \lambda^{1,+}_0 \preceq \lambda^{2,-}_0 \preceq \lambda^{2,-}_0$. 
This convention provides the freedom to label periodic eigenvalues inside the disc $B_0=D_0$ in the intuitive way (compare e.g.~the labelings of Fig.~\ref{fig:singexp1A} and Fig.~\ref{fig:singexp1B}, where we labeled the two periodic \emph{double} eigenvalues closest to the origin either by $\lambda^{2,+}_0$ and $\lambda^{2,-}_0$ (Fig.~\ref{fig:singexp1A}), or by $\lambda^{1,+}_0$ and $\lambda^{1,-}_0$ (Fig.~\ref{fig:singexp1B}), since they lie close to the imaginary axis in Fig.~\ref{fig:singexp1A}, and on the real axis in Fig.~\ref{fig:singexp1B}).

As in the case of Dirichlet eigenvalues, Neumann eigenvalues, and critical points, no general statement can be made regarding the multiplicity of periodic eigenvalues located near the origin; Fig.~\ref{fig:singexp2} illustrates this fact by means of an explicit example. 
Likewise, the labeling of periodic eigenvalues is generally not preserved under continuous deformations of the potential.  
Unlike the situation for Dirichlet eigenvalues, Neumann eigenvalues, and critical points, both multiplicity and labeling of periodic eigenvalues are generally not preserved under continuous deformations asymptotically for large $|n|$. Since each $D^i_n$ contains two periodic eigenvalues, their lexicographical ordering may entail discontinuities.

\begin{remark} \label{rem_sign_Delta_per_ev}
The Counting Lemma allows us to determine the sign of the discriminant at periodic eigenvalues with sufficiently large index $|n|$. Indeed, recall that $|\Delta(\lambda,\psi)|=2$ when $\lambda$ is a periodic eigenvalue, cf.~Theorem \ref{thm_char_func_per_ev}. Fix $\psi\in\LL$ and choose $N\geq 1$
according to the Counting Lemma so that each of the two discs $D^i_n$, $i=1,2$, for $n\in \Z$ with $|n|>N$  contains exactly two periodic eigenvalues.  
In fact, we can without loss of generality assume that $D^i_n$ contains exactly two periodic eigenvalues $\lambda^{i,\pm}_n(s \psi)$ for each potential $s \psi$ belonging to the line segment 
$S\coloneqq\{s \psi \colon 0\leq s\leq 1 \}$. Since $S \subseteq \LL$ is compact, we can choose $N$ uniformly with respect to $S$. Let us now consider the continuous paths $\rho^{i,\pm}_n \colon [0,1] \to \C$, $s\mapsto \lambda^{i,\pm}_n(s \psi)$, $i=1,2$.
Since $\Delta$ is continuous and $\Delta(\lambda^{i,\pm}_n(s \psi),s \psi) \in \{-2,2\}$ for $s \in [0,1]$, we conclude that either
\begin{align*}
\Delta(\rho^{i,\pm}(s), s \psi) \equiv 2 \ \text{on} \ [0,1] 
\qquad &\text{or} \qquad \Delta(\rho^{i,\pm}(s), s \psi) \equiv -2 \ \text{on} \ [0,1].
\end{align*}
Thus
\bew
\Delta(\lambda^{i,\pm}_n(\psi),\psi)  = \Delta(\lambda^{i,\pm}_n(0),0) 
= 2\, \cos n \pi = 2(-1)^n \quad \text{for} \quad |n|>N, \; i=1,2.
\eew
\end{remark}

The next lemma establishes a relation between the discriminant and the anti-discriminant evaluated at Dirichlet or Neumann eigenvalues.
 
\begin{lemma} \label{lem_id_discr_anti-discr}
If $\psi\in\LL$ and $\mu^i_n \equiv \mu^i_n(\psi)$ is any Dirichlet eigenvalue of $\psi$, then
\bew
\Delta^2(\mu^i_n,\psi) - 4 = \delta^2(\mu^i_n,\psi).
\eew 
This identity holds also at any Neumann eigenvalue $\nu^i_n \equiv \nu^i_n(\psi)$ of $\psi$.
\end{lemma}
\begin{proof}
We recall that $\m_1 \m_4 - \m_2 \m_3=1$ by Proposition \ref{Wronskian_identity}. Therefore,
\begin{align*}
\Delta^2 - 4 &= (\m_1 + \m_4)^2 - 4 \\
&= (\m_1 + \m_4)^2 - 4(\m_1 \m_4 - \m_2 \m_3) \\
&= (\m_1 - \m_4)^2 + 4 \m_2 \m_3.
\end{align*}
Let $\mu^i_n$ be a Dirichlet eigenvalue of $\psi\in\LL$, $i=1,2$. Then $\mu^i_n$ is a root of $\m_4 + \m_3 - \m_2 - \m_1$, that is
\bew
(\m_1 - \m_4)\big|_{(\mu^i_n,\psi)} = (\m_3 - \m_2)\big|_{(\mu^i_n,\psi)}.
\eew
Therefore,
\bew
\Delta^2(\mu^i_n,\psi) - 4 = \big(\m_2(\mu^i_n,\psi) + \m_3 (\mu^i_n,\psi)\big)^2 = \delta^2(\mu^i_n,\psi).
\eew 

For Neumann eigenvalues $\nu^i_n(\psi)$, $i=1,2$, we have 
\bew
(\m_1- \m_4)\big|_{(\nu^i_n,\psi)} = (\m_2- \m_3)\big|_{(\nu^i_n,\psi)},
\eew
which again yields the desired identity.
\end{proof}

By employing the identity of Lemma \ref{lem_id_discr_anti-discr}, we can prove the following analog of Theorems \ref{prop_mu} and \ref{prop_nu} for the periodic eigenvalues and the critical points.

\begin{theorem}  \label{prop_la}
Let $B$ be a bounded neighborhood of $\psi=0$ in $\LL$. 
\begin{enumerate}
\item
Uniformly on $B$, 
\be \label{prop_la_1}
\lambda^{i,\pm}_n(\psi) =\lambda^{i,\pm}_n(0) + \ell^{\infty,1}_n \quad \text{and} \quad \,
\dot \lambda^i_n(\psi) =\dot \lambda^i_n(0) + \ell^{\infty,1}_n,
 \quad i=1,2, 
\ee
where $\lambda^{i,\pm}_n(0)=\dot \lambda^i_n(0)$ are given by \eqref{per_ev_zero_pot}
and \eqref{crit_val_Delta_zero_pot}.
\item \label{prop_la_2}
There exists a  neighborhood $W\subseteq B$ of the zero potential such that for all $\psi\in W$ and every $n\in\Z$: 
\begin{enumerate}
\item 
$\sigma_{\mathrm P}(\psi) \cap D^i_n  \label{prop_la_2a}
= \{ \lambda^{i,-}_n (\psi),\lambda^{i,+}_n (\psi) \}$, $i=1,2$; 
\item 
$\Delta(\lambda^{i,\pm}_n (\psi),\psi) = 2(-1)^n$, $i=1,2$;
\item 
$\{\lambda \in \C \colon \dot\Delta(\lambda,\psi)=0 \} \cap D^i_n 
= \{ \dot\lambda^i_n (\psi) \}$, $i=1,2$.  \label{prop_la_2c}
\end{enumerate}
\item \label{prop_la_3}
Let $W\subseteq B$ be an open neighborhood  of the zero potential such that  part \emph{(c)} of \eqref{prop_la_2} is fulfilled.
Then the critical  points $\dot\lambda^i_n$, considered as functions of $\psi$, are analytic on $W$ for all $n\in\Z\setminus\{0\}$, $i=1,2$.
Furthermore, for every $\psi\in W$ and every sequence $(\psi_k)_{k\in\N}$ in $W$ with $\psi_k\to\psi$ as $k\to\infty$ it holds that
\bew
\lim_{k\to\infty} \bigg(\sup_{n\in\Z\setminus\{0\}} (1+n^2)^\frac{1}{2} \big|\dot\lambda^i_n(\psi_k) - \dot\lambda^i_n(\psi)\big| \bigg) = 0, \quad i=1,2.
\eew
\end{enumerate}
\end{theorem}
\begin{proof}
The proofs of the assertions for the critical points are similar to the proofs of the analogous assertions for Dirichlet eigenvalues; see~Theorem~\ref{prop_mu} and its proof. Furthermore, the second assertion follows---in view of the Counting Lemma and Remark \ref{rem_sign_Delta_per_ev}---from the continuity and asymptotics of $\chi_{\mathrm P}$.  
It only remains to show \eqref{prop_la_1} for the periodic eigenvalues $\lambda^{i,\pm}_n\equiv\lambda^{i,\pm}_n(\psi)$.

Since $\mu^i_n= \mu^i_n(0) + \ell^{\infty,1}$ uniformly on $B$ by Theorem \ref{prop_mu} and $E_{\mu^i_n}$ is off-diagonal, we can apply Theorem \ref{thm_asymp_seq_near_per_ev} to infer that $\delta(\mu^i_n) = \ell^{\infty,1/2}_n$ uniformly on $B$. 
Since the quadratic mapping $u_n\mapsto u^2_n$ is continuous as a map $\ell^{\infty,1/2}_\C \to \ell^{\infty,1}_\C$, Lemma \ref{lem_id_discr_anti-discr} yields  
\be \label{pr_prop_la_1}
\Delta^2(\mu^i_n) - 4 = \delta^2(\mu^i_n) = \ell^{\infty,1}_n
\ee
uniformly on $B$.
Since $\Delta(\mu^i_n)=2(-1)^n + \ell^{\infty,1/2}_n$ due to the second part of Theorem \ref{thm_asymp_seq_near_per_ev}, we deduce from \eqref{pr_prop_la_1} by writing the left hand side as $(\Delta(\mu^i_n) - 2) (\Delta(\mu^i_n) + 2)$ that
\be \label{pr_prop_la_2}
\Delta(\mu^i_n) = 2(-1)^n +  \ell^{\infty,1}_n
\ee
uniformly on $B$.

Next we will employ the identity 
\be \label{pr_prop_la_3}
\Delta(\dot\lambda^i_n) - \Delta(\mu^i_n) = (\dot\lambda^i_n - \mu^i_n)
\int^1_0 \dot \Delta(t \dot\lambda^i_n  + (1-t) \mu^i_n )  \, \mathrm d t,
\ee
which is a consequence of the fundamental theorem of calculus.
By Theorem  \ref{thm_asymp_seq_near_per_ev}, all values on the lines connecting $\dot\lambda^i_n$ and $\mu^i_n$ are $\mathcal O(1)$ as $|n|\to\infty$ uniformly on $B$. Moreover, $\dot\lambda^i_n - \mu^i_n=\ell^{\infty,1}_n$ uniformly on $B$ by the first assertion of this theorem concerning the critical points and Theorem \ref{prop_mu}. Hence, we infer from \eqref{pr_prop_la_2} and \eqref{pr_prop_la_3} that 
\be \label{pr_prop_la_4}
\Delta(\dot\lambda^i_n) = 2(-1)^n +  \ell^{\infty,1}_n
\ee
uniformly on $B$.
Furthermore, since $\dot\Delta(\dot\lambda^i_n)=0$ by definition, we obtain
\be \label{pr_prop_la_5}
\Delta(\lambda^{i,\pm}_n) = \Delta(\dot\lambda^i_n) + (\lambda^{i,\pm}_n - \dot\lambda^i_n)^2
\int^1_0 (1-t) \ddot \Delta(t \lambda^{i,\pm}_n + (1-t) \dot\lambda^i_n) \,\mathrm d t.
\ee
By recalling that $\Delta(\lambda^{i,\pm}_n)=2(-1)^n$ for all sufficiently large $|n|$, cf.~Remark \ref{rem_sign_Delta_per_ev}, we deduce from \eqref{pr_prop_la_4} and \eqref{pr_prop_la_5} that 
\be \label{pr_prop_la_6}
(\lambda^{i,\pm}_n - \dot\lambda^i_n)^2
\int^1_0 (1-t) \ddot \Delta(t \lambda^{i,\pm}_n + (1-t) \dot\lambda^i_n) \, \mathrm d t =  \ell^{\infty,1}_n
\ee
uniformly on $B$.
Using Cauchy's estimate and Theorem \ref{thm_asymptotics_M_1}, we find
\bew
\ddot \Delta(\lambda) = -16 \lambda^2 \cos 2 \lambda^2 + \mathcal O\big(|\lambda|\, \e^{2|\Im(\lambda^2)|}\big)
\eew
uniformly on $B$ as $|\lambda|\to\infty$. Hence, for a bi-infinite sequence $(z_n)_{n\in\Z}$, whose entries $z_n$ remain asymptotically in the discs $D^i_n$ ($i=1$ or $i=2$) we have $\ddot \Delta(z_n) = -16 z_n^2 \cos 2 z_n^2 + \mathcal O(|z_n| )$. By the Counting Lemmas, $(t \lambda^{i,\pm}_n + (1-t) \dot\lambda^i_n)=\ell^{\infty,1/2}_n$ uniformly for $t\in[0,1]$ and $\psi\in B$. Thus, the absolute value of the integral in \eqref{pr_prop_la_6} is $\Theta(|n|)$ as $|n|\to\infty$, which means that it grows precisely as fast as $|n|$. As a consequence, 
$(\lambda^{i,\pm}_n - \dot\lambda^i_n)^2= \ell^{\infty,2}_n$ uniformly on $B$; in other words,
\be \label{pr_prop_la_7}
\lambda^{i,\pm}_n - \dot\lambda^i_n = \ell^{\infty,1}_n
\ee
uniformly on $B$.
Since $\dot\lambda^i_n=\dot\lambda^i_n(0)+\ell^{\infty,1}_n = \lambda^{i,\pm}_n(0)+\ell^{\infty,1}_n$ uniformly on $B$, we conclude from \eqref{pr_prop_la_7} that $\lambda^{i,\pm}_n = \lambda^{i,\pm}_n(0) + \ell^{\infty,1}_n$ uniformly on $B$, hence the periodic eigenvalues satisfy the asymptotic estimate of  the first assertion of this theorem. This completes the proof. \end{proof}

\begin{remark}
Let $W$ be an open neighborhood of the zero potential such that Theorem \ref{prop_la} \eqref{prop_la_2} is satisfied. 
Then the map
\bew
W \to \ell^{\infty,1}_\C, \quad \psi \mapsto \big(\lambda^{i,\pm}_n(\psi) - \lambda^{i,\pm}_n(0)\big), \quad i=1,2
\eew
is continuous at $\psi=0$,
that is, for every sequence $(\psi_k)_{k\in\N}$ in $W$ with $\psi_k\to0$ as $k\to\infty$, it holds that
\bew
\lim_{k\to\infty} \bigg(\sup_{n\in\Z} (1+n^2)^\frac{1}{2} \big|\lambda^{i,\pm}_n(\psi_k) - \dot\lambda^{i,\pm}_n(0)\big| \bigg) = 0, \quad i=1,2.
\eew
However, as a consequence of the lexicographical ordering, the periodic eigenvalues $\lambda^{i,-}_n$ and $\lambda^{i,+}_n$ do not define analytic functions from $W$ to $D^i_n$, $n\in\Z\setminus \{0\}$. In order to formulate a version of Theorem \ref{prop_la} \eqref{prop_la_3} for periodic eigenvalues, one can consider suitable subsets of $W$ where $\lambda^{i,-}_n$ and $\lambda^{i,+}_n$ are isolated from each other---the situation near a potential $\psi\in W$ with a double periodic eigenvalue is cumbersome. 
Section \ref{sec_potRI} addresses these questions for potentials of so-called real and imaginary type. For such potentials the fundamental matrix solution possesses additional symmetries, which implies that the two periodic eigenvalues $\lambda^{i,\pm}_n$ in $D^i_n$ are either real or form a complex conjugate pair. In both cases, they are connected by an analytic arc along which the discriminant is real-valued.
\end{remark}

\begin{remark}
The asymptotic localization of Dirichlet eigenvalues, Neumann eigenvalues, periodic eigenvalues, and critical points provided by Theorems \ref{prop_mu}, \ref{prop_nu} and \ref{prop_la} can be slightly improved. Indeed, it is straightforward to adapt the proofs of these theorems to obtain, for each $p>2$, 
\begin{align*}
\mu^i_n &= \mu^i_n(0) + \ell^{p,1/2}_n, \\
\nu^i_n &= \mu^i_n(0) + \ell^{p,1/2}_n, \\
\lambda^{i,\pm}_n &= \mu^i_n(0) + \ell^{p,1/2}_n, \\
\dot\lambda^i_n &= \mu^i_n(0) + \ell^{p,1/2}_n,
\end{align*}
uniformly on bounded subsets of $\LL$, $i=1,2$.
Studies of other related spectral problems---see~e.g.~\cite{DjaMit06,KapMit01,Kappeler_etal_06,Poe11}---suggest that these
localization results can be further sharpened if attention is restricted to subspaces of more regular potentials. 
\end{remark}

\section{Potentials of real and imaginary type} \label{sec_potRI}
\noindent
This section considers  potentials  of  so-called real and imaginary type. These subspaces of the space $\LL$ of general $t$-periodic potentials  consist precisely of those potentials, which are relevant for the  $x$-evolution of the $t$-periodic defocusing NLS (real type) and focusing NLS (imaginary type).
Our main results are Theorems \ref{thm_conn_per_ev} and \ref{cor_conn_per_ev}, which state that for sufficiently small real and imaginary type potentials $\psi$, the corresponding periodic eigenvalues $\lambda^{1,-}_n(\psi)$ and $\lambda^{1,+}_n(\psi)$ are connected by analytic arcs in the complex plane for each $n\in\Z\setminus\{0\}$. These arcs form a subset of $\{\lambda\in\C\colon\Delta(\lambda,\psi)\in\R\}$. These results are needed for the establishment of local Birkhoff coordinates and shall  serve as a solid foundation for future investigations in this direction.
Theorem \ref{cor_conn_per_ev} is inspired by~\cite[Proposition 2.6]{Kappeler_etal_09}, which establishes similar properties for the $x$-periodic potentials of imaginary type for the focusing\footnote{The analogous result for $x$-periodic real type potentials for the defocusing NLS is trivial, since in this case all periodic eigenvalues are real valued due to selfadjointness of the corresponding ZS-operator,~cf.~\cite{GrebertKappeler14}.} NLS.
Our proof makes use of the ideas and techniques of~\cite{Kappeler_etal_09}.  
We refer to~\cite{Kappeler_etal_14} for further related results.

For potentials $\psi\in \LL$, we define
\bew
\psi^* \coloneqq P\bar{\psi} \coloneqq (\bar{\psi}^2, \bar{\psi}^1, \bar{\psi}^4, \bar{\psi}^3),
\eew
where
\bew
P \coloneqq \begin{pmatrix} \sigma_1 & 0 \\ 0 & \sigma_1 \end{pmatrix}, \qquad \sigma_1 = \begin{pmatrix} 0 & 1 \\ 1 & 0 \end{pmatrix}.
\eew
We say that a potential $\psi$ of the Zakharov-Shabat $t$-part \eqref{tlax} is of \emph{real type} if $\psi^* = \psi$.
In this case, $\psi^2 = \bar{\psi}^1$ and $\psi^4 = \bar{\psi}^3$, that is, $\psi = (q_0 + \I p_0, q_0 - \I p_0, q_1 + \I p_1, q_1 - \I p_1)$ for some real-valued functions $\{q_j, p_j\}_{j=0}^1$.
Hence a potential is of real type iff all coefficients of the corresponding AKNS system are real-valued.
The subspace of $\LL$ of all real type potentials will be denoted by
\bew
\LL_{\mathcal R} \coloneqq \{\psi \in \LL \, | \, \psi^* = \psi\}.
\eew
Note that this is a \emph{real} subspace of $\LL$, not a complex one; it consists of those potentials that are relevant for the \emph{defocusing} NLS. 

We can write the Zakharov-Shabat $t$-part \eqref{tlax} as
\bew
\bigg(-\I\sigma_3\partial_t + 2\lambda^2\id + \begin{pmatrix} \psi^1 \psi^2 & \I(2\lambda \psi^1 + \I \psi^3) \\
-2\I \lambda \psi^2 - \psi^4 & \psi^1 \psi^2 \end{pmatrix}\bigg)\phi = 0,
\eew
or, in other words,
\begin{align}\label{LphiRphi}
L(\psi)\phi = R(\lambda,\psi) \phi
\end{align}
with
\begin{align} \label{L,R}
L(\psi) &\coloneqq -\I\sigma_3\partial_t + 
\begin{pmatrix} \psi^1 \psi^2 & -\psi^3 \\
- \psi^4 & \psi^1 \psi^2 \end{pmatrix}, \qquad
R(\lambda,\psi) \coloneqq -2\lambda^2 \id - 2\I\lambda
\begin{pmatrix} 0& \psi^1  \\
-\psi^2 & 0 \end{pmatrix}.
\end{align}

For $v = (v_1, v_2)$ and $w = (w_1, w_2)$, let
\bew
\langle v, w \rangle 
= \int_0^1 (v_1 \bar{w}_1 + v_2 \bar{w}_2) \, \mathrm d t.
\eew
If the eigenfunctions $v,w$ lie in the periodic domain $\mathcal D_{\mathrm P}$, we can integrate by parts without boundary terms and find that
\begin{equation*}
\langle w, L(\psi) v \rangle 
=  \langle L(\psi^*)w, v\rangle.	
\end{equation*}
Therefore, if the potential $\psi$ is of real type and $v$ is a periodic eigenfunction with eigenvalue $\lambda$,  
\bew
\langle R(\lambda,\psi)v, v\rangle = \langle L(\psi)v, v\rangle = \langle v,  L(\psi)v\rangle
= \langle v,  R(\lambda,\psi)v\rangle
\eew
and thus we find that
\bew
\Im \lambda = 0 \quad \text{or} \quad \Re \lambda = \frac{\Im \big( \int^1_0 \psi^1 v_2\bar{v}_1 \, \mathrm d t \big)}{\langle v, v\rangle}.
\eew
According to the Counting Lemma, the periodic eigenvalues of type $\lambda^2_n(\psi)$ for arbitrary $\psi\in\LL$ necessarily possess non-vanishing imaginary parts for sufficiently large $|n|$.
In analogy with the $x$-part \eqref{xlax}, one might expect that $\Im \lambda^{1,\pm}_n(\psi)=0$ for real type potentials. However, we will see in Section \ref{sec_singexp} that this is not the case: there are single exponential potentials of real type for which some $\lambda^{1,\pm}_n$ are nonreal, cf.~Fig~\ref{fig:singexp2}. 

\begin{lemma} \label{lem_SymmM_1}
Let $\psi$ be of real type and let $\lambda\in \R$. Then $m_4=\bar m_1$ and $m_3=\bar m_2$. In particular, $\Delta$ is real-valued on $\R\times \LL_{\mathcal R}$.
Moreover, if a solution $v$ of $L(\psi) v = R(\lambda,\psi)v$ is real in AKNS-coordinates, then $v=\sigma_1 \bar v$.
\end{lemma}
\begin{proof}
Since $\psi=\psi^*$, the AKNS coordinates $(p_j,q_j)$, $j=0,1$, are real. If in addition $\lambda \in \R$, the system \eqref{AKNS_system} has real coefficients, so its fundamental solution $K$ is real-valued. The relation $M=T K T^{-1}$, cf.~\eqref{fund_K}, 
then implies that  $m_4=\bar m_1$ and $m_3=\bar m_2$.
To prove the second claim, we note that $\bar T = \sigma_1 T$ and hence  
\bew
\bar M = \bar T K \bar T^{-1} = \sigma_1 T K T^{-1} \sigma_1 = \sigma_1 M \sigma_1.
\eew 
If $v$ is real in AKNS coordinates, it has real initial data $v_0$ and $v=M T v_0$. Therefore 
\bew
\bar v = \bar M \bar T v_0 = \sigma_1 M T v_0 = \sigma_1 v.
\eew
\end{proof}

We say that a potential $\psi\in\LL$ is of \emph{imaginary type} if $\psi^*=-\psi$. The subspace 
\bew
\LL_{\mathcal I} \coloneqq \{\psi\in\LL\colon \psi^*=-\psi\} 
\eew
of potentials of imaginary type is relevant for the \emph{focusing} NLS.

\begin{proposition} \label{symm_M}
For $\psi \in \LL_{\mathcal R}$ the fundamental solution $M$ satisfies 
\be\label{Msymm}
 M(t,\bar\lambda,\psi)= \sigma_1 \overbar{M(t,\lambda,\psi)} \sigma_1, \quad 
 \lambda \in \C, \; t\geq 0;
\ee
if $\psi \in \LL_{\mathcal I}$ then $M$ satisfies 
\be\label{Msymm'}
 M(t,\bar\lambda,\psi)= \sigma_1 \sigma_3 \overbar{M(t,\lambda,\psi)} \sigma_3 \sigma_1    , \quad 
 \lambda \in \C, \; t \geq 0.
\ee
In particular,
\be \label{symm_Delta}
\Delta (\bar \lambda,\psi) = \overbar{\Delta (\lambda,\psi)} \quad \text{and} \quad 
\dot\Delta (\bar \lambda,\psi) = \overbar{\dot\Delta (\lambda,\psi)}
\ee
for all $\psi \in \LL_{\mathcal R}\cup \LL_{\mathcal I}$ and $\lambda \in \C$, so that $\Delta$ and $\dot\Delta$ are real-valued on $\R \times (\LL_{\mathcal R} \cup \LL_{\mathcal I})$.
\end{proposition}
\begin{proof}
Let us first assume that $\psi \in \LL_{\mathcal R}$ and $\lambda \in \C$. Then a computation using \eqref{L,R} shows that
\bew
L(\psi) v = R (\lambda,\psi) v \quad \iff \quad L(\psi)  v^* = R (\bar \lambda,\psi)  v^*,
\eew
where $ v^*  \coloneqq \sigma_1 \bar v = (\bar v_2, \bar v_1)$. 
The symmetry \eqref{Msymm} follows from uniqueness of the solution of \eqref{LphiRphi} and the initial condition $M(0, \lambda, \psi) = I$. Evaluation of \eqref{Msymm} at $t = 1$ gives \eqref{symm_Delta}. This finishes the proof for the case of real type potentials.
If $\psi \in \LL_{\mathcal I}$, we instead have
\bew
L(\psi) v = R(\lambda,\psi) v
\quad \iff \quad
L(\psi) \hat v = R(\bar \lambda,\psi) \hat v 
\eew
where $\hat v\coloneqq \sigma_1 \sigma_3 \bar v=(-\bar v_2,\bar v_1)$, which leads to \eqref{Msymm'} and \eqref{symm_Delta}.
\end{proof}

\begin{corollary} \label{cor_crit_val_real}
There exists a neighborhood $W$ of the zero potential in $\LL$ such that for each 
$\psi \in W \cap ( \LL_{\mathcal R}\cup \LL_{\mathcal I})$ and each $n\in\Z$, 
\bew
\{\lambda \in \C \colon \dot\Delta(\lambda,\psi)=0 \} \cap D^1_n 
= \{ \dot\lambda^1_n (\psi) \} \quad \text{and} \quad \dot\lambda^1_n (\psi)\in \R.
\eew
\end{corollary}
\begin{proof}
We already know from Theorem \ref{prop_la} that there exists a neighborhood $W$ of the zero potential such that, for all general potentials $\psi\in W$ and all $n\in \Z$, 
\bew
\{\lambda \in \C \colon \dot\Delta(\lambda,\psi)=0 \} \cap D^1_n 
= \{ \dot\lambda^1_n (\psi) \}.
\eew
Due to the symmetry \eqref{symm_Delta} we infer that, for all potentials $\psi\in W \cap ( \LL_{\mathcal R}\cup \LL_{\mathcal I})$ and $n\in\Z$, 
\bew
 0 = \dot\Delta (\dot\lambda^1_n (\psi),\psi) = \overbar{\dot\Delta (\dot\lambda^1_n (\psi),\psi)} 
 = \dot\Delta (\bar{\dot\lambda}^1_n (\psi),\psi).
\eew
Since $\dot\lambda^1_n (\psi)$ is the only root of $\dot \Delta (\cdot,\psi)$ in $D^1_n$, we conclude that 
$\dot\lambda^1_n (\psi)$ is real.
\end{proof}

\begin{theorem} \label{thm_conn_per_ev}
There exists a neighborhood $W$ of the zero potential in $\LL$ and a sequence of nondegenerate rectangles 
\be \label{Rectangle_eps_delta}
R^{\eps,\delta}_n \coloneqq \big\{ \lambda \in \C \colon 
|\Re\lambda-\dot\lambda^1_n(0)|<\delta_n,  \;  |\Im \lambda| <\eps_n \big\}, \quad n\in\Z,
\ee
with $\eps, \delta \in \ell^{\infty,1/2}_\R$,  such that for all $\psi \in W \cap (\LL_{\mathcal R}\cup \LL_{\mathcal I})$ and every 
$n\in\Z \setminus\{0\}$,
\bew
\{\lambda \in \C \colon \Delta(\lambda,\psi)\in\R \} \cap R^{\eps,\delta}_n 
= \gamma_n(\psi) \cup (R^{\eps,\delta}_n\cap \R),
\eew
where the subset $\gamma_n(\psi)\subseteq\C$ forms an analytic arc transversal to the real axis, which crosses the real line in the critical point $\dot \lambda^1_n(\psi)$ of $\Delta(\cdot,\psi)$. These arcs are symmetric under reflection in the real axis and the orthogonal projection of $\gamma_n(\psi)$ to the imaginary axis is a real analytic diffeomorphism onto its image. 
\end{theorem}

We refer to Fig.~\ref{fig:R} for an illustration of the analytic arc $\gamma_n(\psi)$ within the rectangle $R^{\eps,\delta}_n$ centered at the critical point $\dot\lambda^1_n(0)$ of the discriminant $\Delta(\cdot,0)$. 

Before proving Theorem \ref{thm_conn_per_ev}, we state an important consequence of Theorem \ref{thm_conn_per_ev} and Theorem \ref{prop_la}. Namely that for all small enough real and imaginary type potentials, the periodic eigenvalues $\lambda^{1,-}_n(\psi)$ and $\lambda^{1,+}_n(\psi)$ are connected by an analytic arc along which the discriminant is real-valued. More precisely, we will deduce the following result.

\begin{theorem} \label{cor_conn_per_ev}
There exists a neighborhood $W^*$ of $\psi=0$ in $\LL$ such that for each 
$\psi \in  W^* \cap (\LL_{\mathcal R}\cup \LL_{\mathcal I})$ and each $n\in\Z\setminus\{0\}$ there exists an analytic arc $\gamma^*_n\equiv\gamma^*_n(\psi)\subseteq\C$  connecting the two periodic eigenvalues $\lambda^{1,\pm}_n \equiv \lambda^{1,\pm}_n(\psi)$. 
Qualitatively we distinguish two different cases: either (i)
$\gamma^*_n=[\lambda^{1,-}_n,\lambda^{1,+}_n]\subseteq\R$ or (ii) $\gamma^*_n$ is transversal to the real line,
symmetric under reflection in the real axis, and the orthogonal projection of $\gamma^*_n$ to the imaginary axis is a real analytic diffeomorphism onto its image. In both cases, it holds that
\begin{enumerate}
\item $\Delta(\gamma^*_n,\psi) \subseteq [-2,2]$,
\item $\overbar{\gamma^*_n} = \gamma^*_n$,
\item $\dot\lambda^1_n(\psi) \in  \gamma^*_n \cap \R$,
\item For a parametrization by arc length $\rho_n \equiv \rho_n(s)$ of $\gamma^*_n$ with 
$\rho_n(0)=\dot\lambda^1_n(\psi)$, the function 
$s\mapsto \Delta(\rho_n(s),\psi)$ is strictly monotonous along the two connected components of $\gamma^*_n\setminus\{\dot\lambda^1_n(\psi) \}$.  
\end{enumerate}
(We include the possible scenario $\lambda^{1,-}_n(\psi)=\lambda^{1,+}_n(\psi)=\dot\lambda^1_n(\psi)$, where the set $\gamma^*_n(\psi)$ consists of the single element $\dot\lambda^1_n(\psi)$, as a degenerate special case.) 
\end{theorem}

The remainder of this section is devoted to the proofs of Theorems \ref{thm_conn_per_ev} and \ref{cor_conn_per_ev}. We follow closely the ideas and methods of the proof of~\cite[Proposition 2.6]{Kappeler_etal_09}---a related result for the $x$-periodic focusing NLS. The proof is based on an application of the implicit function theorem for real analytic mappings in an infinite dimensional setting. This level of generality is necessary in order to treat the arcs $\gamma_n$ in a uniform way. 

Let us first briefly discuss the strategy of the proof. Writing $\lambda = x + \I y$ with $x,y \in \R$, we split $\Delta(\lambda;\psi) \equiv \Delta(x,y;\psi)$ into its real and imaginary parts and write $\Delta = \Delta_1 + \I \Delta_2$ with
\bew
\Delta_1(x,y;\psi) \coloneqq \Re \big(\Delta (\lambda;\psi)\big), \quad 
\Delta_2(x,y;\psi) \coloneqq \Im \big(\Delta (\lambda;\psi)\big).
\eew
The problem is then transformed into the study of the zero level set of $\Delta_2 (\lambda;\psi)=\Delta_2 (x,y;\psi)$. 
By Proposition \ref{symm_M}, $\Delta_2(x, 0; \psi)=0$ for any $x \in \R$ and $\psi \in \LL_{\mathcal R} \cup \LL_{\mathcal I}$.
Therefore, following~\cite{Kappeler_etal_09}, we introduce the function 
\begin{align}
\begin{aligned} \label{def_tilde_F}
\tilde{F} \colon \R \times (\R\setminus\{0\}) \times (\LL_{\mathcal R}\cup \LL_{\mathcal I}) \to \R, \\
(x,y,\psi) \mapsto\tilde F(x,y;\psi) \coloneqq \frac{\Delta_2(x,y;\psi)}{y},
\end{aligned}
\end{align}
which has the same zeros on $\R \times (\R\setminus\{0\}) \times (\LL_{\mathcal R}\cup \LL_{\mathcal I})$ as $\Delta_2$.
We observe that $\tilde F$ has a real analytic extension 
\be \label{def_F}
F \colon \R \times \R \times (\LL_{\mathcal R}\cup \LL_{\mathcal I}) \to \R.
\ee
To see this, we recall that $\Delta$ is analytic on $\C\times \LL$ and real valued on  
$\R \times  (\LL_{\mathcal R}\cup \LL_{\mathcal I})$.
Hence $\Delta_2$ vanishes on $\R \times \{0\} \times (\LL_{\mathcal R}\cup \LL_{\mathcal I})$ and is real analytic there.
Thus $\Delta_2(x,y;\psi)/y$ admits a Taylor series representation at $y=0$, which converges absolutely to the analytic extension $F$ of $\tilde F$ locally near $y=0$.

For $\psi \in \LL_{\mathcal R}\cup \LL_{\mathcal I}$ and real sequences $u=(u_n)_{n\in\Z}$ and $v=(v_n)_{n\in\Z}$, we define the map 
\be \label{def_calF}
\mathcal F = (\mathcal F_n)_{n\in\Z}, \quad
\mathcal F_n(u,v;\psi) \coloneqq F(\dot\lambda^1_n + u_n,v_n;\psi).
\ee
For the zero potential and the zero sequence, both denoted by $0$, we calculate
\bew
\mathcal F(0,0;0) = \big(-8 \dot\lambda^1_n \sin(2(\dot\lambda^1_n)^2)  \big)_{n\in\Z}
= \big(-8 \dot\lambda^1_n \sin |n|\pi \,  \big)_{n\in\Z}
 = 0.
\eew
In order to determine $\frac{\partial \mathcal F}{\partial u}$ at the origin $(0,0;0)$, we first observe that 
$\frac{\partial \mathcal F}{\partial u}$ has diagonal form because $\mathcal F_j$ is independent of $u_n$ for $j\in\Z$ with $j\neq n$. On the diagonal, we obtain  
\begin{align*}
\frac{\partial \mathcal F_n}{\partial u_n} (u,v;0) 
&= \frac{\partial}{\partial u_n} \bigg[ \frac{\Delta_2(\dot\lambda^1_n +u_n,v_n;0)}{v_n} \bigg] \\
&= - \frac{2}{v_n} \bigg\{ 4(\dot\lambda^1_n + u_n) \cos[2((\dot\lambda^1_n + u_n)^2-v^2_n)]  
\sinh [4(\dot\lambda^1_n + u_n) v_n\big] \\
& \qquad  \qquad \qquad   
+  \sin[2((\dot\lambda^1_n + u_n)^2-v^2_n)] \, 4 v_n \cosh [4(\dot\lambda^1_n + u_n) v_n]  \bigg\},
\end{align*}
thus
\bew
\frac{\partial \mathcal F_n}{\partial u_n} (0,0;0) 
= - 32 (\dot\lambda^1_n)^2 \cos[2(\dot\lambda^1_n)^2],
\eew
and therefore
\be \label{derivative_calF}
\frac{\partial \mathcal F}{\partial u} (0,0;0) = \big(  16 \pi |n| \, \mathrm{diag}((-1)^{n+1}) \big)_{n\in\Z}.
\ee
Consequently, the right-hand side of \eqref{derivative_calF} is at least \emph{formally} bijective in a set-theoretic and algebraic sense, for example as a mapping from the linear space of real sequences $\{u=(u_n)_{n\in\Z} \colon \Z\to\R \,| \, u_0=0 \}$ to itself. In order to give these formal considerations a rigorous justification, we need to consider appropriate subspaces of sequences equipped with suitable topologies. 
Due to the quadratic nature of the underlying generalized eigenvalue problem, the right choice of spaces is quite a delicate issue. 
In contrast to the related $x$-periodic problem for the focusing NLS, see~\cite{Kappeler_etal_09}, we can not rely on $\ell^\infty$ sequences, but need to make use of the weighted $\ell^p$-based spaces of $\ell^{p,s}$ sequences, which we introduced earlier in \eqref{lpsdef}. The establishment of the necessary bounds for the mapping $\mathcal F$ between these spaces turns out to be highly nontrivial, cf.~Lemma \ref{lem_unif_bd_D_yDelta}.

\smallskip

Let us discuss the basic properties of the $\ell^{p,s}_\K$ spaces, where $\K=\R$ or $\K=\C$, which appear in the formulation of Theorem \ref{thm_conn_per_ev}, and Propositions \ref{prop_mu}, \ref{prop_nu} and \ref{prop_la}. 
For $1\leq p \leq \infty$ and $s\in\R$, we consider the linear spaces
\bew
\ell^{p,s}_\K \coloneqq \Big\{ u= (u_n)_{n\in\Z}  \; \big| \; \big((1+n^2)^{\frac{s}{2}} u_n\big)_{n\in\Z} \in \ell^p_\K  \Big\}
\eew
endowed with the norms
\begin{align*}
|u|_{p,s} \; \coloneqq \bigg( \sum^\infty_{n=-\infty} (1+n^2)^{\frac{s p}{2}} |u_n|^p \bigg)^\frac{1}{p}, \quad
1\leq p <\infty; \quad
|u|_{\infty,s} \coloneqq \sup_{n\in\Z} \big\{ (1+n^2)^{\frac{s}{2}} |u_n| \big\}.
\end{align*}
One easily checks that these spaces are Banach spaces. 
Furthermore, defining
\bew
\Lambda_n \coloneqq (1+n^2)^{\frac{1}{2}}, \quad n\in\Z, 
\eew
the map
\bew
\Lambda^r \colon \ell^{p,s}_\K \to \ell^{p,s-r}_\K, \quad u_n \mapsto \Lambda^r_n u_n,
\eew 
is an isometric isomorphism for each $r\in\R$. In particular $\Lambda^s$ maps $\ell^{p,s}_\K$ isometrically onto $\ell^p_\K$.
For $s\in\R$ and $1<p<\infty$, the topological dual of $\ell^{p,s}_\K$
is isometrically isomorphic to $\ell^{q,-s}_\K$, i.e., $(\ell^{p,s}_\K)' \cong \ell^{q,-s}_\K$, where $q$ is the H\"older conjugate of $p$ defined by $1/p+1/q=1$.
The isomorphism is given by the dual pairing
\be \label{duality}
\langle \cdot,\cdot \rangle_{p,s;q,-s} \colon \ell^{p,s}_\K \times \ell^{q,-s}_\K \to \K, \quad 
\langle u,v \rangle_{p,s;q,-s}  \coloneqq \sum^\infty_{n=-\infty} u_n v_n,
\ee
and can be deduced directly from the well-known $\ell^p$-$\ell^q$-duality.
Henceforth, we will identify the dual of $\ell^{p,s}_\K$ with $\ell^{q,-s}_\K$ by means of $\langle \cdot,\cdot \rangle_{p,s;q,-s}$. In particular, $\ell^{p,s}_\K$ is a reflexive Banach space for $1<p<\infty$. 

We will also use the closed subspaces
\bew
\check \ell^{p,s}_\K \coloneqq \{u \in \ell^{p,s}_\K \colon u_0=0  \};
\eew
for $1<p<\infty$ their topological duals are given by:
\be \label{duality_check}
(\check\ell^{p,s}_\K)' \cong \check \ell^{q,-s}_\K.
\ee
The linear operator $T$ defined by 
\bew
T_n u_n \mapsto |n| u_n, \quad T\colon \check \ell^{p,s}_\K \to \check \ell^{p,s-1}_\K
\eew
is a topological isomorphism. Likewise, $T^r\colon  u_n \mapsto T^r_n u_n= |n|^r u_n$ is an isomporhism 
$\check\ell^{p,s}_\K \to \check\ell^{p,s-r}_\K$ for real $r$.

\smallskip

The first part of the proof of Theorem \ref{thm_conn_per_ev} uses techniques from the theory of analytic maps between complex Banach spaces. We therefore review some aspects of this theory. Let $(E,\|\cdot\|_E)$, $(F,\|\cdot\|_F)$ be complex Banach spaces. 
Furthermore, we denote by $\mathcal L (E,F)$ the Banach space of bounded $\C$-linear operators $E\to F$ endowed with the operator norm 
$\|\cdot\|_{\mathcal L(E,F)}$, where $\|L\|_{\mathcal L(E,F)}=\sup_{0\neq h\in E}\frac{\|L h\|_F}{\|h\|_E}<\infty$ for $L\in \mathcal L(E,F)$. In the special case  $F=\C$, we denote by $E'=\mathcal L(E,\C)$ the topological dual space of $E$.
Let $O\subseteq E$ be an open subset. A map $f \colon O \to F$ is called \emph{analytic} or \emph{holomorphic}, if it is 
Fr\'echet differentiable in the complex sense
at every $u\in O$, i.e., if for each $u\in O$ there exists a bounded linear operator $A(u)\in\mathcal L(E,F)$ such that 
\bew
\lim_{\|h\|_E\to 0} \frac{\| f(u+h)-f(u)-A(u) h\|_F}{\|h\|_E} = 0.
\eew
In this case we call $A(u)$ the derivative of $f$ at $u$ and write $\mathrm d f(u)$ for $A(u)$. In the special case $E=F=\C$, we simply write $\mathrm d f(u)=f'(u)\in\C \cong \C'$.
We call $f$ 
\emph{weakly analytic on} $O$ if for every $u\in O$, $h\in E$ and $L\in F'$ the function
\bew
z \mapsto L f(u+z h)
\eew 
is analytic in some neighborhood of zero. 

We provide the basic characterization of analytic maps between complex Banach spaces in the following lemma.  
\begin{lemma}\cite[Theorem A.4]{KappelerPoeschel03} \label{lem_char_analytic_maps}
Let $E$ and $F$ be complex Banach spaces, let $O\subseteq E$ be open and let $f\colon O\to F$ be a mapping. The following statements are equivalent.
\begin{enumerate}
\item $f$ is analytic in $O$.
\item $f$ is  weakly analytic and locally bounded on $O$.
\item $f$ is infinitely many times differentiable on $O$ and for each $u\in O$ the Taylor series of $f$ at $u$, given by 
\bew
f(u+h) = f(u) + \mathrm d f(u) h + \frac{1}{2} \mathrm d^2 f(u) (h,h) + \cdots + \frac{1}{n!} \mathrm d^n f(u) (h,\dots,h) + \cdots
\eew 
converges to $f$ absolutely and uniformly in a neighborhood of $u$.
\end{enumerate}
\end{lemma}
The next lemma, also referred to as Cauchy's inequality, provides an important estimate for the multilinear map $\mathrm d^n f(u)$.

\begin{lemma} \label{lem_Cauchy}
Let $E$ and $F$ be complex Banach spaces and let $f$ be analytic from the open ball of radius $r$ around $u$ in $E$ into $F$ such that $\|f\|_F \leq M$ on this ball. Then for all integers $n\geq0$,
\bew
\max_{0\neq h\in E}\frac{\|\mathrm d^n f(u) (h,\dots,h)\|_F}{\|h\|^n_E} \leq \frac{M n!}{r^n}.
\eew
\end{lemma}
\begin{proof}
By the previous lemma, $f$ is infinitely often differentiable at $u$ with $n$-th derivative $\mathrm d^n f(u) \in \mathcal L^n(E,F)$, the space of continuous $n$-linear mappings $E\times \cdots \times E \to F$.
The lemma now follows directly from the usual Cauchy inequality for holomorphic  Banach space valued functions on a complex domain by considering the holomorphic map  $\varphi(z) \coloneqq f(u+ z h)$ for arbitrary $h\neq0$ on the disc with radius $r/|h|$ centered at the origin of $\C$. See e.g.~\cite[Lemma A.2]{KappelerPoeschel03}, see also e.g.~\cite[Chapter III.14]{DunSchw} for the generalization of the classical theory of complex analysis for functions $f\colon \C\supseteq O\to\C$ to complex Banach space valued functions $f\colon \C\supseteq O\to F$ defined on a complex domain, and~\cite{Mujica86} for a general account on complex analysis in Banach spaces. 
\end{proof}

The purpose of the next lemma is the establishment of certain bounds we will use later on.

\begin{lemma} \label{lem_unif_bd_D_yDelta}
Let $\psi\in\LL_{\mathcal R}\cup \LL_{\mathcal I}$, let $\lambda = x + \I y \in \C$ with $x,y \in \R$, and let $\Delta_1$ and $\Delta_2$ denote the real and imaginary parts of $\Delta = \Delta_1 + \I \Delta_2$. 
\begin{enumerate}
\item \label{lem_unif_bd_D_yDelta_pt1} 
As $|\lambda| \to\infty$, the partial derivative 
$\partial_y \Delta_2$ satisfies the asymptotic estimate
\begin{align}
\begin{aligned} \label{lem_unif_bd_D_yDelta_0}
\partial_y \Delta_2(x,y;\psi) = & -8\big(x \sin [2(x^2-y^2)] \cosh[4x y] - y \cos [2(x^2-y^2)] \sinh[4x y] \big) \\
& - 4 \I \grave\Gamma(\psi) \cos[2(x^2-y^2)] \cosh[4x y] + \mathcal O \bigg(\frac{\e^{4|x y|}}{\sqrt{x^2+y^2}}\bigg)
\end{aligned}
\end{align}
uniformly for $\psi$ in bounded subsets of $\LL_{\mathcal R}\cup \LL_{\mathcal I}$, where
\bew
\grave\Gamma(\psi) \coloneqq \Gamma(1,\psi) 
=  \int^1_0 ( \psi^1 \psi^4 - \psi^2 \psi^3 ) \, \mathrm d t.
\eew

\item \label{lem_unif_bd_D_yDelta_pt2}
Set $\dot \lambda^1_n\coloneqq\dot \lambda^1_n(0)$. The mapping 
\be \label{lem_unif_bd_D_yDelta_1}
(x_n,y_n) \mapsto \partial_y \Delta_2(\dot \lambda^1_n + x_n,y_n;\psi),
\ee
which is real analytic in each coordinate,
maps bounded sets in $\ell^{\infty,1/2}_\R \times \ell^{\infty,1/2}_\R$ to bounded sets in $\ell^{\infty,-1/2}_\R$; the corresponding bound in $\ell^{\infty,-1/2}_\R$ can be chosen uniformly for $\psi$ varying within bounded subsets of $\LL_{\mathcal R}\cup \LL_{\mathcal I}$. 

The assertion remains true when considering   \eqref{lem_unif_bd_D_yDelta_1} as a mapping  
\bew
\ell^{\infty,1/2}_\C \times \ell^{\infty,1/2}_\C \times \LL \supseteq 
\big(\ell^{\infty,1/2}_\R \times \ell^{\infty,1/2}_\R \times (\LL_{\mathcal R}\cup \LL_{\mathcal I})\big)\otimes \C 
\to  \ell^{\infty,-1/2}_\R \otimes \C
\eew
by means of the coordinatewise analytic extension to all of $\big(\ell^{\infty,1/2}_\R \times \ell^{\infty,1/2}_\R \times (\LL_{\mathcal R}\cup \LL_{\mathcal I})\big)\otimes \C$; that is, this mapping maps bounded sets in
$\ell^{\infty,1/2}_\C \times \ell^{\infty,1/2}_\C$ to bounded sets in $\ell^{\infty,-1/2}_\C$ uniformly on bounded subsets of 
$(\LL_{\mathcal R}\cup \LL_{\mathcal I})\otimes \C$.
\end{enumerate}
\end{lemma}
\begin{proof}
In order to prove part \eqref{lem_unif_bd_D_yDelta_pt1}, we recall from Theorem \ref{thm_asymptotics_M_1} that 
\bew
\M(\lambda,\psi) = \e^{-2\I \lambda^2 \sigma_3} + \mathcal O \bigg(\frac{\e^{2|\Im (\lambda^2)|}}{|\lambda|}\bigg).
\eew
In the proof of Theorem \ref{thm_asymptotics_M_1}, we gained additional information on the remainder term: it is of the form
\bew
\frac{Z_1(\psi)}{\lambda} \e^{-2\I \lambda^2 \sigma_3} + \frac{W_1(\psi)}{\lambda} \e^{2\I \lambda^2 \sigma_3}
+ \mathcal O \bigg(\frac{\e^{2|\Im (\lambda^2)|}}{|\lambda|^2}\bigg),
\eew
where the diagonal part of the $1/\lambda$-terms is given by 
\bew
\frac{1}{2\lambda}\grave\Gamma(\psi)\, \sigma_3  \e^{-2\I \lambda^2 \sigma_3}.
\eew
Thus the discriminant satisfies
\bew
\Delta(\lambda,\psi) = 2 \cos 2\lambda^2 - \frac{\I \grave\Gamma}{\lambda} \sin 2 \lambda^2 
+  \mathcal O \big(|\lambda|^{-2} \, \e^{2|\Im (\lambda^2)|}\big)
\eew
for any potential $\psi\in\LL$, and its $\lambda$-derivative satisfies
\be \label{lem_unif_bd_D_yDelta_2}
\dot\Delta(\lambda,\psi) = -8\lambda \sin 2\lambda^2 - 4\I \grave\Gamma(\psi) \cos 2\lambda^2 
+ O \big(|\lambda|^{-1} \, \e^{2|\Im (\lambda^2)|}\big).
\ee
Since $\grave\Gamma(\psi) \in \I\R$ for $\psi \in \LL_{\mathcal R} \cup \LL_{\mathcal I}$, the asymptotic estimate \eqref{lem_unif_bd_D_yDelta_0} follows by taking the real part of \eqref{lem_unif_bd_D_yDelta_2}.

We prove part \eqref{lem_unif_bd_D_yDelta_pt2} of the lemma in the complex setting; this includes the real setting as a special case. 
The analytic extension $\widetilde{\partial_y \Delta_2}$ of $\partial_y \Delta_2$ to $\C\times\C\times (\LL_{\mathcal R}\cup \LL_{\mathcal I})\otimes \C$ is given by 
\bew
\widetilde{\partial_y \Delta_2}(x,y;\psi) = \dot\Delta(x + \I y,\psi), \quad x,y\in\C, \; 
\psi\in (\LL_{\mathcal R}\cup \LL_{\mathcal I})\otimes \C,
\eew
which, according to the first part of the proof, satisfies the asymptotic estimate
\be \label{D_yDelta_complexified}
\widetilde{\partial_y \Delta_2}(x,y;\psi) = -8(x + \I y) \sin[2(x + \I y)^2] - 4\I \grave\Gamma(\psi) \cos[2(x + \I y)^2] 
+ O \bigg(\frac{\e^{2|\Im ((x + \I y)^2)|}}{|x + \I y|}  \bigg).
\ee
The error term holds uniformly on bounded subsets of $(\LL_{\mathcal R}\cup \LL_{\mathcal I})\otimes \C$ for $|x + \I y|\to\infty$; likewise $\grave\Gamma$ possesses a uniform bound on bounded subsets of this potential space. Therefore we only need to establish the desired bounds for arbitrary potentials $\psi \in (\LL_{\mathcal R}\cup \LL_{\mathcal I})\otimes \C$ and the uniformity on bounded subsets follows automatically.
We write the complexification of \eqref{lem_unif_bd_D_yDelta_1} (by means of analytic extensions in all  coordinates) as
\be \label{D_yDelta_complexified_2}
\ell^{\infty,1/2}_\C \times \ell^{\infty,1/2}_\C \times [(\LL_{\mathcal R}\cup \LL_{\mathcal I})\otimes \C] \to \ell^{\infty,-1/2}_\C, \quad
(x_n,y_n) \mapsto \widetilde{\partial_y \Delta_2}(\dot \lambda^1_n + x_n,  y_n;\psi), 
\ee
and employ the asymptotic estimate \eqref{D_yDelta_complexified} to deduce the asserted bounds for \eqref{D_yDelta_complexified_2}. 
Let us verify the bounds separately for the three components of \eqref{D_yDelta_complexified_2}, which arise from the three terms in \eqref{D_yDelta_complexified},  beginning with the error term. That is, we first show that 
\be \label{pr_lem_unif_bd_D_yDelta_errorterm}
(x_n,y_n) \mapsto \frac{\e^{2|\Im ((\dot \lambda^1_n+x_n + \I y_n)^2)|}}{|\dot \lambda^1_n+x_n + \I y_n|}
\ee
maps bounded sets in $\ell^{\infty,1/2}_\C \times \ell^{\infty,1/2}_\C$ to bounded sets in $\ell^{\infty,-1/2}_\R$.
We clearly have that 
\bew
(x_n,y_n) \mapsto |\Im ((\dot \lambda^1_n+x_n + \I y_n)^2)|
\eew
maps bounded sets in $\ell^{\infty,1/2}_\C \times \ell^{\infty,1/2}_\C$ to bounded sets in $\ell^\infty_\R$, hence the nominator in \eqref{pr_lem_unif_bd_D_yDelta_errorterm} is bounded in $\ell^\infty_\R$ uniformly on bounded sets in $\ell^{\infty,1/2}_\C \times \ell^{\infty,1/2}_\C$. 
It follows that the whole expression on the right-hand side of \eqref{pr_lem_unif_bd_D_yDelta_errorterm} is bounded in $\ell^{\infty,1/2}_\R \subseteq \ell^{\infty,-1/2}_\R$ 
on bounded sets in $\ell^{\infty,1/2}_\C \times \ell^{\infty,1/2}_\C$.
Next we show that 
\be \label{pr_lem_unif_bd_D_yDelta_secondterm}
(x_n,y_n) \mapsto \big|\cos[2(\lambda^1_n+x_n + \I y_n)^2]\big|
\ee
maps bounded sets in $\ell^{\infty,1/2}_\C \times \ell^{\infty,1/2}_\C$ to bounded sets in $\ell^{\infty,-1/2}_\R$, which ensures that the second component of \eqref{D_yDelta_complexified_2} has the asserted property. By recalling that $\sin[2(\dot\lambda^1_n)^2]=0$ and $\cos[2(\dot\lambda^1_n)^2]=(-1)^n$ and employing the classical trigonometric addition formulas, we obtain  
\begin{align} \label{pr_lem_unif_bd_D_yDelta_cos}
\begin{aligned}
\big|\cos[2(\lambda^1_n+x_n + \I y_n)^2]\big| &\leq
\big| \cos[2(x^2_n + 2\I x_n y_n - y^2_n)] \cos[4 \lambda^1_n (x_n + \I y_n)]  \big| \\ 
 &\qquad + \big|\sin[2(x^2_n + 2\I x_n y_n - y^2_n)] \sin[4 \lambda^1_n (x_n + \I y_n)] \big|.
\end{aligned}
\end{align}
The second term on the right-hand side of \eqref{pr_lem_unif_bd_D_yDelta_cos} is bounded in $\ell^{\infty,1/2}_\R$  on bounded sets of $\ell^{\infty,1/2}_\C \times \ell^{\infty,1/2}_\C$; the first term is bounded in $\ell^\infty_\R$. Thus \eqref{pr_lem_unif_bd_D_yDelta_secondterm} maps bounded sets in $\ell^{\infty,1/2}_\C \times \ell^{\infty,1/2}_\C$ to bounded sets in $\ell^\infty_\R \subseteq \ell^{\infty,-1/2}_\R$.
Finally, we infer by similar arguments that the first component of \eqref{D_yDelta_complexified_2} maps bounded sets in $\ell^{\infty,1/2}_\C \times \ell^{\infty,1/2}_\C$ to bounded sets in $\ell^{\infty,-1/2}_\C$. Indeed, 
\bew
(x_n,y_n) \mapsto \big|\sin[2(\lambda^1_n+x_n + \I y_n)^2]\big|
\eew
maps bounded sets in $\ell^{\infty,1/2}_\C \times \ell^{\infty,1/2}_\C$ to bounded sets in $\ell^\infty_\R$, thus 
\bew
(x_n,y_n) \mapsto \big|(\lambda^1_n+x_n + \I y_n)\sin[2(\lambda^1_n+x_n + \I y_n)^2]\big|
\eew
maps bounded sets in $\ell^{\infty,1/2}_\C \times \ell^{\infty,1/2}_\C$ to bounded sets in $\ell^{\infty,-1/2}_\R$. We conclude that the mapping \eqref{D_yDelta_complexified_2} has the asserted boundedness properties, which finishes the proof of the lemma.
\end{proof}

Below we provide an elementary criterion, which helps to show analyticity of  functions, which map to $\ell^{\infty,s}_\C$; 
unlike Lemma \ref{lem_char_analytic_maps} it does not involve the dual $(\ell^{\infty,s}_\C)'$ in this particular situation. 
The criterion is formulated for the target space $\ell^{\infty,s}_F$, $s\in\R$, where $F$ is a complex Banach space; i.e.~$\ell^{\infty,s}_F$ denotes the Banach space 
\bew
\Big\{ u=(u_n)_{n\in\Z}\colon \Z\to F \;\Big|\; \|u\|_{\infty,s}\coloneqq\sup_{n\in\Z} \big\{ (1+n^2)^\frac{s}{2} |u_n|_F \big\} < \infty  \Big\}
\eew
with norm $\|u\|_{\infty,s}$.

\begin{lemma} \label{lem_analytic_linfty}
Let $s\in\R$, let $E$, $F$ be complex Banach spaces  and let $O\subseteq E$ be an open subset.  
If the function
\bew
f \colon O \to \ell^{\infty,s}_F, \quad u\mapsto f(u) = \big(f_n(u)\big)_{n\in\Z} 
\eew 
is locally bounded and  each coordinate function $f_n \colon O \to F$ is analytic, then $f$ is analytic.
\end{lemma}
\begin{proof}
A proof for the case $s=0$ can be found in~\cite[Theorem A.3]{KappelerPoeschel03}, and this proof can easily be generalized to the case of arbitrary $s\in\R$. Let us for convenience state this proof, which verifies the differentiability of $f$ at an arbitrary point $u\in O$ directly. 
By assumption there is a ball centered at  $u$ such that $f$ is bounded in $\ell^{\infty,s}_F$ on this ball. In particular, each 
$(1+n^2)^{s/2} |f_n|_F$ is bounded by the same constant. Since all $f_n$ are moreover analytic, it follows from the Taylor series representation applied to each $f_n$, cf.~Lemma~\ref{lem_char_analytic_maps}, and an application of Cauchy's estimate, cf.~Lemma~\ref{lem_Cauchy}, that 
\bew
\|f_n(u+h) - f_n(u) - \mathrm d f_n(u) h\|_F \leq C (1+n^2)^{-\frac{s}{2}}  \|h\|^2_E
\eew 
for small enough $\|h\|_E$, where $C$ is independent of $n\in\Z$. This means that for small  $\|h\|_E$, 
\bew
\|f(u+h) - f(u) - (\mathrm d f_n(u) h)_{n\in\Z} \|_{\infty,s} \leq C  \|h\|^2_E,
\eew
from which we infer that $f$ is differentiable at $u$ with derivative 
\bew
\mathrm d f(u) = (\mathrm d f_n(u))_{n\in\Z}\in \mathcal L(E,\ell^{\infty,s}_F).
\eew 
\end{proof}

\begin{proof}[Proof of Theorem \ref{thm_conn_per_ev}]
The real analytic extension $F\colon \R \times \R \times (\LL_{\mathcal R}\cup \LL_{\mathcal I}) \to\R$ of $\tilde F$, cf.~\eqref{def_tilde_F} and \eqref{def_F}, can be written as
\be \label{int_rep_F}
F(x,y;\psi) = \int^1_0 (\partial_2 \Delta_2) (x, s y; \psi)  \, \mathrm d s,
\ee
where $\partial_2$ denotes the partial derivative with respect to the second variable.
Indeed, 
\begin{align*}
\int^1_0 (\partial_2 \Delta_2)(x, s y; \psi) y \, \mathrm d s 
&= \int^y_0 (\partial_2 \Delta_2)(x, y'; \psi)  \, \mathrm d y' = \Delta_2(x,y;\psi) - \Delta_2(x,0;\psi),
\end{align*}
where $\Delta_2(x,0;\psi)=0$ because, by Proposition \ref{symm_Delta}, $\Delta$ is real-valued on $\R \times (\LL_{\mathcal R}\cup \LL_{\mathcal I})$. We obtain from \eqref{int_rep_F} that 
\be \label{est_int_rep_F}
|F(x,y;\psi)| \leq \max_{s\in[0,1]} |(\partial_2 \Delta_2)(x, s y; \psi)|.
\ee

In view of Lemma \ref{lem_unif_bd_D_yDelta} and \eqref{est_int_rep_F}, the operator $\mathcal F$ given by \eqref{def_calF} defines a well-defined map
\bew
\mathcal F \colon \check \ell^{\infty,1/2}_\R \times \check \ell^{\infty,1/2}_\R \times (\LL_{\mathcal R}\cup \LL_{\mathcal I}) \supseteq
B^{\infty,1/2}_1 \times B^{\infty,1/2}_1 \times  (\LL_{\mathcal R}\cup \LL_{\mathcal I})
\to \check \ell^{\infty,-1/2}_\R,
\eew
where $B^{\infty,1/2}_1 \equiv B^{\infty,1/2}_1(\check \ell^{\infty,1/2}_\R)$ denotes the open unit ball in $\check \ell^{\infty,1/2}_\R$ centered at $0$; the space
$\check \ell^{\infty,1/2}_\R \times \check \ell^{\infty,1/2}_\R \times ( \LL_{\mathcal R}\cup \LL_{\mathcal I})$ is endowed with the usual product topology. 
Since $F$ is real analytic, it admits an analytic extension $F_\C$ to some open set in $\C \times \C \times [(\LL_{\mathcal R}\cup \LL_{\mathcal I})\otimes \C]$, which contains $\R \times \R \times (\LL_{\mathcal R}\cup \LL_{\mathcal I})$.
Let us consider the complexification $\big(B^{\infty,1/2}_1 \times B^{\infty,1/2}_1 \times  (\LL_{\mathcal R}\cup \LL_{\mathcal I})\big)\otimes \C$ of $B^{\infty,1/2}_1 \times B^{\infty,1/2}_1 \times  (\LL_{\mathcal R}\cup \LL_{\mathcal I})$. 
By an application of Lemma \ref{lem_unif_bd_D_yDelta} and \eqref{est_int_rep_F}, there exists an open set
$\mathcal U_\C \subseteq \big( \check \ell^{\infty,1/2}_\R \times \check \ell^{\infty,1/2}_\R \times (\LL_{\mathcal R}\cup \LL_{\mathcal I})\big)\otimes \C$, which contains 
$B^{\infty,1/2}_1 \times B^{\infty,1/2}_1 \times ( \LL_{\mathcal R}\cup \LL_{\mathcal I})$, such that the coordinatewise analytic extension 
$\mathcal F_\C \colon \mathcal U_\C \to \check \ell^{\infty,-1/2}_\C$ of $\mathcal F$ is bounded on bounded subsets of $\mathcal U_\C$; in particular $\mathcal F_\C$ is locally bounded on $\mathcal U_\C$. 
From Lemma \ref{lem_analytic_linfty} we conclude that $\mathcal F_\C$ is analytic on $\mathcal U_\C$; in particular, $\mathcal F$ is real analytic.

The partial derivative $\partial_u \mathcal F(0,0;0)$, which is given by \eqref{derivative_calF}, is a topological isomorphism $\check \ell^{\infty,1/2}_\R \to \check \ell^{\infty,-1/2}_\R$ and $\mathcal F(0,0;0)=0$. Thus we can apply the implicit function theorem for Banach space valued real analytic functions, cf.~\cite{Whittlesey65}. 
We infer the existence of an open neighborhood $W$ of the zero potential in $\LL_{\mathcal R}\cup \LL_{\mathcal I}$, an open $\eps$-ball 
$B^{\infty,1/2}_\eps$ and a $\delta$-ball $B^{\infty,1/2}_\delta$ around the origin in $\check \ell^{\infty,1/2}_\R$ and a real analytic function 
\bew
\mathcal G \colon B^{\infty,1/2}_\eps \times W \to B^{\infty,1/2}_\delta
\eew
such that, for all $v\in B^{\infty,1/2}_\eps$ and $\psi \in  W$,
\bew
\mathcal F(\mathcal G(v,\psi),v,\psi) = 0,
\eew
and such that the map
\bew
(v,\psi) \mapsto (\mathcal G(v,\psi),v,\psi), \quad B^{\infty,1/2}_\eps \times W \to 
B^{\infty,1/2}_\delta \times B^{\infty,1/2}_\eps \times W
\eew
describes the zero level set of $\mathcal F$ in $B^{\infty,1/2}_\delta \times B^{\infty,1/2}_\eps \times W$. 

We may assume that the sequences $\eps=(\eps_n)_{n\in\Z}$ and $\delta =(\delta_n)_{n\in\Z}$ satisfy $\eps_n>0$ and $\delta_n>0$ for $n\in \Z\setminus\{0\}$ and $\eps_0=\delta_0=0$. Clearly, if $-1\leq\tau_n\leq1$ for each $n$, then $(\tau_n \eps_n)_{n\in\Z} \in B^{\infty,1/2}_\eps$ and $(\tau_n \delta_n)_{n\in\Z} \in B^{\infty,1/2}_\delta$. Thus we can run through the intervals in each coordinate in a uniform way. Let $R^{\eps,\delta}_n$, $n\in\Z\setminus\{0\}$, be the associated sequence of nondegenerate rectangles defined in \eqref{Rectangle_eps_delta}.
\begin{figure}[h!] \centering
\begin{subfigure}{.43\textwidth}  \centering
\begin{overpic}[width=.95\textwidth]{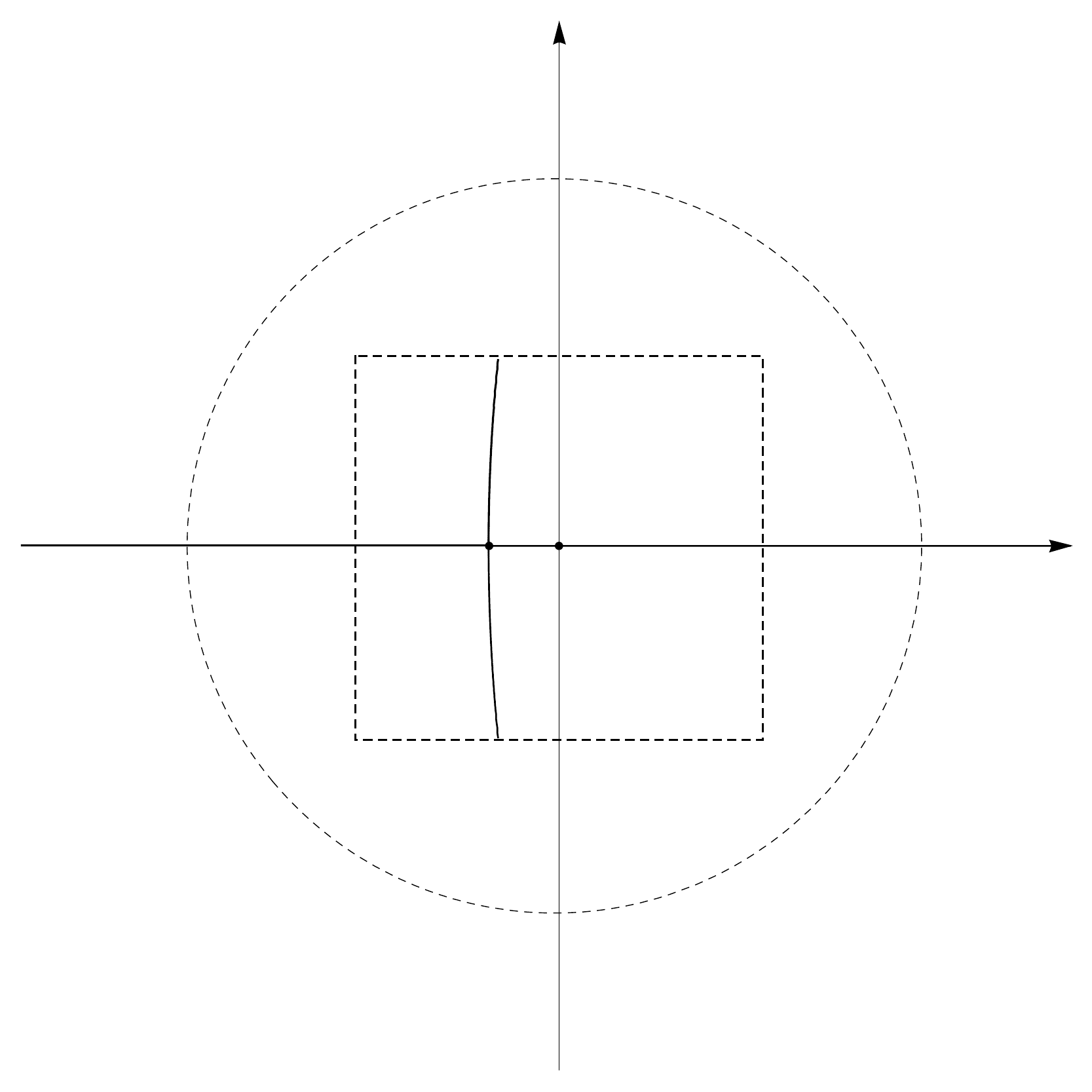}
      \put(60,35){\tiny{$R^{\eps,\delta}_n$}}
      \put(43,20){\tiny{$D^1_n$}}
      \put(39,52){\tiny{$\dot\lambda^1_n$}}
      \put(52,52){\tiny{$\dot\lambda^1_n(0)$}}
      \put(40,38){\tiny{$\gamma_n$}}   
      \put(92,52){\tiny{$\Re \lambda$}}  
\end{overpic}
\caption{}
\label{fig:R}
\end{subfigure}
\begin{subfigure}{.43\textwidth}  \centering
\begin{overpic}[width=.95\textwidth]{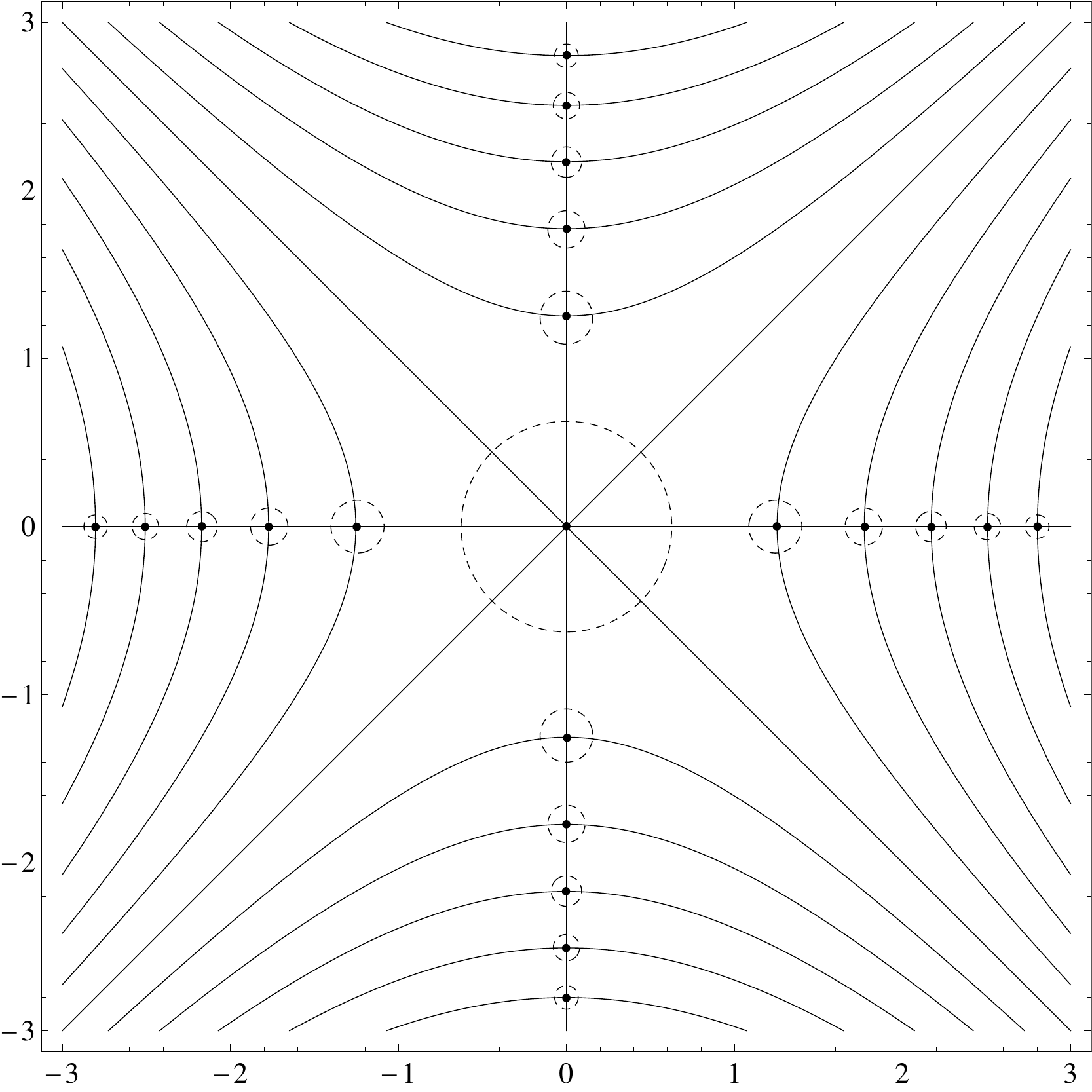}
\end{overpic}
\caption{}
\label{fig:0_pot}
\end{subfigure}
\caption{Fig.~\ref{fig:R} shows an illustration of the path $\gamma_n$ within the rectangle $R^{\eps,\delta}_n$ which is contained in the disc $D^1_n$. The critical points $\dot\lambda^1_n=\dot\lambda^1_n(\psi)=\gamma_n\cap\R$ and $\dot\lambda^1_n(0)$ are marked with dots.
Fig.~\ref{fig:0_pot} shows a plot of the zero set of $\Delta_2(\cdot,0)$ in the complex $\lambda$-plane; the boundaries of the discs $D^i_n$, $i=1,2$, are indicated by dashed circles, the periodic eigenvalues (which coincide with the critical points of $\Delta(\cdot,0)$ and the Dirichlet and Neumann eigenvalues) are indicated by dots.}
\label{}
\end{figure}
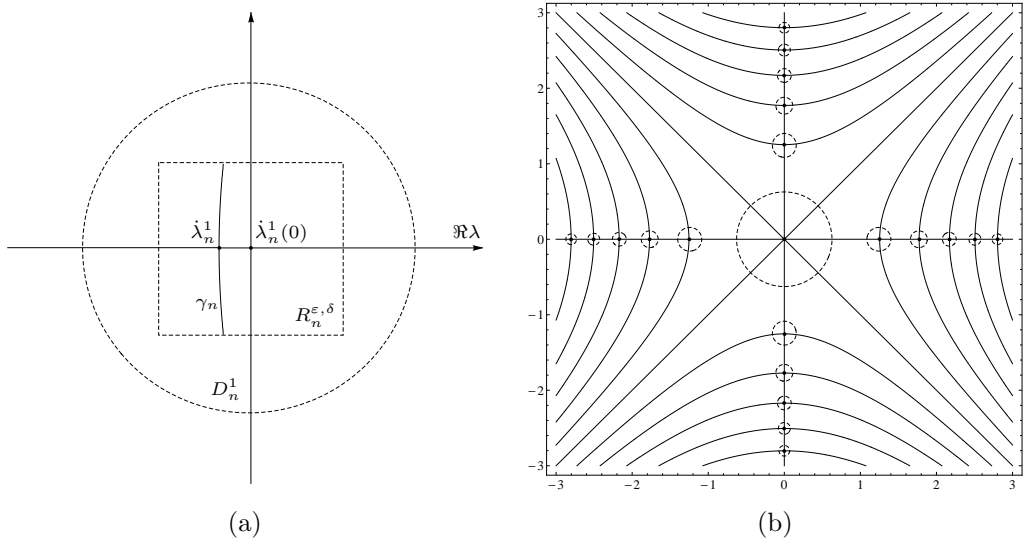
Our considerations show that, for every $\psi\in W$ and $n\in\Z\setminus\{0\}$, the zero set of $F$ can be parametrized locally near $\dot\lambda^1_n$ by the real analytic function
\bew
z_n(\psi) \colon (-\eps_n,\eps_n) \to R^{\eps,\delta}_n, \quad
y_n \mapsto \dot\lambda^1_n + \mathcal G_n(y_n,\psi) + \I y_n.
\eew
We set
\bew
\gamma_n(\psi) \coloneqq z_n(\psi)( (-\eps_n,\eps_n)) \subseteq R^{\eps,\delta}_n
\eew
and denote the zero set of $\Delta_2(\cdot,\psi)$ by 
\bew
N_{\Delta_2}(\psi) \coloneqq \{(x,y) \in \R^2 \colon \Delta_2(x,y,\psi)=0 \} \subseteq \R^2.
\eew
By construction,
\bew
 \gamma_n(\psi) \setminus \R = N_{\Delta_2}(\psi) \cap (R^{\eps,\delta}_n \setminus \R),
\eew
and furthermore, since $\Delta$ is real-valued on $\R \times (\LL_{\mathcal R}\cup \LL_{\mathcal I})$, cf.~Proposition \ref{symm_M}, we have 
\bew
N_{\Delta_2}(\psi) \cap R^{\eps,\delta}_n =  \gamma_n(\psi) \cup (R^{\eps,\delta}_n \cap\R) \eqqcolon Z_n(\psi) \subseteq \C.
\eew
Thus for arbitrary $\psi\in W$ and every $n\in\Z\setminus\{0\}$, $\lambda \in R^{\eps,\delta}_n$ satisfies
\be \label{pr_thm_conn_per_ev_01}
\Delta(\lambda,\psi)\in \R \iff \lambda \in Z_n(\psi).
\ee

The intersection $\gamma_n(\psi) \cap \R$ consists of a single point which we denote by $\xi_n \equiv \xi_n(\psi) \in R^{\eps,\delta}_n \subseteq D^1_n$. We will show that $\xi_n=\dot \lambda^1(\psi)$. Since $\Delta_2$ vanishes on the curve $\gamma_n(\psi)$, which is orthogonal to the real line at the point $\xi_n$, we have $\partial_y \Delta_2(\xi_n,\psi)=0$. Furthermore, we know that $\Delta_2$ vanishes on $\R$, hence $\partial_x \Delta_2(\xi_n,\psi)=0$. The Cauchy-Riemann equations then imply that 
$\dot \Delta(\xi_n,\psi)=\partial_y \Delta_2(\xi_n,\psi) + \I \partial_x \Delta_2(\xi_n,\psi) = 0$; hence $\xi_n(\psi)$ is a critical point of 
$\Delta(\cdot,\psi)$. Since $\Delta(\cdot,\psi)$ has only one critical point in $D^1_n$, namely $\dot \lambda^1_n(\psi)$ according to Corollary \ref{cor_crit_val_real}, we conclude that $\gamma_n(\psi)$ crosses the real line in the point $\dot \lambda^1_n(\psi) \in  R^{\eps,\delta}_n$.
\end{proof}

\smallskip

\begin{proof}[Proof of Theorem~\ref{cor_conn_per_ev}]
According to Theorem \ref{thm_conn_per_ev} there exists a neighborhood $W$ of $0$ in $\LL$ such that for $\psi \in W\cap (\LL_{\mathcal R}\cup \LL_{\mathcal I})$ and arbitrary $n\in\Z \setminus \{0\}$, the analytic arc $\gamma_n(\psi)$ and the respective part of the real line describe the preimage of $\R$ under $\Delta(\cdot,\psi)$ locally around $\dot\lambda^1_n(0)$.  This arc is transversal to the real line, symmetric under reflection in the real axis, and the orthogonal projection of $\gamma_n(\psi)$ to the imaginary axis is a real analytic diffeomorphism onto its image.

Let us consider the rectangles $R^{\eps,\delta}_n$, which are centered at  $\lambda^{1,\pm}_n(0)=\dot\lambda^1_n(0)$ such that $\gamma_n(\psi)\subseteq R^{\eps,\delta}_n$ uniformly for all $\psi \in W\cap (\LL_{\mathcal R}\cup \LL_{\mathcal I})$ and all $n\in\Z \setminus \{0\}$. 
Since the lengths  and widths $\delta_n$ and $\eps_n$ are of order $\ell^{\infty,1/2}_\R$, it is guaranteed by Theorem \ref{prop_la} that there exists a neighborhood $W^*$ of $\psi=0$ in $\LL$ such that 
$W^*\cap (\LL_{\mathcal R}\cup \LL_{\mathcal I}) \subseteq W$ and such that $\lambda^{1,\pm}_n(\psi),\dot\lambda^{1}_n(\psi) \in R^{\eps,\delta}_n$ for all $\psi \in W^* \cap (\LL_{\mathcal R}\cup \LL_{\mathcal I})$ and all $n\in\Z \setminus \{0\}$. 

Furthermore, Theorem \ref{prop_la} tells us that
\be \label{pr_thm_conn_per_ev_02}
\Delta(\lambda^{1,\pm}_n(\psi),\psi)= 2(-1)^n.
\ee
Using the notation of the proof of Theorem \ref{thm_conn_per_ev}, we infer from \eqref{pr_thm_conn_per_ev_01} in combination with \eqref{pr_thm_conn_per_ev_02} and Corollary 
\ref{cor_crit_val_real} that
\bew
\lambda^{1,\pm}_n(\psi) \in Z_n(\psi) \quad \text{and} \quad 
\dot\lambda^1_n(\psi) \in Z_n(\psi) \cap \R =R^{\eps,\delta}_n \cap \R
\eew 
for all $\psi \in W \cap( \LL_{\mathcal R}\cup \LL_{\mathcal I})$. If both $\lambda^{1,-}_n(\psi)$ and $\lambda^{1,+}_n(\psi)$ are real, we set 
\bew
\gamma^*_n \equiv \gamma^*_n(\psi) \coloneqq [\lambda^{1,-}_n(\psi),\lambda^{1,+}_n(\psi)]
\subseteq R^{\eps,\delta}_n \cap \R;
\eew
otherwise we set
\bew
\gamma^*_n \equiv \gamma^*_n(\psi) \coloneqq  \gamma_n(\psi) \cap \{ \lambda \in \C \colon |\Delta(\lambda,\psi)|\leq 2 \}.
\eew
In both cases, we have that
$\Delta(\{\gamma^*_n \},\psi) \subseteq [-2,2]$, $\overbar{\gamma^*_n} = \gamma^*_n$ and
$\dot\lambda^1_n(\psi) \in  \gamma^*_n \cap \R$. 

Finally, considering a parametrization by arc length $\rho_n \equiv \rho_n(s)$ of $\gamma^*_n$ with 
$\rho_n(0)=\dot\lambda^1_n(\psi)$, we have that
\bew
\frac{\mathrm d}{\mathrm d s}\big[\Delta(\rho_n(s),\psi)\big]=0 \iff 
\dot \Delta (\rho_n(s),\psi) =0 \iff s=0,
\eew
because $\big| \frac{\mathrm d}{\mathrm d s} \rho_n\big|\equiv 1$ by assumption and, by Corollary \ref{cor_crit_val_real}, $\rho_n(0)=\dot\lambda^1_n(\psi)$ is the only root of $\dot \Delta(\cdot,\psi)$ in 
$R^{\eps,\delta}_n \subseteq D^1_n$.
\end{proof}


\section{Example: single exponential potential} \label{sec_singexp}
\vspace{0em}

\noindent
In this section, we consider single exponential potentials $\psi$ of real and imaginary type:
\be \label{singexp}
\psi(t) = (\alpha \e^{\I \omega t}, \sigma\bar\alpha \e^{-\I \omega t},c \e^{\I \omega t}, \sigma \bar c \e^{-\I \omega t}),
\quad \alpha, c \in \C, \; \omega \in \R, \;\sigma\in\{\pm1\}.
\ee 
The fundamental matrix solution that corresponds to a single exponential potential can be calculated explicitly. 
In Figures~\ref{fig:ex_singexp1} and \ref{fig:singexp2} we provide numerical plots of the periodic eigenvalues and the set $\{\lambda\in\C \colon \Delta(\cdot,\psi) \in \R\}$ for several particular potentials $\psi$ of the form \eqref{singexp}.

\begin{figure}[h!] \centering
\begin{subfigure}{.43\textwidth}  \centering
\begin{overpic}[width=.95\textwidth]{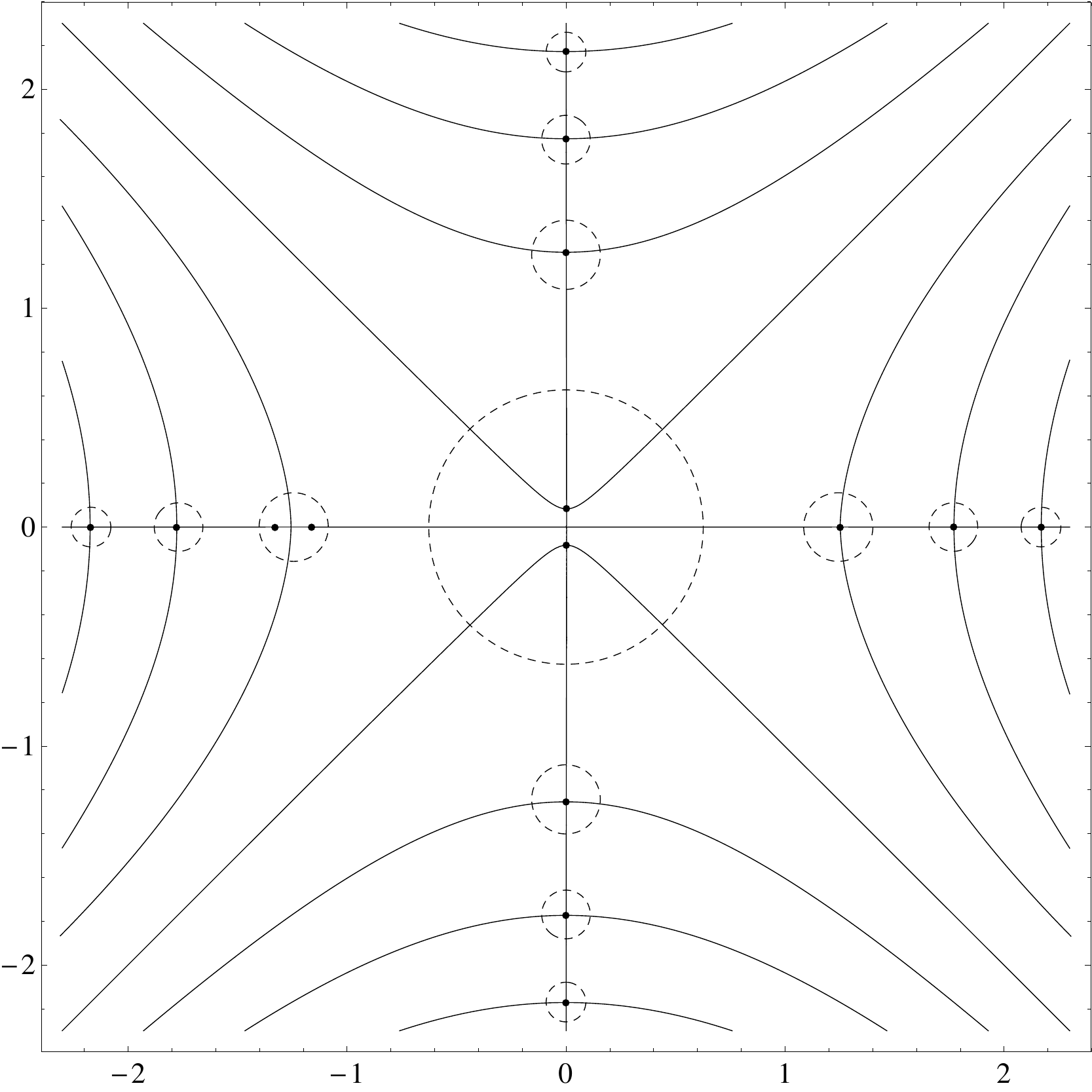}
     \put(19,47.5){{\tiny $\lambda^{1,-}_{-1}$}}
     \put(29,54){{\tiny $\lambda^{1,+}_{-1}$}}
     \put(55,74){{\tiny $\lambda^{2,\pm}_{1}$}}
     \put(55,27){{\tiny $\lambda^{2,\pm}_{-1}$}}
     \put(78,47){{\tiny $\lambda^{1,\pm}_{1}$}}
     \put(47,56){{\tiny $\lambda^{2,+}_{0}$}}
     \put(47,44){{\tiny $\lambda^{2,-}_{0}$}}
\end{overpic}
\caption{$\sigma=1$, $\omega=-2\pi$, $\alpha=\frac{1}{12}$, $c=\I \alpha \beta$}
  \label{fig:singexp1A}
\end{subfigure} 
\quad
\begin{subfigure}{.43\textwidth}  \centering
\begin{overpic}[width=.95\textwidth]{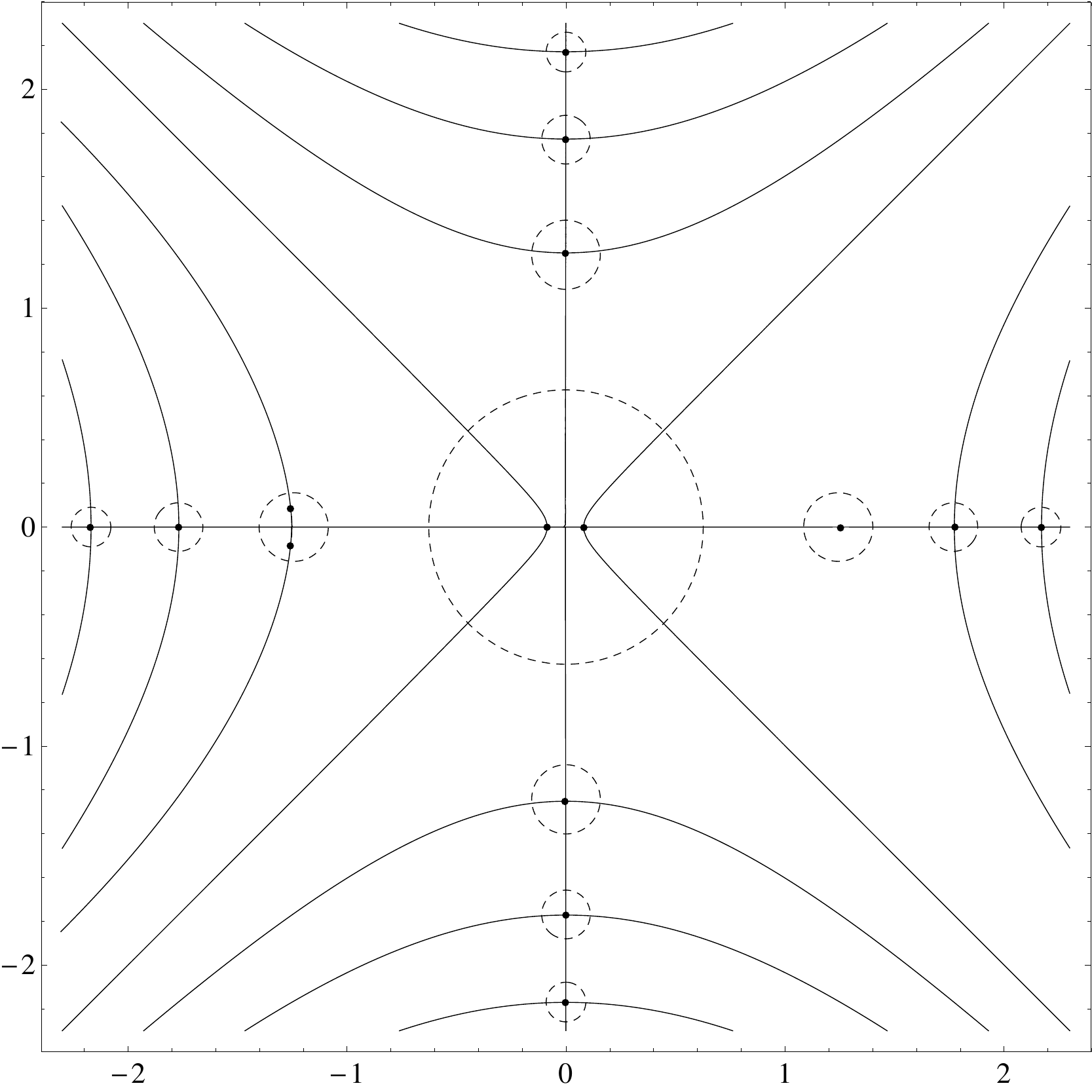}
     \put(27,56){{\tiny $\lambda^{1,+}_{-1}$}}
     \put(27,45){{\tiny $\lambda^{1,-}_{-1}$}}
     \put(42.5,53.5){{\tiny $\lambda^{1,-}_{0}$}}
     \put(56.5,53.5){{\tiny $\lambda^{1,+}_{0}$}}
     \put(54,72){{\tiny $\lambda^{2,\pm}_{1}$}}
     \put(55,28){{\tiny $\lambda^{2,\pm}_{-1}$}}
     \put(78,47){{\tiny $\lambda^{1,\pm}_{1}$}}
\end{overpic}
\caption{$\sigma=-1$, $\omega=-2\pi$, $\alpha=\frac{1}{12}$, $c=\I \alpha \beta$}
  \label{fig:singexp1B}
\end{subfigure}
\vspace{1em}

\begin{subfigure}{.43\textwidth}  \centering
\begin{overpic}[width=.95\textwidth]{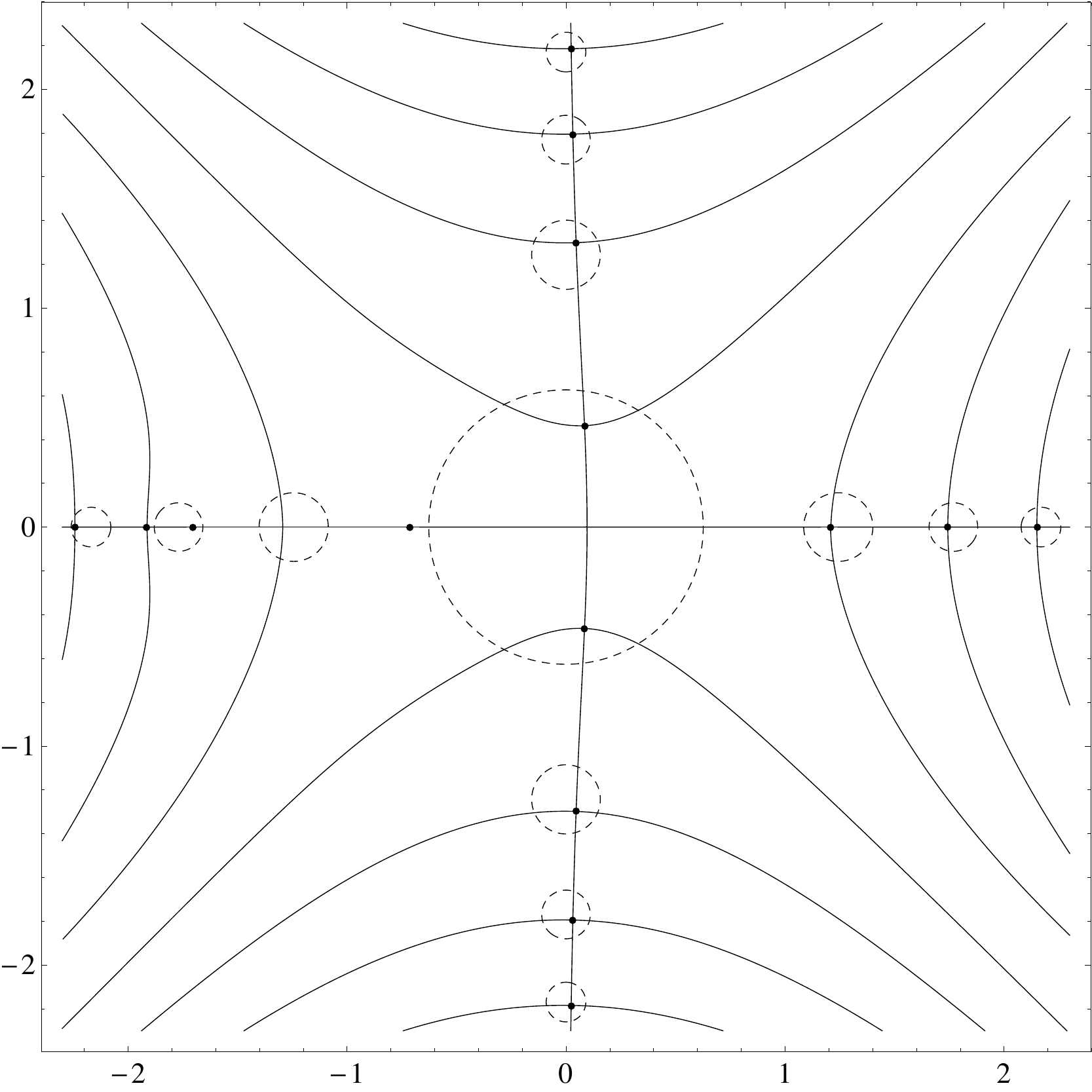}
     \put(17,47){{\tiny $\lambda^{1,-}_{-1}$}}
     \put(33,54){{\tiny $\lambda^{1,+}_{-1}$}}
     \put(55,74){{\tiny $\lambda^{2,\pm}_{1}$}}
     \put(55,27){{\tiny $\lambda^{2,\pm}_{-1}$}}
     \put(78,47){{\tiny $\lambda^{1,\pm}_{1}$}}
     \put(46,57){{\tiny $\lambda^{2,+}_{0}$}}
     \put(46,43){{\tiny $\lambda^{2,-}_{0}$}}
\end{overpic}
\caption{$\sigma=1$, $\omega=-2\pi$, $\alpha=\frac{1}{2}$, $c=\I \alpha \beta$}
  \label{fig:singexp1C}
\end{subfigure}
\quad
\begin{subfigure}{.43\textwidth}  \centering
\begin{overpic}[width=.95\textwidth]{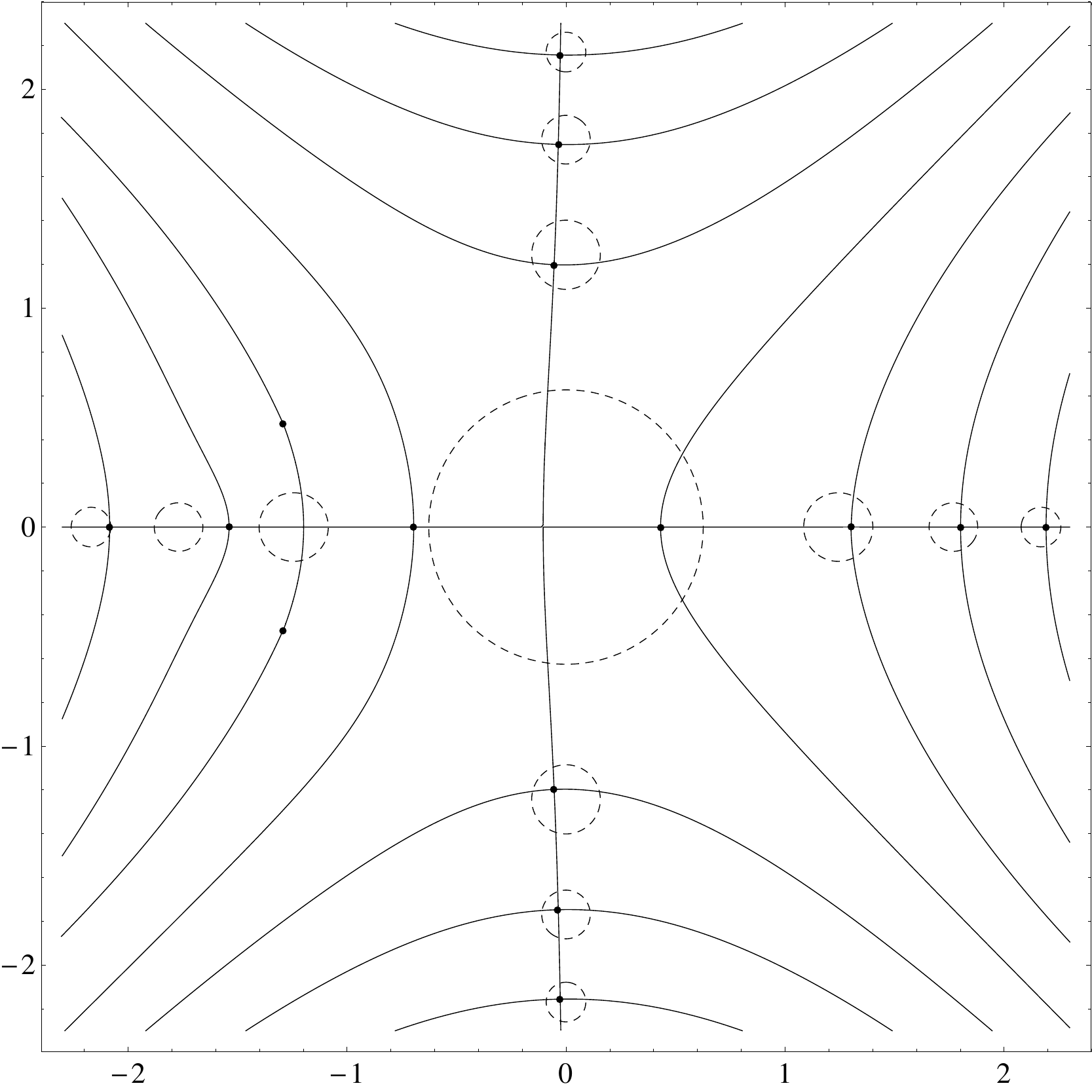}
     \put(27,62){{\tiny $\lambda^{1,+}_{-1}$}}
     \put(27,39){{\tiny $\lambda^{1,-}_{-1}$}}
     \put(32.5,53.5){{\tiny $\lambda^{1,-}_{0}$}}
     \put(61,53.5){{\tiny $\lambda^{1,+}_{0}$}}
     \put(54,72){{\tiny $\lambda^{2,\pm}_{1}$}}
     \put(55,28){{\tiny $\lambda^{2,\pm}_{-1}$}}
     \put(78,47){{\tiny $\lambda^{1,\pm}_{1}$}}
\end{overpic}
\caption{$\sigma=-1$, $\omega=-2\pi$, $\alpha=\frac{1}{2}$, $c=\I \alpha \beta$}
  \label{fig:singexp1D}
\end{subfigure}
\caption{Plots of the zero level sets of $\Delta_2(\cdot,\psi)$ for single exponential potentials of real type (left column) and imaginary type (right column); periodic eigenvalues are indicated with dots and the dashed circles are the boundaries of the discs $D^i_n$.}
\label{fig:ex_singexp1}
\end{figure}
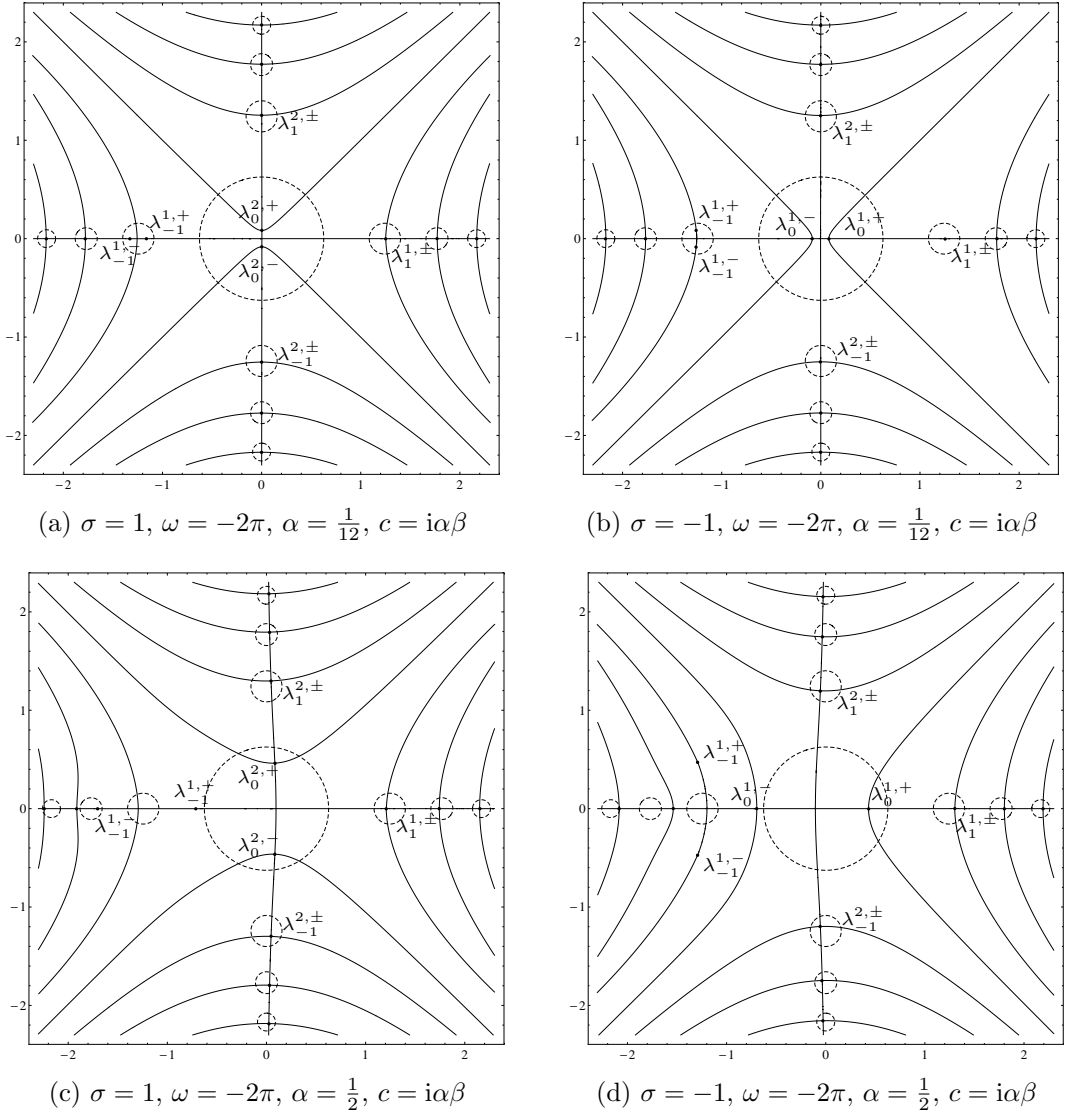

To ensure that $\psi\in\LL$, i.e.~that $\psi$ has period one, we require that $\omega \in 2\pi\Z$.
If $\sigma=1$, the potential $\psi(t)$ in \eqref{singexp} is of real type and hence relevant for the defocusing NLS; if $\sigma=-1$, it is of imaginary type and hence relevant for the focusing NLS. 
A direct computation shows that the associated fundamental solution $M(t,\lambda,\psi)$ is explicitly given by 
\be \label{M_singexp}
\e^{\frac{\I\omega}{2}t\sigma_3}
\begin{pmatrix}
\cos(\Omega t) + \frac{4\lambda^2 + 2\sigma |\alpha|^2 + \omega}{2 \I \Omega} \sin(\Omega t) 
&\frac{2\alpha\lambda+\I c}{\Omega} \sin(\Omega t)\\
\sigma\frac{2\bar\alpha\lambda-\I \bar c}{\Omega} \sin(\Omega t)
&\cos(\Omega t) - \frac{4\lambda^2 + 2\sigma |\alpha|^2 + \omega}{2 \I \Omega} \sin(\Omega t) 
\end{pmatrix},
\ee
where
\be \label{Omega}
\Omega=\Omega(\lambda)=\sqrt{4\lambda^4 + 2\omega \lambda^2 + 4\sigma \Im(\bar \alpha c) \lambda
+\Big( \frac{\omega}{2} +\sigma |\alpha|^2 \Big)^2 - \sigma |c|^2}.
\ee
We fix the branch of the root in \eqref{Omega} by requiring that 
\bew
\Omega(\lambda)=2\lambda^2+\frac{\omega}{2} + \mathcal O(\lambda^{-1}) \quad  \text{as} \quad |\lambda|\to\infty.
\eew
Thus the discriminant $\Delta$, i.e.~the trace of \eqref{M_singexp}, and the characteristic function for the periodic eigenvalues $\chi_{\mathrm P}$ defined in \eqref{char_func_per} are given by
\bew
\Delta(\lambda,\psi)= -2 \cos(\Omega), \quad 
\chi_{\mathrm P}(\lambda,\psi) = 4\sin^2(\Omega).
\eew

\smallskip

Fig.~\ref{fig:ex_singexp1} shows plots of the zero set of 
$\Delta_2(\cdot,\psi)=\Im\Delta(\cdot,\psi)$ and the periodic eigenvalues
in the complex $\lambda$-plane for four different choices of the parameters $\sigma$, $\alpha$, $c$ and $\omega$. 
All four choices correspond to exact plane wave solutions of NLS. Indeed, if $\omega=-2\pi$ and  $\alpha>0$ is chosen such that $-\sigma 2\alpha^2-\omega>0$ is satisfied for $\sigma=\pm 1$, then 
\bew
u(x,t) = \alpha \, \e^{\I\beta x + \I\omega t} \quad \text{with} \quad \beta=\sqrt{-\sigma 2\alpha^2-\omega}
\eew
solves the defocusing (focusing) NLS if $\sigma=1$ ($\sigma=-1$). Moreover, it holds that
\bew
u(0,t)= \alpha \, \e^{\I\omega t}, \quad u_x(0,t)= c\, \e^{\I\omega t}, \quad \text{with} \quad  c=\I\alpha\beta. 
\eew
In the left and right columns of Fig.~\ref{fig:ex_singexp1} we find examples for the defocusing and focusing case, respectively.
In the top row, the norm of the potential is small enough ($\alpha=1/12$) that each periodic eigenvalue $\lambda^{i,\pm}_n$ is contained in the disc $D^i_n$, $i=1,2$, $n\in\Z$. In Fig.~\ref{fig:singexp1A}, all periodic eigenvalues $\lambda^{1,\pm}_n$ are real and there is a spectral gap 
$[\lambda^{1,-}_{-1},\lambda^{1,+}_{-1}]$; the remaining periodic eigenvalues satisfy $\lambda^{1,-}_n=\lambda^{1,+}_n$, $n\in\Z\setminus\{-1,0\}$. The periodic eigenvalues $\lambda^{2,\pm}_n$, $n\in\Z$, lie on a curve that asymptotes to the imaginary axis. 
In Fig.~\ref{fig:singexp1B}, $\lambda^{1,-}_{-1}$ and $\lambda^{1,+}_{-1}$ are not real but lie on the (global) arc $\gamma_{-1}$, which is symmetric with respect to the real axis and crosses the real line at the critical point $\dot\lambda^1_{-1}$.
In Fig.~\ref{fig:singexp1C} and Fig.~\ref{fig:singexp1D}, the spectral gaps are larger than in Fig.~\ref{fig:singexp1A} and Fig.~\ref{fig:singexp1B}, because the parameter $\alpha$ is larger  than in the previous examples ($\alpha=1/2$).

\begin{figure}[h!]
\begin{overpic}[width=.43\textwidth]{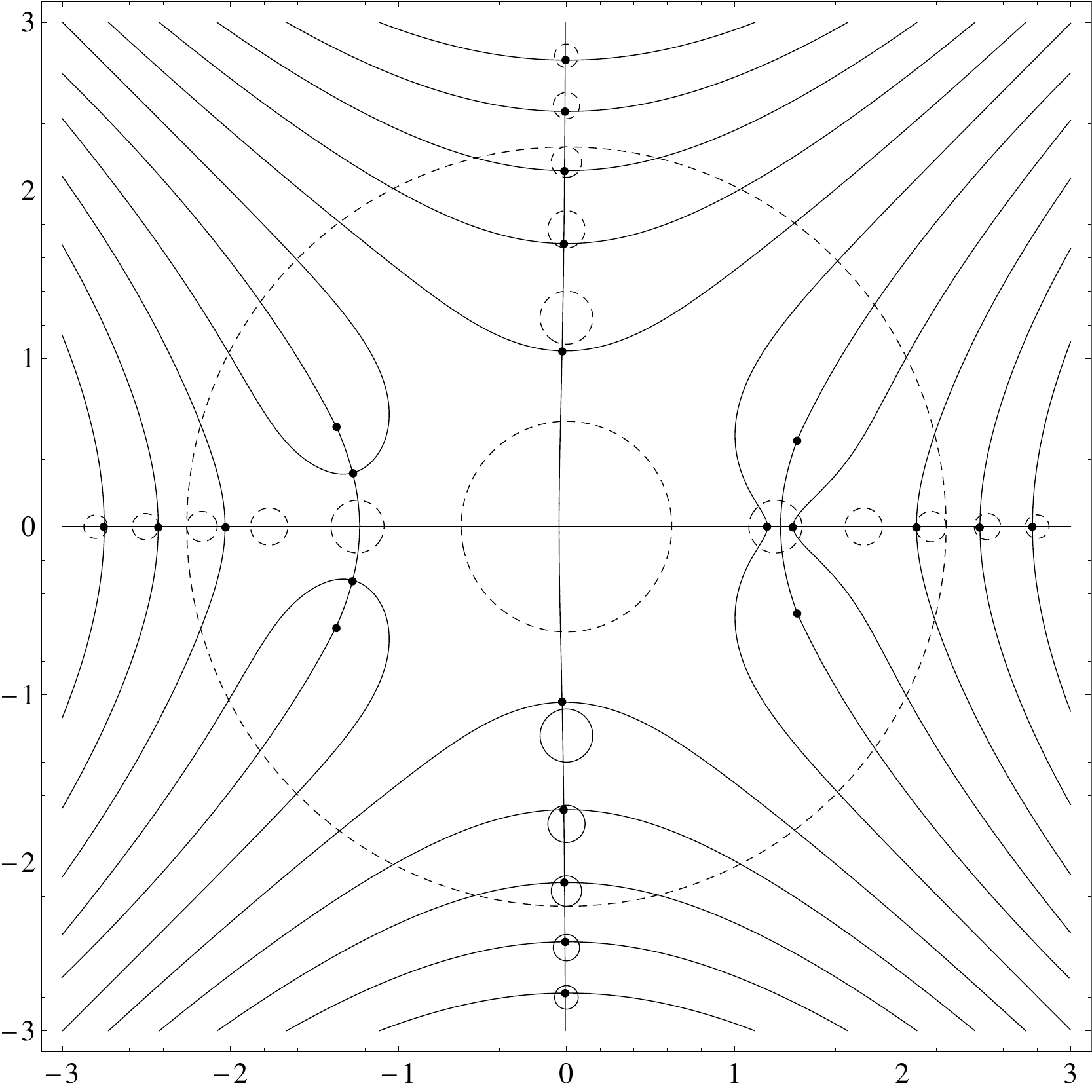}
\end{overpic}
\caption{A plot of the zero level set of $\Delta_2(\cdot,\psi)$  for the real type single exponential potential \eqref{singexp} with $\sigma=1$, $\omega=-2\pi$, $\alpha=\frac{6}{15}+\frac{11}{4}\I$, $c=\frac{1}{10}$. Periodic eigenvalues are indicated with dots, the large dashed circle is the boundary of the disc $B_3$ and the remaining dashed circles are the boundaries of the discs $D^i_n$.}
  \label{fig:singexp2}
\end{figure}

Fig.~\ref{fig:singexp2} shows the zero set of $\Delta_2(\cdot,\psi)$ in the complex $\lambda$-plane for the real type single exponential potential \eqref{singexp} with parameters $\sigma=1$, $\omega=-2\pi$, $\alpha=\frac{6}{15}+\frac{11}{4}\I$ and $c=\frac{1}{10}$. This example clearly demonstrates that Theorems \ref{thm_conn_per_ev} and \ref{cor_conn_per_ev} fail to remain true for potentials with sufficiently large $\LL$-norms. 
We further notice that some arcs $\gamma_n$ do not only ``leave'' the discs $D^i_n$ (and hence also the rectangles $R^{\eps,\delta}_n$ from Theorem \ref{thm_conn_per_ev}), but the zero set differs qualitatively from the previous examples: certain arcs    
``merge'' with other arcs and subsequently ``split'' into new components. 
This example also illustrates that the labeling of periodic eigenvalues is not preserved under continuous deformations of the potential.
We furthermore note that for this particular potential (and consequently all potentials in $\LL$ with smaller $\LL$-norm), the assertion of the Counting Lemma holds already true for $N=3$: there are $4(2\cdot 3+1)=28$ periodic eigenvalues contained in the disc $B_3$ (when counted with multiplicity: 12 double eigenvalues plus 4 simple eigenvalues) and each disc $D^i_n$, $i=1,2$, $|n|>3$, contains precisely one periodic double eigenvalue.


\section{Formulas for gradients} \label{sec_formgrad}
\vspace{0em}

\subsection{Gradient of the fundamental solution}

Let $\mathrm d F$ denote the \emph{Fr\'echet derivative} 
of a functional $F\colon Y \to \C$ on a (complex) Banach space $Y$. If it exists, $\mathrm d F \colon Y \to Y'$ is the unique map from $Y$ into its topological dual space $Y'$ such that
\bew
F(u+h)=F(u) + (\mathrm d F)(u) h + o(h) \quad  \text{as} \quad  \|h\|\to 0 
\eew
for $u\in Y$. The map $\mathrm d F h\colon Y \to \C$ (also denoted by $\D_h F$) is the \emph{directional derivative} of $F$ in direction $h\in Y$.

For any differentiable functional $F \colon \LL \to \C$ and $h\in \LL$, we have that
\bew
\mathrm d F h = \D_h F = \int^1_0 \big(F^1 h^1 + F^2 h^2 + F^3 h^3 + F^4 h^4 \big) \, \mathrm d t
\eew
for some uniquely determined function $\D F=(F^1,F^2, F^3,F^4) \colon \LL \to \LL$. We denote the components of $\D F$ by $\D_j F$, $j=1,2,3,4$, and define the \emph{gradient} $\D F$ of $F$ by
\bew
\D F = (\D_1 F,\D_2 F,\D_3 F,\D_4 F) = (F^1,F^2, F^3,F^4).
\eew
The following proposition gives formulas for the partial derivatives of the fundamental solution 
\bew
M(t, \lambda, \psi) = 
\begin{pmatrix}
m_1 
& m_2 \\ 
m_3
& m_4
\end{pmatrix}.
\eew
For fixed $t\geq 0$ and $\lambda \in \C$, we consider $M$ as a map $\LL\to M_{2\times2}(\C)$. In particular, each matrix entry $m_i$, $i=1,2,3,4$ gives rise to a functional $\LL \to \C$. 
Let us set 
\bew
\gamma \equiv \gamma(M) \coloneqq \det M^{\mathrm d} -  \det M^{\mathrm{od}} = m_1 m_4 + m_2 m_3.
\eew

\begin{proposition}
For any $t\geq 0$ and $0 \leq s \leq t$, the gradient of the fundamental solution $M$, defined on the interval $[0,t]$, is given by 
\begin{align*}
\big(\D_1 M(t)\big)(s) &= M(t)  \, \begin{pmatrix}
-\I  \gamma \psi^2 + 2\lambda m_3 m_4 
& -2\I\psi^2 m_2 m_4 +2\lambda m^2_4 \\ 
2\I\psi^2 m_1 m_3 - 2\lambda m^2_3
& \I  \gamma \psi^2  - 2\lambda m_3 m_4
\end{pmatrix} (s),
	\\
\big(\D_2 M(t)\big)(s) &= M(t)  \, \begin{pmatrix} 
 -\I  \gamma \psi^1 - 2\lambda m_1 m_2 
& -2\I\psi^1 m_2 m_4 -2\lambda m^2_2 \\ 
2\I\psi^1 m_1 m_3 + 2\lambda m^2_1
& \I \gamma \psi^1  + 2\lambda m_1 m_2
\end{pmatrix} (s),
	\\
\big(\D_3 M(t)\big) (s) &= M(t)  \, \begin{pmatrix}
\I m_3 m_4 
& \I m^2_4 \\ 
-\I m^2_3
& -\I m_3 m_4
\end{pmatrix} (s),
	 \\
\big(\D_4 M(t) \big) (s) &= M(t)  \, \begin{pmatrix}
\I m_1 m_2 
& \I m^2_2 \\ 
-\I m^2_1
& -\I m_1 m_2
\end{pmatrix} (s).
\end{align*}
Moreover, at the zero potential $\psi=0$,
\begin{align*}
\big(\D_1 E_\lambda (t)\big)(s) &= 
\begin{pmatrix}
0 & 2\lambda \, \e^{-  2 \I \lambda^2 (t-2s)} \\ 
0 & 0
\end{pmatrix},  
&\big(\D_2 E_\lambda (t)\big)(s) &= 
\begin{pmatrix}
0 & 0 \\ 
2\lambda \, \e^{2 \I \lambda^2 (t-2s)} & 0
\end{pmatrix}, 
	\\ 
\big(\D_3 E_\lambda (t)\big)(s) &= 
\begin{pmatrix}
0 & \I \, \e^{-2 \I \lambda^2 (t-2s)} \\ 
0 & 0
\end{pmatrix},  
&\big(\D_4 E_\lambda (t)\big)(s) &= 
\begin{pmatrix}
0 & 0 \\ 
- \I  \, \e^{2 \I \lambda^2 (t-2s)} & 0
\end{pmatrix}.
\end{align*}
\end{proposition}
\begin{proof}
By Theorem \ref{thm_fund_sol} the fundamental solution $M$ is analytic in $\psi$.
It suffices therefore to verify the above formulas for smooth potentials $\psi$ for which the order of differentiation with respect to $t$ and $\psi$ can be interchanged. The general result then follows by a density argument. 

Applying the directional derivative $\D_h$ to both sides of equation \eqref{fund_solution}, we obtain
\bew
\mathrm D \, \D_h M = (R+V) \, \D_h M + \D_h (R+ V) \, M. 
\eew
Since both $M(0)$ and $R$ are independent of $\psi$, Proposition \ref{Prop_inhom_eq} implies
\bew
\D_h M(t) = M(t) \int^t_0 M^{-1} (s) \, \D_h V (s) \, M(s) \, \mathrm d s.
\eew
The integrand equals 
\begin{align*}
 \begin{pmatrix}m_4 & -m_2 \\ -m_3& m_1\end{pmatrix} \, 
 \begin{pmatrix}-\I ( \psi^2 h^1 + \psi^1 h^2) & 2 \lambda h^1 + \I h^3 \\ 2 \lambda h^2 - \I h^4& \I ( \psi^2 h^1 + \psi^1 h^2)\end{pmatrix} \,
  \begin{pmatrix}m_1 & m_2 \\ m_3& m_4\end{pmatrix},
\end{align*}
which can be rewritten as
\begin{align*}
& \begin{pmatrix}
-\I\psi^2(m_1 m_4 + m_2 m_3) + 2\lambda m_3 m_4 
& -2\I\psi^2 m_2 m_4 +2\lambda m^2_4 \\ 
2\I\psi^2 m_1 m_3 - 2\lambda m^2_3
& \I\psi^2(m_1 m_4 + m_2 m_3) - 2\lambda m_3 m_4
\end{pmatrix} h^1 
	\\
& +
\begin{pmatrix} 
 -\I\psi^1 (m_1 m_4 + m_2 m_3) - 2\lambda m_1 m_2 
& -2\I\psi^1 m_2 m_4 -2\lambda m^2_2 \\ 
2\I\psi^1 m_1 m_3 + 2\lambda m^2_1
& \I\psi^1 (m_1 m_4 + m_2 m_3) + 2\lambda m_1 m_2
\end{pmatrix} h^2 
	\\
& +
\begin{pmatrix}
\I m_3 m_4 
& \I m^2_4 \\ 
-\I m^2_3
& -\I m_3 m_4
\end{pmatrix} h^3 
+
\begin{pmatrix}
\I m_1 m_2 
& \I m^2_2 \\ 
-\I m^2_1
& -\I m_1 m_2
\end{pmatrix} h^4.
\end{align*}
The expression for the gradient $\D M(t)$ follows.
In the case of the zero potential $\psi=0$, we have $m_1 = \e^{-2\lambda^2 \I t}$, $m_4 = \e^{2\lambda^2 \I t}$ and $m_2=m_3=0$, so the gradient $\D E_\lambda(t)$ is easily computed.
\end{proof}

The following notation is useful to express the gradient of $M$ more compactly. Let $M_1$ and $M_2$ denote the first and second columns of $M$,
and denote by $\psi^{1,2}$ the first two components, and by $\psi^{3,4}$ the last two components of the four-vector $\psi$:
\bew
\psi^{1,2} \coloneqq 
\begin{pmatrix}
\psi^1 \\ 
\psi^2
\end{pmatrix}, \quad
\psi^{3,4} \coloneqq 
\begin{pmatrix}
\psi^3 \\ 
\psi^4
\end{pmatrix}.
\eew
Analogously, let
\bew
\D^{1,2} \coloneqq 
\begin{pmatrix}
\D_1 \\ 
\D_2
\end{pmatrix}, \quad
\D^{3,4} \coloneqq 
\begin{pmatrix}
\D_3 \\ 
\D_4
\end{pmatrix}.
\eew
Following~\cite{GrebertKappeler14}, we introduce the \emph{star product} of two 2-vectors $a=(a_1,a_2)$ and $b=(b_1,b_2)$  by
\bew
a \star b \coloneqq
\begin{pmatrix}
a_2 b_2 \\ 
a_1 b_1
\end{pmatrix}.
\eew
Moreover, recall that $\gamma = m_1 m_4 + m_2 m_3$.
With this notation, we obtain
\begin{corollary} \label{Cor_grad_M}
For any $t\geq 0$, the gradient of the fundamental solution $M$ is given by
\begin{align*}
& \D^{1,2}
M(t) = M(t) \begin{pmatrix}
-\I \gamma \sigma_1 \psi^{1,2} + 2 \lambda \sigma_3 (M_1 \star M_2) 
& -2 \I m_2 m_4 \sigma_1 \psi^{1,2} + 2 \lambda \sigma_3 (M_2 \star M_2)  \\ 
2\I m_1 m_3 \sigma_1 \psi^{1,2} - 2 \lambda \sigma_3 (M_1 \star M_1 )
& \I \gamma \sigma_1 \psi^{1,2} - 2 \lambda \sigma_3 ( M_1 \star M_2)
\end{pmatrix}(s), 
	\\
&\I
\D^{3,4}
M(t) = M(t) \begin{pmatrix}
-M_1 \star M_2
& -M_2 \star M_2  \\ 
M_1 \star M_1 
&  M_1 \star M_2
\end{pmatrix}(s).
\end{align*}
\end{corollary}

In the special case when $\psi=0$ and $\lambda$ is a periodic eigenvalue corresponding to the zero potential (i.e.~$\lambda= \lambda^{i,\pm}_n(0)$, $i=1,2$, $n\in\Z$), we find 
\bew
e^+_n = M_1 \star M_1, \quad
e^-_n = M_2 \star M_2, \quad n\in\Z,
\eew
where
\bew
e^+_n \coloneqq 
\begin{pmatrix}
0 \\
\e^{-2 \pi \I n t}
\end{pmatrix}
, \quad
e^-_n \coloneqq 
\begin{pmatrix}
\e^{2 \pi \I n t} \\
0
\end{pmatrix}
, \quad n\in\Z.
\eew
\subsection{Discriminant and anti-discriminant}

\begin{proposition}
The gradient of $\Delta$ is given by
\begin{align*}
\begin{split}
\D^{1,2} \Delta &= 
 \m_2 [2\I m_1 m_3 \sigma_1 \psi^{1,2} - 2 \lambda \sigma_3 (M_1 \star M_1 )]  - \m_3 [ 2 \I m_2 m_4 \sigma_1 \psi^{1,2} - 2 \lambda \sigma_3 (M_2 \star M_2)] \\
& \quad + (\m_4 - \m_1) [\I \gamma \sigma_1 \psi^{1,2} - 2 \lambda \sigma_3 ( M_1 \star M_2)],
\end{split}\\
\begin{split}
\I \D^{3,4} \Delta &= 
 \m_2 M_1 \star M_1
- \m_3 M_2 \star M_2
+ (\m_4-\m_1) M_1 \star M_2.
\end{split}
\end{align*}
At the zero potential, $\D \Delta(\lambda,0) =0$ for all  $\lambda\in\C$.
\end{proposition}
\begin{proof} 
The formula for the gradient follows directly from Corollary \ref{Cor_grad_M}. In the case of the zero potential, $m_2=m_3=0$; hence $M_1 \star M_2=0$ and therefore $\D \Delta(\lambda,0) =0$ for all  $\lambda\in\C$. 
\end{proof}
The following formulas for the derivative of the anti-discriminant are derived in a similar way.
\begin{proposition}
The gradient of the anti-discriminant $\delta$ is given by
\begin{align*}
\partial^{1,2} \delta = &\; \m_4 [2\I m_1 m_3 \sigma_1 \psi^{1,2} - 2 \lambda \sigma_3 (M_1 \star M_1 )] 
+ (\m_2-\m_3) [\I \gamma \sigma_1 \psi^{1,2} - 2 \lambda \sigma_3 ( M_1 \star M_2)] \\
& - \m_1 [2 \I m_2 m_4 \sigma_1 \psi^{1,2} - 2 \lambda \sigma_3 (M_2 \star M_2)],
	\\
\I \partial^{3,4} \delta = &\; \m_4 M_1 \star M_1 + (\m_2-\m_3) M_1 \star M_2 - \m_1 M_2 \star M_2.
\end{align*}
In the special case when $\psi=0$ and $\lambda$ is a periodic eigenvalue corresponding to the zero potential, i.e.~$\lambda= \lambda^{i,\pm}_n(0)$, $i=1,2$, $n\in\Z$, 
\bew
\partial^{1,2}\delta = 2\lambda^{i,\pm}_n(0)(-1)^n (e^+_n + e^-_n), \quad 
\I \partial^{3,4}\delta = (-1)^n (e^+_n - e^-_n).
\eew
\end{proposition}

\section{Hamiltonian structure of the nonlinear Schr\"odinger system} \label{sec_biHam}
\noindent
Consider the NLS system
\begin{align}\label{NLSsystem}
\begin{cases}
  \I q_t + q_{xx} - 2 q^2 r = 0,
  	\\
  -\I r_t + r_{xx} - 2 r^2 q = 0,	
\end{cases}
\end{align}
where $q(x,t)$ and $r(x,t)$ are independent complex-valued functions. If $r = \sigma \bar{q}$, the system \eqref{NLSsystem} reduces to the NLS equation \eqref{nls}. 
We can view \eqref{NLSsystem} as an evolution equation with respect to $t$ by writing
\begin{align}\label{NLSsystemt}
\begin{pmatrix} q \\ r \end{pmatrix}_t = \I \begin{pmatrix} q_{xx} - 2 q^2 r \\ - r_{xx} +2 r^2q  \end{pmatrix}.
\end{align}
On the other hand, introducing $p(x,t)$ and $s(x,t)$ by
\be \label{def_ps}
p = q_x, \quad s = r_x,
\ee
we can also write \eqref{NLSsystem} as an evolution equation with respect to $x$:
\begin{align}\label{NLSsystemx}
\begin{pmatrix} q \\ r \\p \\ s \end{pmatrix}_x = \begin{pmatrix} p \\ s \\ -\I q_t +2 q^2 r \\ \I r_t + 2 r^2 q \end{pmatrix}.
\end{align}
The potentials $\{\psi^j\}_1^4$ of equation \eqref{Vdef} can be viewed as the initial data for \eqref{NLSsystemx} according to the identifications
$$\psi^1(t) = q(0,t), \quad \psi^2(t) = r(0,t), \quad \psi^3(t) = p(0,t), \quad \psi^4(t) = s(0,t).$$

In this section, we first review the bi-Hamiltonian formulation of \eqref{NLSsystem} when viewed as an evolution equation with respect to $t$. We also recall how this formulation gives rise to an infinite number of conservation laws. We then show that \eqref{NLSsystem} admits a Hamiltonian formulation also when viewed as an evolution equation with respect to $x$. Although this one Hamiltonian formulation is enough for the purpose of establishing local Birkhoff coordinates for the $x$-evolution of NLS, we also consider the existence of a second Hamiltonian structure for  \eqref{NLSsystemx}. We find that even though the infinitely many conservation laws of \eqref{NLSsystemt} transfer to the $x$-evolution equation \eqref{NLSsystemx}, the naive way of deriving a second Hamiltonian structure for this system fails. Indeed, the obvious guess for a second Hamiltonian structure yields a Poisson bracket which does not satisfy the Jacobi identity.

In contrast to the rest of the paper, we will not specify the functional analytic framework in terms of Sobolev spaces such as $H^1(\T,\C)$. Instead we will adopt the more algebraic point of view of \cite{O1986}, which is not restricted to the periodic setting; roughly speaking, this means that we will assume that all functions can be differentiated to any order and that partial integrations can be performed freely with vanishing boundary terms. We will use the symbol $\int$ to denote integration over the relevant $x$ or $t$ domain.

\subsection{The bi-Hamiltonian structure of \eqref{NLSsystemt}} \label{subsec_standard-biHam}
In the current framework, we define the gradient $\partial F$ of a functional $F = F[q,r]$ by 
\bew
\partial F = \begin{pmatrix} \partial_1 F \\ \partial_2 F \end{pmatrix}
= \begin{pmatrix} \frac{\partial F}{\partial q} \\ \frac{\partial F}{\partial r} \end{pmatrix}
\eew
whenever there exist functions $\partial_1 F$ and $\partial_2 F$ such that
\bew
\frac{\mathrm d}{\mathrm d\epsilon} F[q + \epsilon \varphi_1, r + \epsilon \varphi_2]\biggr|_{\epsilon =0} = \int \Big[(\partial_1 F) \varphi_1 + (\partial_2 F) \varphi_2 \Big] \, \mathrm d x
\eew
for any smooth functions $\varphi_1$ and $\varphi_2$ of compact support.
The system \eqref{NLSsystemt} admits the bi-Hamiltonian formulation \cite{M1978}
\be \label{bi-Ham_formulation}
\begin{pmatrix} q \\ r \end{pmatrix}_t = \mathcal{D} \, \partial H_1 = \mathcal{E} \, \partial H_2,
\ee
where the Hamiltonian functionals $H_1[q,r]$ and $H_2[q,r]$ are defined by
\begin{align}\label{H1H2def}
&H_1 = \I \int q_x r\, \mathrm d x, \qquad H_2 = \int \Big(-q_{xx}r+q^2 r^2\Big)\, \mathrm d x,
\end{align}
and the operators $\mathcal{D}$ and $\mathcal{E}$ are given by
\begin{align}\label{calDcalEdef}
 \mathcal{D} &=\begin{pmatrix} 
  2q D_x^{-1}q & D_x - 2q D_x^{-1} r 
 \\  D_x - 2r D_x^{-1} q  & 2r D_x^{-1} r \end{pmatrix},
 \qquad \mathcal{E} = \begin{pmatrix}  0 & -\I 
 \\ \I & 0  \end{pmatrix}.
\end{align}
The equalities in \eqref{bi-Ham_formulation} are easy to verify using that
\bew
\partial H_1 = \begin{pmatrix} -\I r_x \\ \I q_x \end{pmatrix}, \qquad
\partial H_2 = \begin{pmatrix} -r_{xx} + 2qr^2 \\ - q_{xx} + 2q^2 r \end{pmatrix}.
\eew
The operators $\mathcal{D}$ and $\mathcal{E}$ are Hamiltonian operators in the sense that the associated Poisson brackets
\begin{align}\label{PoissonbracketsDEdef}
\{F, G\}_\mathcal{D} = \int (\partial F)^\intercal \,  \mathcal{D} \, \partial G \, \mathrm d x, \qquad \{F, G\}_\mathcal{E} = \int (\partial F)^\intercal \, \mathcal{E} \, \partial G \, \mathrm d x,
\end{align}
are skew-symmetric and satisfy the Jacobi identity \cite[Definition 7.1]{O1986}. Furthermore, $\mathcal{D}$ and $\mathcal{E}$ form a Hamiltonian pair in the sense that any linear combination $a\mathcal{D} + b \mathcal{E}$, $a,b \in \R$, is also a Hamiltonian operator \cite[Definition 7.19]{O1986}.
We will review the proofs of these properties below.

The bi-Hamiltonian formulation \eqref{bi-Ham_formulation} together with the fact that $\mathcal{D}$ and $\mathcal{E}$ form a Hamiltonian pair implies that $\mathcal{D} \mathcal{E}^{-1}$ is a recursion operator for \eqref{NLSsystemt} and that a hierarchy of conserved quantities $H_n$ can be obtained (at least formally) by means of the recursive definition (see \cite[Theorem 7.27]{O1986})
\bew
\mathcal{D} \, \partial H_n = \mathcal{E} \, \partial H_{n+1}.
\eew
The first few conserved quantities $H_0, H_1, H_2, H_3$ for \eqref{NLSsystemt} are given by \eqref{H1H2def} and
\begin{align*}
&H_0 = \int qr \, \mathrm d x, \qquad
H_3 = \I \int\Big(-q_{xxx}r + \frac{3}{2}(q^2)_x r^2\Big)\, \mathrm d x.
\end{align*}
In differential form, the associated conservation laws are given by
\begin{align}  \label{differentialconservationlaws} 
\begin{aligned}
& H_0: \quad (qr)_t = \I (q_xr - qr_x)_x,
 	\\ 
& H_1: \quad \I ( q_xr)_t = (q^2r^2 + q_xr_x - q_{xx} r)_x,
 	\\ 
& H_2: \quad \big(-q_{xx} r +q^2 r^2\big)_t = \I (4qq_xr^2 + q_{xx}r_x -q_{xxx} r\big)_x,
	\\ 
& H_3: \quad \I \Big(-q_{xxx}r + \frac{3}{2}(q^2)_x r^2\Big)_t = \\ 
& \qquad \qquad \Big(2q^3r^3 - 5q_x^2 r^2 - 2qq_xrr_x + q^2 r_x^2 - 5qq_{xx}r^2  - 2q^2rr_{xx}  - q_{xxx}r_x + q_{xxxx}r\Big)_x.
\end{aligned}
\end{align}

Even for relatively simple brackets such as those defined in \eqref{PoissonbracketsDEdef}, the direct verification of the Jacobi identity is a very complicated computational task. In the next lemma, we give a proof of the well-known fact that $\mathcal{D}$ and $\mathcal{E}$ are Hamiltonian operator by appealing to the framework of \cite[Chapter 7]{O1986} which shortens the argument significantly. 

\begin{lemma} \label{lem_Ham_op_D_E_t-ev}
  $\mathcal{D}$ and $\mathcal{E}$ are Hamiltonian operators. 
\end{lemma}
\begin{proof}
It is easy to verify that $\mathcal{D}$ and $\mathcal{E}$ are skew-symmetric with respect to the bracket
 \begin{equation}\label{bracket}
   \left \langle \begin{pmatrix} f_1 \\ f_2 \end{pmatrix},  \begin{pmatrix} g_1 \\ g_2 \end{pmatrix} \right\rangle = \int (f_1g_1 + f_2 g_2) \, \mathrm d x.
\end{equation}
Since $\mathcal{E}$ has constant coefficients, $\mathcal{E}$ is a Hamiltonian operator by \cite[Corollary 7.5]{O1986}.
It remains to show that the bracket defined by $\mathcal{D}$ satisfies the Jacobi identity.
According to \cite[Proposition 7.7]{O1986}, it is enough to show that the functional tri-vector $\Psi_{\mathcal{D}}$ defined by
\bew
\Psi_{\mathcal{D}} = \frac{1}{2} \int \big\{\theta \wedge \pr \mathbf{v}_{\mathcal{D}\theta}(\mathcal{D}) \wedge \theta\big\} \, \mathrm d x
= \frac{1}{2} \int \sum_{\alpha, \beta = 1}^2 
\big\{\theta^\alpha \wedge (\pr \mathbf{v}_{\mathcal{D}\theta}(\mathcal{D}))_{\alpha\beta} \wedge \theta^\beta \big\} \, \mathrm d x
\eew
vanishes, where we refer to \cite{O1986} for the definitions of the wedge product $\wedge$, the functional vector $\theta = (\theta^1, \theta^2)$, the vector  field $\mathbf{v}_{\mathcal{D}\theta}$, and its prolongation $\pr \mathbf{v}_{\mathcal{D}\theta}$.
We have (see \cite[p. 442]{O1986})
\begin{align} \nonumber 
& \pr \mathbf{v}_{\mathcal{D}\theta}(q) = (\mathcal{D}\theta)^1,\qquad
\pr \mathbf{v}_{\mathcal{D}\theta}(r) = (\mathcal{D}\theta)^2,
	\\ \label{calDtheta}
& \mathcal{D}\theta = \begin{pmatrix}(\mathcal{D}\theta)^1 \\ (\mathcal{D}\theta)^2 \end{pmatrix} =  \begin{pmatrix}
  2q (D_x^{-1}(q\theta^1)) + (D_x\theta^2) - 2q (D_x^{-1}(r \theta^2)) 
 \\  
 (D_x\theta^1) - 2r (D_x^{-1}(q\theta^1))  + 2r (D_x^{-1}(r\theta^2)) 
 \end{pmatrix},
\end{align}
and 
\begin{align*}
\pr \mathbf{v}_{\mathcal{D}\theta}(\mathcal{D}) 
=2 \begin{pmatrix} 
  (\mathcal{D}\theta)^1 D_x^{-1}q + q D_x^{-1}(\mathcal{D}\theta)^1 
  &  -(\mathcal{D}\theta)^1 D_x^{-1} r -q D_x^{-1} (\mathcal{D}\theta)^2
 \\  - (\mathcal{D}\theta)^2 D_x^{-1} q - r D_x^{-1} (\mathcal{D}\theta)^1  & (\mathcal{D}\theta)^2 D_x^{-1} r + r D_x^{-1} (\mathcal{D}\theta)^2 \end{pmatrix}.
\end{align*}
Hence
\begin{align}\nonumber
\Psi_{\mathcal{D}} 
= &\; \int \bigg\{ 
\theta^1 \wedge (\mathcal{D}\theta)^1 \wedge D_x^{-1}(q \theta^1) 
+ \theta^1 \wedge q D_x^{-1}((\mathcal{D}\theta)^1  \wedge \theta^1)
	\\\nonumber
& -\theta^1 \wedge  (\mathcal{D}\theta)^1 \wedge D_x^{-1}(r \theta^2) 
 - \theta^1 \wedge q D_x^{-1} ((\mathcal{D}\theta)^2 \wedge \theta^2)
 	\\\nonumber
&  - \theta^2 \wedge (\mathcal{D}\theta)^2 \wedge D_x^{-1}(q \theta^1)
 - \theta^2 \wedge r D_x^{-1} ((\mathcal{D}\theta)^1 \wedge \theta^1)
 	\\\label{PsicalDexpression}
& + \theta^2 \wedge (\mathcal{D}\theta)^2 \wedge D_x^{-1}(r \theta^2) 
 + \theta^2 \wedge r D_x^{-1} ((\mathcal{D}\theta)^2 \wedge \theta^2)
 \bigg\}\, \mathrm d x.
\end{align}
An integration by parts shows that the first two terms on the right-hand side of \eqref{PsicalDexpression} are equal:
\begin{align*}
\int \theta^1 \wedge q D_x^{-1}((\mathcal{D}\theta)^1  \wedge \theta^1) \, \mathrm d x
& = - \int \big(D_x^{-1}(q \theta^1)\big) \wedge  (\mathcal{D}\theta)^1  \wedge \theta^1 \, \mathrm d x
	\\
& = \int \theta^1 \wedge  (\mathcal{D}\theta)^1  \wedge D_x^{-1}(q \theta^1)  \, \mathrm d x.
\end{align*}
In the same way, the third and sixth terms are equal, the fourth and fifth are equal, and the last two terms are equal. Thus we find
\begin{align}\label{PsicalD}
\Psi_{\mathcal{D}} 
= &\; \int \bigg\{ 
\theta^1 \wedge (\mathcal{D}\theta)^1 \wedge D_x^{-1}(q \theta^1 - r \theta^2) 
   - \theta^2 \wedge (\mathcal{D}\theta)^2 \wedge D_x^{-1}(q \theta^1 - r \theta^2)
 \bigg\}\, \mathrm d x.
\end{align}
Substituting in the expressions \eqref{calDtheta} for $(\mathcal{D}\theta)^1$ and $(\mathcal{D}\theta)^2$, this becomes
\begin{align*}
\Psi_{\mathcal{D}} 
= &\; \int \bigg\{ 
\theta^1 \wedge \big[2q (D_x^{-1}(q\theta^1)) + (D_x\theta^2) - 2q (D_x^{-1}(r \theta^2)) \big] \wedge D_x^{-1}(q \theta^1 - r \theta^2) 
	\\
&   - \theta^2 \wedge \big[ (D_x\theta^1) - 2r (D_x^{-1}(q\theta^1))  + 2r (D_x^{-1}(r\theta^2)) \big] \wedge D_x^{-1}(q \theta^1 - r \theta^2)
 \bigg\}\, \mathrm d x.
\end{align*}
Using that $(D_x^{-1}(q \theta^j)) \wedge (D_x^{-1}(q \theta^j)) = 0$ and $(D_x^{-1}(r \theta^j)) \wedge (D_x^{-1}(r \theta^j)) = 0$, a simplification gives
\begin{align*}
\Psi_{\mathcal{D}} 
= &\; \int \bigg\{ 
\theta^1 \wedge (D_x\theta^2) \wedge D_x^{-1}(q \theta^1 - r \theta^2) 
  - \theta^2 \wedge (D_x\theta^1) \wedge D_x^{-1}(q \theta^1 - r \theta^2) \bigg\}\, \mathrm d x.
\end{align*}
Integrating by parts in the first term on the right-hand side, we arrive at
\begin{align*}
\Psi_{\mathcal{D}} 
= &\; \int \bigg\{ 
-(D_x\theta^1) \wedge \theta^2 \wedge D_x^{-1}(q \theta^1 - r \theta^2) 
   - \theta^2 \wedge (D_x\theta^1) \wedge D_x^{-1}(q \theta^1 - r \theta^2)
 \bigg\}\, \mathrm d x = 0.
\end{align*}
This shows that $\mathcal{D}$ is Hamiltonian and completes the proof. 
\end{proof}

\begin{lemma}\label{DEpairlemma}
  $\mathcal{D}$ and $\mathcal{E}$ form a Hamiltonian pair. 
\end{lemma}
\begin{proof}
By \cite[Corollary 7.21]{O1986}, it is enough to verify that 
\begin{align}\label{prDprE}
\pr \mathbf{v}_{\mathcal{D}\theta}(\Theta_{\mathcal{E}}) + \pr \mathbf{v}_{\mathcal{E}\theta}(\Theta_{\mathcal{D}}) = 0,
\end{align}
where
\bew
\Theta_{\mathcal{D}} = \frac{1}{2} \int \big\{\theta \wedge \mathcal{D} \theta\big\} \, \mathrm d x,
\qquad
\Theta_{\mathcal{E}} = \frac{1}{2} \int \big\{\theta \wedge \mathcal{E} \theta\big\} \, \mathrm d x,
\eew
are the functional bi-vectors representing the associated Poisson brackets.
Since $\mathcal{E}$ has constant coefficients, we have $\pr \mathbf{v}_{\mathcal{D}\theta}(\Theta_{\mathcal{E}}) = 0$.
Moreover, the same computations that led to the expression \eqref{PsicalD} for $\Psi_{\mathcal{D}} 
= -\pr \mathbf{v}_{\mathcal{D}\theta}(\Theta_{\mathcal{D}})$ (with $(\mathcal{D}\theta)^j$ replaced with $(\mathcal{E}\theta)^j$) imply that
\begin{align*}
\pr \mathbf{v}_{\mathcal{E}\theta}(\mathcal{D})
= &\; - \int \bigg\{ 
\theta^1 \wedge (\mathcal{E}\theta)^1 \wedge D_x^{-1}(q \theta^1 - r \theta^2) 
   - \theta^2 \wedge (\mathcal{E}\theta)^2 \wedge D_x^{-1}(q \theta^1 - r \theta^2)
 \bigg\}\, \mathrm d x.
\end{align*}
Since $(\mathcal{E}\theta)^1 =  -\I \theta^2$ and $(\mathcal{E}\theta)^2 = \I \theta^1$, this gives
\begin{align*}
\pr \mathbf{v}_{\mathcal{E}\theta}(\mathcal{D})
= &\; \I \int \bigg\{ 
\theta^1 \wedge \theta^2 \wedge D_x^{-1}(q \theta^1 - r \theta^2) 
   + \theta^2 \wedge \theta^1 \wedge D_x^{-1}(q \theta^1 - r \theta^2)
 \bigg\}\, \mathrm d x  = 0,
\end{align*}
which completes the proof of the lemma.
\end{proof}

\subsection{The NLS system as an evolution in $x$}

The system \eqref{NLSsystemx} expresses the NLS system \eqref{NLSsystem} as an evolution equation with respect to $x$. We first present a Hamiltonian structure for the system \eqref{NLSsystemx}.

\subsubsection{A Hamiltonian structure for \eqref{NLSsystemx}}
The system \eqref{NLSsystemx} can be written as
\begin{align}\label{NLSsystemxfirsthamiltonianstructure}
\begin{pmatrix}
q \\
r\\
p\\
s
\end{pmatrix}_x = \tilde{\mathcal{D}} \, \partial \tilde{H}_1,
\end{align}
where the Hamiltonian functional $\tilde{H}_1[q,r,p,s]$ is defined by
\begin{align}\label{H1tildedef}
& \tilde{H}_1 = \int (ps + \I q_t r - q^2 r^2) \, \mathrm d t
\end{align}
and the operator $\tilde{\mathcal{D}}$ is defined by
\begin{align*}
\tilde{\mathcal{D}} =
\begin{pmatrix}  0 & 0 & 0 & 1  \\
 0 & 0 &1&0  \\
 0 &-1&0& 0 \\
  -1 &0& 0 &0
 \end{pmatrix}.
\end{align*}
The next lemma shows that \eqref{NLSsystemxfirsthamiltonianstructure} is a Hamiltonian formulation of \eqref{NLSsystemx}.

\begin{lemma}
The operator $\tilde{\mathcal{D}}$ is Hamiltonian. 
\end{lemma}
\begin{proof}
It is clear that the bracket $\{F, G\}_{\tilde{\mathcal{D}}}$ defined by
\begin{align*}
\{F, G\}_{\tilde{\mathcal{D}}} &= \int  (\partial F)^\intercal \, \tilde{\mathcal{D}} \, \partial G \, \mathrm d t
\end{align*}
is skew-symmetric. 
The Jacobi identity is satisfied because $\tilde{\mathcal{D}}$ has constant coefficients (see \cite[Corollary 7.5]{O1986}).
\end{proof}

\subsubsection{Conservation laws}
We conclude from the conservation laws in \eqref{differentialconservationlaws} that if $(q,r,p,s)$ evolves in $x$ according to the NLS system \eqref{NLSsystemx}, then the functionals
\begin{align*}\nonumber
& \tilde{H}_0 \coloneqq \I \int  (q_xr - qr_x) \, \mathrm d t, \qquad \tilde{H}_1 \coloneqq \int (q^2r^2 + q_xr_x - q_{xx} r) \, \mathrm d t,
	\\
& \tilde{H}_2 \coloneqq \I \int  (4qq_xr^2 + q_{xx}r_x -q_{xxx} r) \, \mathrm d t, \qquad \text{etc.}
\end{align*}
are conserved under the flow, i.e., 
\bew
\frac{\mathrm d\tilde{H}_n}{\mathrm d x} = 0.
\eew
Using \eqref{def_ps} and \eqref{NLSsystemx} to eliminate the $x$-derivatives from the above expressions, we find that, on solutions of \eqref{NLSsystemx}, 
\begin{align}\label{tildeHjdef}
& \tilde{H}_0 = \I \int  (pr - qs) \, \mathrm d t, \quad  \tilde{H}_1 = \int (ps + \I q_tr - q^2 r^2) \, \mathrm d t, \quad \tilde{H}_2 = \int (q_t s - p_t r) \, \mathrm d t, \quad \text{etc.}
\end{align}
In this way, we obtain an infinite number of conserved quantities for \eqref{NLSsystemx}.
In differential form, the first few conservation laws are given by
\begin{align*}
& \tilde{H}_0: \quad  \I (pr - qs)_x = (qr)_t,
 	\\
& \tilde{H}_1: \quad  (ps + \I q_tr - q^2 r^2)_x = \I (p r)_t,
	\\ 
& \tilde{H}_2: \quad (q_t s - p_t r)_x = (\I q_t r - q^2 r^2)_t.
\end{align*}
The gradients of the first few functionals $\tilde{H}_j$ are given by
\begin{align*}
\partial \tilde{H}_0 = 
\begin{pmatrix}
- \I s\\
\I p \\
\I r \\
-\I q
\end{pmatrix}, \qquad
\partial \tilde{H}_1 = 
\begin{pmatrix}
-\I r_t - 2 r^2 q\\
\I q_t - 2 q^2 r\\
s\\
p
\end{pmatrix}, \qquad
\partial \tilde{H}_2 = 
\begin{pmatrix}
-s_t \\
-p_t\\
r_t\\
q_t
\end{pmatrix}.
\end{align*}

\subsubsection{A candidate for a second Hamiltonian structure of \eqref{NLSsystemx}}
Inspired by the bi-Hamiltonian formulation \eqref{bi-Ham_formulation} of \eqref{NLSsystemt}, it is natural to seek a second Hamiltonian formulation of \eqref{NLSsystemx} of the form
\begin{align}\label{NLSsystemxsecondhamiltonianstructure}
\begin{pmatrix}
q \\
r\\
p\\
s
\end{pmatrix}_x = \tilde{\mathcal{E}} \, \partial \tilde{H}_2,
\end{align}
where $\tilde{H}_2$ is the conserved functional defined by \eqref{tildeHjdef} and $\tilde{\mathcal{E}}$ is an appropriate Hamiltonian operator. 
It is easy to check that \eqref{NLSsystemxsecondhamiltonianstructure} is satisfied for any choice of the constant $\alpha \in \C$ provided that $\tilde{\mathcal{E}} = \tilde{\mathcal{E}}_\alpha$ is defined by
\begin{align}\label{tildecalEdef}
\tilde{\mathcal{E}}_\alpha = 
\begin{pmatrix}  0 & - D_t^{-1} & 0 & 0  \\
 - D_t^{-1} & 0 &0&0  \\
 0 & 0 & 2\alpha qD_t^{-1} q & -\I + 4 (1-\alpha) r D_t^{-1}q + 2\alpha q D_t^{-1} r \\
  0 &0& \I + 4(1-\alpha)q D_t^{-1}r + 2\alpha r D_t^{-1} q & 2\alpha r D_t^{-1} r
 \end{pmatrix}.
\end{align}
This suggests that we seek a second Hamiltonian operator for \eqref{NLSsystemx} of the form \eqref{tildecalEdef}. The bracket
\begin{align*}
\{F, G\}_{\tilde{\mathcal{E}}_\alpha} &= \int  (\partial F)^\intercal \, \tilde{\mathcal{E}}_\alpha \, \partial G \, \mathrm d t,
\end{align*}
is skew-symmetric for each $\alpha \in \C$. However, the next lemma shows that $\tilde{\mathcal{E}}_\alpha$ is not Hamiltonian for any choice of $\alpha$ because the bracket $\{\cdot, \cdot\}_{\tilde{\mathcal{E}}_\alpha}$ fails to satisfy the Jacobi identity.
\begin{lemma} \label{E_not_Ham}
The operator $\tilde{\mathcal{E}} = \tilde{\mathcal{E}}_\alpha$ defined in \eqref{tildecalEdef} is not Hamiltonian for any $\alpha \in \C$. 
\end{lemma}
\begin{proof}
Fix $\alpha \in \C$.
We will show that $\{F, G\}_{\tilde{\mathcal{E}}}$ does not satisfy the Jacobi identity. By \cite[Proposition 7.7]{O1986}, it is enough to show that the tri-vector 
\bew
\Psi_{\tilde{\mathcal{E}}} = \frac{1}{2} \int \big\{\theta \wedge \pr \mathbf{v}_{\tilde{\mathcal{E}}\theta}(\tilde{\mathcal{E}}) \wedge \theta\big\} \, \mathrm d t
\eew
does not vanish identically. Since
\bew
\pr \mathbf{v}_{\tilde{\mathcal{E}}\theta}(q) = (\tilde{\mathcal{E}}\theta)^1,\qquad
\pr \mathbf{v}_{\tilde{\mathcal{E}}\theta}(r) = (\tilde{\mathcal{E}}\theta)^2,
\eew
we find
\begin{align*}
\pr \mathbf{v}_{\tilde{\mathcal{E}}\theta}(\tilde{\mathcal{E}}) 
=
\begin{pmatrix}  0 & 0 & 0 & 0  \\
 0 & 0 &0&0  \\
 0 & 0 & (\pr \mathbf{v}_{\tilde{\mathcal{E}}\theta}(\tilde{\mathcal{E}}) )_{33} & 
(\pr \mathbf{v}_{\tilde{\mathcal{E}}\theta}(\tilde{\mathcal{E}}) )_{34} \\
  0 &0& (\pr \mathbf{v}_{\tilde{\mathcal{E}}\theta}(\tilde{\mathcal{E}}) )_{43}  & (\pr \mathbf{v}_{\tilde{\mathcal{E}}\theta}(\tilde{\mathcal{E}}) )_{44} 
 \end{pmatrix},
\end{align*}
where
\begin{align*}
& (\pr \mathbf{v}_{\tilde{\mathcal{E}}\theta}(\tilde{\mathcal{E}}) )_{33} = 2\alpha (\tilde{\mathcal{E}}\theta)^1 D_t^{-1} q + 2\alpha qD_t^{-1} (\tilde{\mathcal{E}}\theta)^1,
	\\
& (\pr \mathbf{v}_{\tilde{\mathcal{E}}\theta}(\tilde{\mathcal{E}}) )_{34} = 4 (1-\alpha) (\tilde{\mathcal{E}}\theta)^2 D_t^{-1}q + 4 (1-\alpha) r D_t^{-1}(\tilde{\mathcal{E}}\theta)^1 + 2\alpha (\tilde{\mathcal{E}}\theta)^1 D_t^{-1} r  + 2\alpha q D_t^{-1} (\tilde{\mathcal{E}}\theta)^2,
	\\
& (\pr \mathbf{v}_{\tilde{\mathcal{E}}\theta}(\tilde{\mathcal{E}}) )_{43} = 4(1-\alpha)(\tilde{\mathcal{E}}\theta)^1 D_t^{-1}r +4(1-\alpha)q D_t^{-1}(\tilde{\mathcal{E}}\theta)^2+ 2\alpha (\tilde{\mathcal{E}}\theta)^2 D_t^{-1} q+ 2\alpha r D_t^{-1} (\tilde{\mathcal{E}}\theta)^1,
	\\
& (\pr \mathbf{v}_{\tilde{\mathcal{E}}\theta}(\tilde{\mathcal{E}}) )_{44} = 2\alpha (\tilde{\mathcal{E}}\theta)^2 D_t^{-1} r +  2\alpha r D_t^{-1} (\tilde{\mathcal{E}}\theta)^2.
\end{align*}
Thus
\begin{align*}
\Psi_{\tilde{\mathcal{E}}} 
= &\; \frac{1}{2}\int \bigg\{
\theta^3 \wedge (\pr \mathbf{v}_{\tilde{\mathcal{E}}\theta}(\tilde{\mathcal{E}}) )_{33} \wedge \theta^3 
+ \theta^3 \wedge (\pr \mathbf{v}_{\tilde{\mathcal{E}}\theta}(\tilde{\mathcal{E}}) )_{34} \wedge \theta^4
	\\
& +\theta^4 \wedge (\pr \mathbf{v}_{\tilde{\mathcal{E}}\theta}(\tilde{\mathcal{E}}) )_{43} \wedge \theta^3 
+\theta^4 \wedge (\pr \mathbf{v}_{\tilde{\mathcal{E}}\theta}(\tilde{\mathcal{E}}) )_{44} \wedge \theta^4 \bigg\} \, \mathrm d t
\end{align*}
is given by
\begin{align}\nonumber
\Psi_{\tilde{\mathcal{E}}} 
= &\; \int \bigg\{ 
\alpha \theta^3 \wedge (\tilde{\mathcal{E}}\theta)^1 \wedge D_t^{-1} (q \theta^3)
+ \alpha q \theta^3 \wedge D_t^{-1} ((\tilde{\mathcal{E}}\theta)^1 \wedge \theta^3)
	\\\nonumber
& + 2 (1-\alpha) \theta^3 \wedge (\tilde{\mathcal{E}}\theta)^2 \wedge D_t^{-1}(q \theta^4 )
 +2 (1-\alpha) r \theta^3 \wedge  D_t^{-1}((\tilde{\mathcal{E}}\theta)^1  \wedge \theta^4)
 	\\\nonumber
 & +\alpha \theta^3 \wedge (\tilde{\mathcal{E}}\theta)^1 \wedge D_t^{-1} (r \theta^4)
 +\alpha q \theta^3 \wedge D_t^{-1} ((\tilde{\mathcal{E}}\theta)^2 \wedge \theta^4)
	\\\nonumber
& +2(1-\alpha) \theta^4 \wedge (\tilde{\mathcal{E}}\theta)^1 \wedge D_t^{-1}(r \theta^3)
+2(1-\alpha)q \theta^4 \wedge D_t^{-1}((\tilde{\mathcal{E}}\theta)^2 \wedge \theta^3)
	\\ \nonumber
& + \alpha\theta^4 \wedge  (\tilde{\mathcal{E}}\theta)^2 \wedge D_t^{-1}(q \theta^3)
+ \alpha r \theta^4 \wedge D_t^{-1} ((\tilde{\mathcal{E}}\theta)^1 \wedge \theta^3)
	\\ \label{PsitildecalEexpression}
& + \alpha\theta^4 \wedge (\tilde{\mathcal{E}}\theta)^2 \wedge D_t^{-1} (r \theta^4)
+ \alpha r \theta^4 \wedge D_t^{-1} ((\tilde{\mathcal{E}}\theta)^2 \wedge \theta^4)
 \bigg\}\, \mathrm d t.
\end{align}
An integration by parts shows that the first two terms on the right-hand side of \eqref{PsitildecalEexpression} are equal:
\begin{align*}
\int q \theta^3 \wedge D_t^{-1}((\tilde{\mathcal{E}}\theta)^1  \wedge \theta^3) \, \mathrm d t
& = - \int D_t^{-1}(q \theta^3) \wedge  (\tilde{\mathcal{E}}\theta)^1  \wedge \theta^3 \, \mathrm d t
	\\
& = \int \theta^3 \wedge  (\tilde{\mathcal{E}}\theta)^1  \wedge D_t^{-1}(q \theta^3)  \, \mathrm d t.
\end{align*}
In the same way, the third and eighth terms are equal, the fourth and seventh are equal, the fifth and tenth are equal, the sixth and ninth are equal, and the eleventh and twelfth are equal. Thus we find
\begin{align}\nonumber
\Psi_{\tilde{\mathcal{E}}} 
= &\; 2\int \bigg\{ 
\alpha \theta^3 \wedge (\tilde{\mathcal{E}}\theta)^1 \wedge D_t^{-1} (q \theta^3)
 + 2 (1-\alpha) \theta^3 \wedge (\tilde{\mathcal{E}}\theta)^2 \wedge D_t^{-1}(q \theta^4 )
 	\\ \nonumber
 & + \alpha \theta^3 \wedge (\tilde{\mathcal{E}}\theta)^1 \wedge D_t^{-1} (r \theta^4)
 +2(1-\alpha) \theta^4 \wedge (\tilde{\mathcal{E}}\theta)^1 \wedge D_t^{-1}(r \theta^3)
	\\ \label{PsitildecalE}
& +  \alpha\theta^4 \wedge  (\tilde{\mathcal{E}}\theta)^2 \wedge D_t^{-1}(q \theta^3)
 +  \alpha\theta^4 \wedge (\tilde{\mathcal{E}}\theta)^2 \wedge D_t^{-1} (r \theta^4)
 \bigg\}\, \mathrm d t.
\end{align}
Using that
\bew
\tilde{\mathcal{E}}\theta = 
\begin{pmatrix} - (D_t^{-1}\theta^2)  \\
 - (D_t^{-1}\theta^1) \\
2\alpha qD_t^{-1}(q\theta^3) -\I\theta^4 + 4 (1-\alpha) r D_t^{-1}(q\theta^4) + 2\alpha q D_t^{-1} (r\theta^4) \\
\I\theta^3 + 4(1-\alpha)q D_t^{-1}(r\theta^3) + 2\alpha r D_t^{-1}(q\theta^3) + 2\alpha r D_t^{-1}(r\theta^4)
 \end{pmatrix},
\eew
this becomes
\begin{align}\nonumber
\Psi_{\tilde{\mathcal{E}}} 
= & -2\int \bigg\{ 
\alpha \theta^3 \wedge (D_t^{-1}\theta^2)  \wedge D_t^{-1} (q \theta^3)
 + 2 (1-\alpha) \theta^3 \wedge (D_t^{-1}\theta^1)  \wedge D_t^{-1}(q \theta^4 )
 	\\\nonumber
 & + \alpha \theta^3 \wedge (D_t^{-1}\theta^2)  \wedge D_t^{-1} (r \theta^4)
 +2(1-\alpha) \theta^4 \wedge (D_t^{-1}\theta^2) \wedge D_t^{-1}(r \theta^3)
	\\ \label{PsitildecalEexpression2}
& +  \alpha\theta^4 \wedge  (D_t^{-1}\theta^1)  \wedge D_t^{-1}(q \theta^3)
 +  \alpha\theta^4 \wedge (D_t^{-1}\theta^1)  \wedge D_t^{-1} (r \theta^4)
 \bigg\}\, \mathrm d t.
\end{align}
Consider the two terms which involve all three of the uni-vectors $\theta^1$, $\theta^3$, and $\theta^4$:
\begin{align*}
 \Xi \coloneqq & -2\int \bigg\{ 2 (1-\alpha) \theta^3 \wedge (D_t^{-1}\theta^1)  \wedge D_t^{-1}(q \theta^4 )
 +  \alpha\theta^4 \wedge  (D_t^{-1}\theta^1)  \wedge D_t^{-1}(q \theta^3)\bigg\}\, \mathrm d t
 	\\
= &\; 4 (1-\alpha) \int (D_t^{-1}\theta^1)\wedge  \theta^3    \wedge D_t^{-1}(q \theta^4 ) \, \mathrm d t
  + 2\alpha\int (D_t^{-1}\theta^1)\wedge \theta^4 \wedge  D_t^{-1}(q \theta^3)\, \mathrm d t.	
\end{align*}
Let $P^j = (P_1^j, P_2^j,P_3^j,P_4^j)$, $j = 1,2,3$, where each $P_i^j$ is a differential function (i.e., a smooth function of $t, q,r,p,s$ and $t$-derivatives of $q,r,p,s$ up to some finite, but unspecified, order). 
Then (see \cite[p. 440]{O1986})
\begin{align*}
\langle \Xi; P^1,P^2,P^3\rangle
= &\; 4 (1-\alpha) \int \begin{vmatrix}
(D_t^{-1}P_1^1) &    P_3^1  & D_t^{-1}(q P_4^1) \\
(D_t^{-1}P_1^2) &   P_3^2  & D_t^{-1}(q P_4^2) \\
(D_t^{-1}P_1^3) &   P_3^3  & D_t^{-1}(q P_4^3) \\
\end{vmatrix} \, \mathrm d t
	\\
 &+ 2 \alpha \int \begin{vmatrix}
(D_t^{-1}P_1^1) &    P_4^1  & D_t^{-1}(q P_3^1) \\
(D_t^{-1}P_1^2) &   P_4^2  & D_t^{-1}(q P_3^2) \\
(D_t^{-1}P_1^3) &   P_4^3  & D_t^{-1}(q P_3^3) \\
\end{vmatrix} \, \mathrm d t.
\end{align*}
Choosing for example
\bew
P^1 = (nq^{n-1}q_t, 0, 0, 0), \quad P^2 = (0, 0, 1, 0), \quad P^3 = (0, 0, 0, q_t),
\eew
where $n \geq 1$ is an integer, we see that
\begin{align*}
\langle \Xi; P^1,P^2,P^3\rangle
= &\; 4 (1-\alpha) \int (D_t^{-1}P_1^1) P_3^2  D_t^{-1}(q P_4^3) \, \mathrm d t
- 2 \alpha \int  (D_t^{-1}P_1^1)  P_4^3 D_t^{-1}(q P_3^2) \, \mathrm d t
	\\
= &\; 4 (1-\alpha) \int q^n \frac{q^2}{2} \, \mathrm d t
- 2 \alpha \int  q^n q_t D_t^{-1}(q) \, \mathrm d t
	\\
= &\; 4 (1-\alpha) \int \frac{q^{n+2}}{2} \, \mathrm d t
+ 2 \alpha \int  D_t^{-1}(q^n q_t) q \, \mathrm d t
	\\
= &\; \int \Big(4 (1-\alpha) \frac{q^{n+2}}{2} + 2 \alpha \frac{q^{n+2}}{n+1} \Big) \, \mathrm d t.
\end{align*}
Regardless of the value of $\alpha$, this is nonzero for some integer $n \geq 0$.
Since all the other terms in the expression \eqref{PsitildecalEexpression2} for $\Psi_{\tilde{\mathcal{E}}}$ vanish when applied to this choice of $(P^1, P^2, P^3)$, we conclude that $\Psi_{\tilde{\mathcal{E}}} \neq 0$.
\end{proof}

\begin{remark}
The inverse operator $D_t^{-1}$ in the above computations can be treated as a pseudo-differential operator in the sense of \cite[Definition 5.37]{O1986} by appealing to the identity (see \cite[Eq. (5.55)]{O1986})
$$D_t^{-1} q = \sum_{i=0}^\infty (-1)^i (D_t^iq) D_t^{-i-1}.$$
\end{remark}

\begin{remark}
The failure of $\tilde{\mathcal{E}}$ to be Hamiltonian presumably has to do with the fact that solutions of NLS only live on the submanifold where $q_x = p$ and $r_x = s$, so that one should restrict the Poisson bracket to this submanifold before considering the Jacobi identity. Since the Hamiltonian formulation \eqref{NLSsystemxfirsthamiltonianstructure} is sufficient for our objective of establishing local Birkhoff coordinates for the $x$-evolution \eqref{NLSsystemx} of NLS, and the infinite sequence of conserved quantities can be obtained from the recursion operator for \eqref{NLSsystemt}, we do not pursue this matter further. 
\end{remark}

\begin{remark}
An alternative proof of Lemma \ref{E_not_Ham} proceeds as follows: As in the proof of Lemma \ref{DEpairlemma}, it can be shown that $\tilde{\mathcal{D}}$ and $\tilde{\mathcal{E}}$ satisfy the following analog of \eqref{prDprE} for any value of $\alpha$:
$$\pr \mathbf{v}_{\tilde{\mathcal{D}}\theta}(\Theta_{\tilde{\mathcal{E}}}) + \pr \mathbf{v}_{\tilde{\mathcal{E}}\theta}(\Theta_{\tilde{\mathcal{D}}}) = 0.$$
Thus, if $\tilde{\mathcal{E}}$ were Hamiltonian, then $\tilde{\mathcal{D}}$ and $\tilde{\mathcal{E}}$ would form a Hamiltonian pair and then $\tilde{\mathcal{R}} = \tilde{\mathcal{E}}\tilde{\mathcal{D}}^{-1}$ would be a recursion operator for \eqref{NLSsystemx}. 
However, a direct computation shows that $\tilde{\mathcal{R}}$ does not satisfy the defining relation \cite[Eq. (5.43)]{O1986} of a recursion operator for any value of $\alpha \in \C$.
\end{remark}

\medskip
\noindent
\textbf{Acknowledgment.} The authors are grateful to Thomas Kappeler for valuable discussions. Furthermore, the authors want to thank the anonymous referees for several helpful suggestions which led to a considerable improvement of the manuscript.
Support is acknowledged from the \emph{G\"oran Gustafsson Foundation}, the \emph{European Research Council, Grant Agreement No.~682537} and the \emph{Swedish Research Council, Grant No.~2015-05430}.

\end{document}